\documentclass[sn-mathphys-num]{sn-jnl}
\usepackage[utf8]{inputenc}
\usepackage{amsmath,color}
\usepackage{amsfonts}
\usepackage{graphicx}
\usepackage{amssymb,fullpage}
\numberwithin{equation}{section}

\usepackage{amsthm}

\newtheorem{theorem}{Theorem}
\newtheorem{definition}{Definition}
\newtheorem{lemma}{Lemma}
\begin{document}

\title[Chemical reactions systems with limit cycles]{Planar chemical reaction systems with algebraic and non-algebraic limit cycles}
\author*[1]{\fnm{Gheorghe} \sur{Craciun}}\email{craciun@wisc.edu}

\author*[2]{\fnm{Radek} \sur{Erban}}\email{erban@maths.ox.ac.uk}

\affil[1]{\orgdiv{Department of Mathematics and Department of Biomolecular Chemistry}, \orgname{University of Wisconsin - Madison}, \orgaddress{480 Lincoln Dr}, \city{Madison}, \postcode{WI 53706-1388}, \country{United States}}

\affil[2]{\orgdiv{Mathematical Institute}, \orgname{University of Oxford}, \orgaddress{\street{Radcliffe Observatory Quarter
Woodstock Road}, \city{Oxford}, \postcode{OX2 6GG}, \country{United Kingdom}}}

\abstract{The Hilbert number $H(n)$ is defined as the maximum number of limit cycles of a planar autonomous system of ordinary differential equations (ODEs) with right-hand sides containing polynomials of degree at most $n \in {\mathbb N}$. The dynamics of chemical reaction systems with two chemical species  can be (under mass-action kinetics) described by such planar autonomous ODEs, where $n$ is equal to the maximum order of the chemical reactions in the system. Analogues of the Hilbert number~$H(n)$ for three different classes of chemical reaction systems are investigated: (i) chemical systems with reactions up to the $n$-th order; (ii) systems with up to $n$-molecular chemical reactions; and (iii) weakly reversible chemical reaction networks. In each case (i), (ii) and (iii), the question on the number of limit cycles is considered. Lower bounds on the modified Hilbert numbers are provided for both algebraic and non-algebraic limit cycles. Furthermore, given a general algebraic curve $h(x,y)=0$ of degree $n_h \in {\mathbb N}$ and containing one or more ovals in the positive quadrant, a chemical system is constructed which has the oval(s) as its stable algebraic limit cycle(s). The ODEs describing the dynamics of the constructed chemical system contain polynomials of degree at most $n=2\,n_h+1.$ Considering $n_h \ge 4,$ the algebraic curve $h(x,y)=0$ can contain multiple closed components with the maximum number of ovals given by Harnack's curve theorem as $1+(n_h-1)(n_h-2)/2$, which is equal to 4 for $n_h=4.$ Algebraic curve $h(x,y)=0$ with $n_h=4$ and the maximum number of four ovals is used to construct a chemical system which has four stable algebraic limit cycles.}


\maketitle

\section{Introduction}

The dynamics of chemical reaction networks under mass-action kinetics is inherently connected with the investigation of the dynamics of ordinary differential equations (ODEs) with polynomial right-hand sides~\cite{Feinberg:2019:FCR,Angeli:2009:TCR}. In this paper, we consider chemical reaction networks with two chemical species $X$ and $Y.$ Denoting the time-dependent concentrations of chemical species $X$ and $Y$ by $x(t)$ and $y(t)$, respectively, their time evolution is described by a planar system of ODEs
\begin{eqnarray}
\frac{\mbox{d}x}{\mbox{d}t}
& = &
f(x,y) \, , 
\label{odex_general}
\\
\frac{\mbox{d}y}{\mbox{d}t}
& = &
g(x,y) \, ,
\label{odey_general}
\end{eqnarray}
where $f(x,y)$ and $g(x,y)$ are polynomials. The ODE systems in the form~(\ref{odex_general})--(\ref{odey_general}) have been investigated in detail since the pioneering work of Poincar\'e and Bendixson~\cite{Bendixson:1901:SCD}, who showed that the most complex long term dynamics one can expect to observe in planar systems are multiple limit cycles. Considering that $f(x,y)$ and $g(x,y)$ are polynomials of degree at most $n \in {\mathbb N}$, it is interesting to ask how many limit cycles ODEs~(\ref{odex_general})--(\ref{odey_general}) can have. While this question in its full-generality is a part of the yet-unsolved Hilbert's 16th problem~\cite{Christopher:2007:LCD}, it can be answered when there are additional restrictions imposed on the right-hand sides of the ODEs~\cite{Pota:1983:TBS,Schuman:2003:LCT,Gasull:2023:NLC}. A related question is to find (in some sense minimal) examples of planar polynomial ODE systems~(\ref{odex_general})--(\ref{odey_general}) for low values of $n$, which have a certain number of limit cycles or specific bifurcation structure~\cite{Shi:1980:CEE,Li:2009:CST,Plesa:2016:CRS,Erban:2009:ASC}.

Chemical reaction networks consisting of reactions with order at most $n \in {\mathbb N}$ can be described by ODEs in the form~(\ref{odex_general})--(\ref{odey_general}), where $f(x,y)$ and $g(x,y)$ are polynomials of degree at most $n$, which have some further restrictions on their coefficients. In Section~\ref{sec2}, we define three important classes of chemical reactions networks: (i) set ${\mathbb S}_n$ consisting of reactions of at most $n$-th order; (ii) set ${\mathbb M}_n$ consisting of reactions which are at most $n$-molecular; and (iii) set ${\mathbb W}_{\!n}$ consisting of the networks in ${\mathbb M}_n$ which are weakly reversible~\cite{Craciun:2020:ECC}. Our main question is to understand the existence of limit cycles in sets ${\mathbb S}_n$, ${\mathbb M}_n$ and ${\mathbb W}_{\!n}$, either by finding relatively simple chemical networks with a certain number of limit cycles, or by proving that certain numbers and configurations of the limit cycles cannot exist. Denoting the maximum number of limit cycles in sets ${\mathbb S}_n$, ${\mathbb M}_n$ and ${\mathbb W}_{\!n}$ by $S(n)$, $M(n)$ and $W{\hskip -0.3mm}(n)$, respectively, we study the counterparts of the Hilbert number $H(n)$ in the chemical reaction network theory~\cite{Erban:2023:CSL}. In~Section~\ref{sec3}, we provide estimates on the values of $S(n)$, $M(n)$ and $W{\hskip -0.3mm}(n)$ for small values of $n$ and in the asymptotic limit $n \to \infty.$

An important subset of limit cycles in planar polynomial ODE systems~~(\ref{odex_general})--(\ref{odey_general}) are algebraic limit cycles~\cite{Chavarriga:2004:ALC,Gasull:2023:NLC}. An algebraic limit cycle is not only a closed isolated solution of the ODE system, but it can also be represented as a closed component of an algebraic curve $h(x,y)=0$, where $h$ is a polynomial. The simplest examples of algebraic curves include circles and ellipses. Some chemical systems that have an `exactly evaluable' (algebraic) limit cycle given as an ellipse were analyzed by Escher~\cite{escher1979models,Escher:1980:GSA,Escher:1980:MCR}. In Section~\ref{sec4}, we focus on constructing chemical systems with {\em algebraic} limit cycles. In particular, we further specialize the numbers of limit cycles $S(n)$, $M(n)$ and $W{\hskip -0.3mm}(n)$ in sets ${\mathbb S}_n$, ${\mathbb M}_n$ and ${\mathbb W}_{\!n}$ to quantities giving the maximum number of {algebraic} limit cycles, denoting them by $S^a(n)$, $M^a(n)$ and $W^a(n)$. If the degree of polynomial $h$ is $n_h=2$ or $n_h=3$, then the algebraic curve $h(x,y)=0$ contains at most one closed oval (a connected component diffeomorphic to a circle). If $n_h=4,$ then Harnack's curve theorem implies that the maximum number of connected components is 4. In Section~\ref{sec61}, we study an algebraic curve $h(x,y)=0$ of degree $n_h=4$ which has the maximum number of ovals, $4$, for some parameter values. We construct a chemical system which has 
all four ovals as stable limit cycles. We conclude with the discussion of our results and the literature in Section~\ref{secdiscussion}.

\section{Chemical reaction networks in sets ${\mathbb S}_n$, ${\mathbb M}_n$ and ${\mathbb  W}_{\!n}$}

\label{sec2}

We consider chemical reaction networks with two chemical species $X$ and $Y$ which are subject to $m \in {\mathbb N}$ chemical reactions
\begin{equation}
\alpha_i X \, + \, \beta_i Y 
\;
\mathop{\longrightarrow}^{k_i}
\;
\gamma_i X \, + \, \delta_i Y \, ,
\qquad
\mbox{for}
\quad
i = 1, 2, \dots, m,
\label{crn}    
\end{equation}
where $\alpha_i \in {\mathbb N}_0,$ 
$\beta_i \in {\mathbb N}_0$, 
$\gamma_i \in {\mathbb N}_0$ and 
$\delta_i \in {\mathbb N}_0$ are nonnegative integers, called {\it stoichiometric coefficients}, and $k_i$ is the corresponding reaction {\it rate constant} which has physical units of [time]$^{-1}$[concentration]$^{1-\alpha_i-\beta_i}$. However, in what follows, we assume that all chemical models have been non-dimensionalized, {\it i.e.} the rate constant $k_i$ is assumed to be a positive real number for $i=1,2,\dots,m$. Moreover, to avoid some degenerate cases, we assume that the same four-tuple $(\alpha_i, \beta_i,\gamma_i,\delta_i)$ does not occur more than once in the set of $m$ chemical reactions~(\ref{crn}) and we have $(\alpha_i, \beta_i) \ne (\gamma_i,\delta_i)$, {\it i.e.} in each reaction step at least one of the two chemical species $X$ and $Y$ changes. To get non-trivial planar systems, we also assume that both species take part in at least one reaction in the  set of $m$ chemical reactions~(\ref{crn}). We then define the {\it order} of the chemical reaction network~(\ref{crn}) by
\begin{equation}
n 
\, = \,
\max_{i=1,2,\dots,m}
(\alpha_i+\beta_i) \, ,
\label{deforder}
\end{equation}
{\it i.e.} the chemical reaction network~(\ref{crn}) is of the $n$-th order, if all individual reactions are of at most $n$-th order, where the order of an individual reaction is given as $(\alpha_i+\beta_i)$.
Assuming mass-action kinetics, the time evolution of the chemical reaction network~(\ref{crn}) is given by the reaction rate equations which is a planar polynomial ODE system in the following form~\cite{Feinberg:2019:FCR,Erban:2020:SMR} 
\begin{eqnarray}
\frac{\mbox{{\rm d}}x}{\mbox{{\rm d}}t}
& = &
\sum_{i=1}^m 
k_i \, (\gamma_i-\alpha_i) \, x^{\alpha_i} \, y^{\beta_i} \, , 
\label{general_chemical_system_x}
\\
\frac{\mbox{{\rm d}}y}{\mbox{{\rm d}}t}
& = &
\sum_{i=1}^m k_i \, (\delta_i-\beta_i) \, x^{\alpha_i} \, y^{\beta_i}   \, .
\label{general_chemical_system_y}
\end{eqnarray}
This ODE system is of the form~(\ref{odex_general})--(\ref{odey_general}), where $f(x,y)$ and $g(x,y)$ are polynomials of degree at most~$n$, where $n$ is the order of the chemical reaction network given by~(\ref{deforder}). 

\vskip 2mm

\noindent
{\bf Example.} The Lotka-Volterra system can be written as a chemical reaction network in the form~(\ref{crn}) with $m=3$ chemical reactions and stochiometric coefficients
$$
(\alpha_1,\beta_1,\gamma_1,\delta_1)
=
(1,0,2,0),
\qquad
(\alpha_2,\beta_2,\gamma_2,\delta_2)
=
(1,1,0,2),
\qquad
(\alpha_3,\beta_3,\gamma_3,\delta_3)
=
(0,1,0,0),
$$
{\it i.e.} the chemical reactions are
\begin{equation}
X 
\;
\mathop{\longrightarrow}^{k_1}
\;
2 X \, ,
\qquad\qquad
X \, + \, Y
\;
\mathop{\longrightarrow}^{k_2}
\;
2 Y \, ,
\qquad\qquad
Y
\;
\mathop{\longrightarrow}^{k_3}
\;
\emptyset,
\label{LotkaVolterraCRN}
\end{equation}
where the reaction rate constants $k_1$, $k_2$ and $k_3$ are positive real numbers. The Lotka-Volterra system~(\ref{LotkaVolterraCRN}) is given by a chemical reaction network of the second-order, {\it i.e.} equation~(\ref{deforder}) gives $n=2$. The corresponding ODE system~(\ref{general_chemical_system_x})--(\ref{general_chemical_system_y}) is a planar quadratic ODE system of the form
\begin{eqnarray}
\frac{\mbox{d}x}{\mbox{d}t}
\,& = &\,
k_1 \, x \, - \, k_2 \, x \, y \, , 
\label{LotkaVolterraODEx}
\\
\frac{\mbox{d}y}{\mbox{d}t}
\,& = &\,
k_2 \, x \, y \, - \, k_3 \, y \, .
\label{LotkaVolterraODEy}
\end{eqnarray}

\vskip 2mm

\begin{definition}
Let $n \in {\mathbb N}.$ We denote by $\,{\mathbb S}_n$ the set of chemical networks~$(\ref{crn})$ which are of at most $n$-th order, where the order is defined by~$(\ref{deforder})$.
\label{defonset}
\end{definition}

\vskip 2mm

\noindent
In what follows we will make a convenient abuse of terminology, and use ${\mathbb S}_n$ to denote not only the set of the chemical reaction networks described above, but also the set of the corresponding ODEs~(\ref{general_chemical_system_x})--(\ref{general_chemical_system_y}), which are planar autonomous ODE systems in the form~(\ref{odex_general})--(\ref{odey_general}), where $f(x,y)$ and $g(x,y)$ are polynomials of degree at most $n$. In general, such polynomials can be expressed in terms of coefficients as
\begin{equation}
f(x,y) 
\, = \, 
\sum_{\mbox{$\{$}i,j \ge 0 \; \mbox{$|$} \;i+j \le n\mbox{$\}$}}  a_{i,j} \, x^i \, y^j
\qquad
\mbox{and}
\qquad
g(x,y) 
\, = \, 
\sum_{\mbox{$\{$}i,j \ge 0 \; \mbox{$|$} \;i+j \le n\mbox{$\}$}}  b_{i,j} \, x^i \, y^j \,,
\label{fg_with_coefficients}
\end{equation}
where $a_{ij}$, $b_{ij}$ are real numbers. The set ${\mathbb S}_n$ can then be characterized in terms of restrictions on these coefficients, which we state as our next lemma.

\vskip 2mm

\begin{lemma}
Consider an {\rm ODE} system in the form $(\ref{odex_general})$--$(\ref{odey_general})$, with $f$ and $g$ given by $(\ref{fg_with_coefficients})$. Then a necessary and sufficient condition for belonging to the set $\,{\mathbb S}_{n}$ is that the coefficients of $f$ and $g$ in~$(\ref{fg_with_coefficients})$ satisfy the following condition
\begin{equation}
a_{0,i} \ge 0 \qquad
\mbox{and}
\qquad
b_{i,0} \ge 0 \qquad
\mbox{for} \quad i=0,1,2,\dots,n.
\label{condO}    
\end{equation}
\label{lemma1}
\end{lemma}

\begin{proof}
This is a classical result~\cite{Hars:1981:IPR}. We include a short proof here, since the construction in this proof  will be used again later. 
Consider some term of the form $a_{i,j} \, x^i \, y^j$ that is one of the monomials within $f(x,y)$ in~(\ref{fg_with_coefficients}). In particular, we have $i+j \le n$. We will exhibit a reaction that gives rise to this term $a_{i,j} \, x^i \, y^j$, and no other terms; in other words, this reaction generates the system
$$
\frac{\mbox{d}x}{\mbox{d}t}
\, = \,
a_{i,j} \, x^i \, y^j  
\qquad\qquad
\mbox{and}
\qquad\qquad
\frac{\mbox{d}y}{\mbox{d}t}
\, = \,
0. 
$$
Indeed, if $a_{i,j} < 0$ it follows that $i \ge 1$ (because $a_{0,i} \ge 0$), and then this term can be obtained by using the reaction $iX + jY \to (i-1)X + jY$, if we choose its reaction rate constant to be $k=-a_{i,j}$. 
Similarly, if $a_{i,j} > 0$, this term can be obtained by using the reaction $iX + jY \to (i+1)X + jY$, by choosing its reaction rate constant to be $k=a_{i,j}$. 

We can proceed similarly for monomials of $g(x,y)$, by using reactions of the from $iX + jY \to iX + (j-1)Y$ and $iX + jY \to iX + (j+1)Y$. We conclude that, by using the reaction network which consists of the reactions described above ({\it i.e.}, one reaction for each monomial of $f(x,y)$, and one reaction for each monomial of $g(x,y)$), we can obtain the system (\ref{odex_general})-(\ref{odey_general}). 

Conversely, we can see in $(\ref{general_chemical_system_x})$--$(\ref{general_chemical_system_y})$ that if $a_i = 0$ then $k_i(c_i - a_i) \ge 0$ and, similarly,   if $b_i = 0$ then $k_i(d_i - b_i) \ge 0$. This implies that any polynomial dynamical system that is generated by a chemical reaction network satisfies the inequalities (\ref{condO}).
\end{proof}

\vskip 2mm

\noindent
The Lotka-Volterra system~(\ref{LotkaVolterraCRN}) is an example of a chemical reaction network in set ${\mathbb S}_2$, because every reaction is of at most second-order. Moreover, we observe that not only each reaction in~(\ref{LotkaVolterraCRN}) has at most two reactants, but is also has at most two products. We will denote the set of such networks as ${\mathbb M}_2$ and call them {\em bimolecular} reaction networks. In general, we define the set of $n$-molecular reaction networks~${\mathbb M}_n$ as follows.

\vskip 2mm

\begin{definition}
Let $n \in {\mathbb N}.$ We denote by $\,{\mathbb M}_n$ the subset of $\,{\mathbb S}_n$ which corresponds to chemical reaction networks~$(\ref{crn})$, where each chemical reaction has at most $n$ reactants and $n$ products, {\it i.e.}
\begin{equation}
\max_{i=1,2,\dots,m}
\max \big\{
(\alpha_i+\beta_i),
(\gamma_i+\delta_i)
\big\} \, \le \, n \, .
\label{nmol}
\end{equation}
\label{defmnset}
\end{definition}

\noindent
The set of $n$-molecular reaction networks ${\mathbb M}_n$ can again be characterized in terms of the coefficients as stated in the next lemma.

\vskip 2mm

\begin{lemma}
Consider an {\rm ODE} system in the form $(\ref{odex_general})$--$(\ref{odey_general})$, with $f$ and $g$ given by $(\ref{fg_with_coefficients})$.  Then a necessary and sufficient condition for belonging to $\,{\mathbb M}_{n}$ is to satisfy inequalities~$(\ref{condO})$ together with
\begin{equation}
a_{i,n-i} + b_{i,n-i} \le 0 \qquad \quad
\mbox{for} \quad i=0,1,2,\dots,n.
\label{condM}    
\end{equation}
\end{lemma}

\begin{proof}
Consider an {\rm ODE} system~$(\ref{odex_general})$--$(\ref{odey_general})$, with $f$ and $g$ expressed in coefficients $(\ref{fg_with_coefficients})$ that satisfy~$(\ref{condO})$ and $(\ref{condM})$. 
In particular, according to Lemma~\ref{lemma1}, this system belongs to ${\mathbb S}_n$, which implies that it can be realized by a reaction network where all the reactions are of the form $iX+jY \to pX+qY$ with $i+j \le n$. 
We need to show that we can also choose these reactions such that  $p+q \le n$. 

Note that we know from the proof of Lemma~\ref{lemma1} that we can obtain the same monomial terms by using reactions such that $(p,q) = (i \pm 1, j)$ or $(p,q) = (i, j \pm 1)$, and therefore $p+q \le i+j+1$.
This gives us the desired conclusion if $i+j \le n-1$. Consider now the case $i+j = n$, {\it i.e.} $j=n-i.$
If $i>0$, $n-i>0$ and $a_{i,n-i} \ge b_{i,n-i}$, then we can realize the system
\begin{equation}
\frac{\mbox{d}x}{\mbox{d}t}
\, = \,
a_{i,n-i} x^i y^{n-i}  
\qquad \mbox{and} \qquad
\frac{\mbox{d}y}{\mbox{d}t}
\, = \,
b_{i,n-i} x^i y^{n-i}
\label{boundsystem}
\end{equation}
by using the following two chemical reactions 
$$
i \, X \, + \, (n-i) \, Y 
\;
\xrightarrow{\;\frac{a_{i,n-i} - b_{i,n-i}}{2}\;}
\;
(i+1) \, X \, + \, (n-i-1) \,Y
$$
and
$$
i \, X \, + \, (n-i) \, Y 
\;
\xrightarrow{\;\frac{-a_{i,n-i} - b_{i,n-i}}{2}\;}
\;
(i-1) \, X \, + \, (n-i-1) \, Y.
$$
Similarly, if $i>0$, $n-i>0$ and $a_{i,n-i} \le b_{i,n-i}$, then we can realize ODE system~(\ref{boundsystem}) by using  two reactions $iX+(n-i)Y \to (i-1)X+(n-i+1)Y$ and $iX+(n-i)Y \to (i-1)X+(n-i-1)Y$.

Finally, if $i=0$, then we know that $a_{0,n} \ge 0$, and we can realize ODE system~(\ref{boundsystem}) by using  two reactions $nY \to X+(n-1)Y$ and $nY \to (n-1)Y$ with reaction rate constants $a_{0,n}$ and $-a_{0,n} - b_{0,n}$, respectively. The case where $i=n$ is analogous to the case $i=0$.
\end{proof}

\noindent
The Lotka-Volterra system~(\ref{LotkaVolterraCRN}) is an example of a chemical reaction network belonging to both ${\mathbb S}_2$ and~${\mathbb M}_2$. It is 
described by the conservative dynamical system~(\ref{LotkaVolterraODEx})--(\ref{LotkaVolterraODEy}), with the conserved quantity $k_2 (x+y) - k_3 \log (x) - k_1 \log(y)$ on orbits. In particular, the ODE system~(\ref{LotkaVolterraODEx})--(\ref{LotkaVolterraODEy}) admits periodic solutions, but it has no limit cycles. In Section~\ref{sec3}, we will start our investigation of the existence and number of limit cycles of chemical reaction networks in sets ${\mathbb S}_n$ and~${\mathbb M}_n$. We will observe that ODE systems in ${\mathbb M}_2$ do not have any limit cycles, but there are reaction systems in ${\mathbb S}_2$ with limit cycles. There are also other important properties and classes of chemical reaction networks, including weakly reversible chemical systems~\cite{Craciun:2020:ECC,Boros:2020:WRM}. To define weak reversibility, we embed the chemical reaction network~(\ref{crn}) as a planar E-graph ({\it i.e.} Euclidean embedded graph, see~\cite{craciun2019polynomial,yu2018mathematical}) where the nodes have coordinates
$(\alpha_i,\beta_i)$ and $(\gamma_i,\delta_i)$ and each reaction corresponds to the edge 
$(\alpha_i,\beta_i) \longrightarrow (\gamma_i,\delta_i)$. For example, the planar E-graph corresponding to the Lotka-Volterra chemical system~(\ref{LotkaVolterraCRN}) has three edges 
$$
(1,0) \to (2,0),
\qquad
(1,1) \to (0,2),
\qquad
(0,1) \to (0,0)
$$
and it is schematically shown in Figure~\ref{figure1}(a). We say that
a chemical reaction network is {\it weakly reversible} if every edge of the associated planar E-graph is a part of an oriented cycle. Clearly, Lotka-Volterra system is not weakly reversible.

\vskip 2mm

\begin{figure}
\leftline{(a) \hskip 4.86cm (b) \hskip 5.0cm (c)}
\centerline{\hskip 5mm \includegraphics[height=3.5cm]{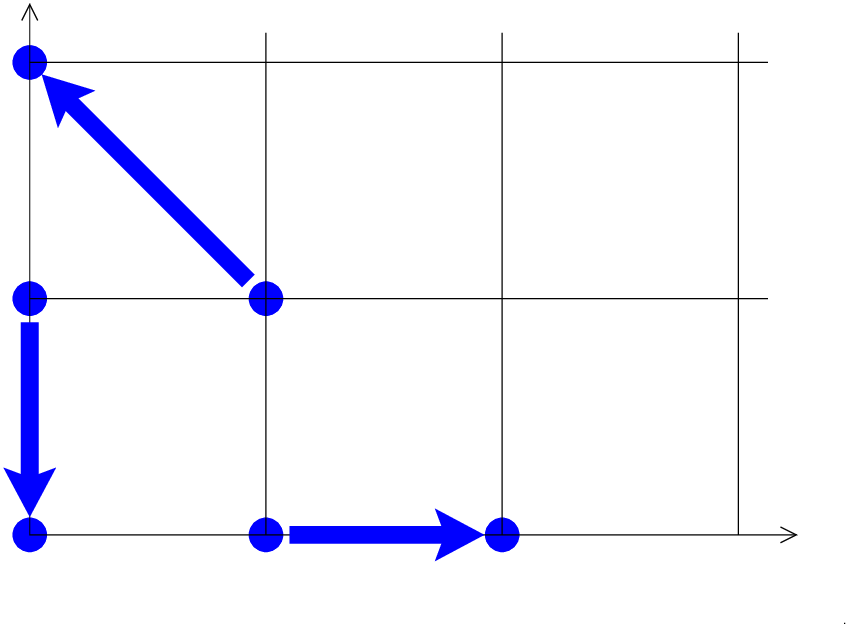}\hskip 7mm \includegraphics[height=3.5cm]{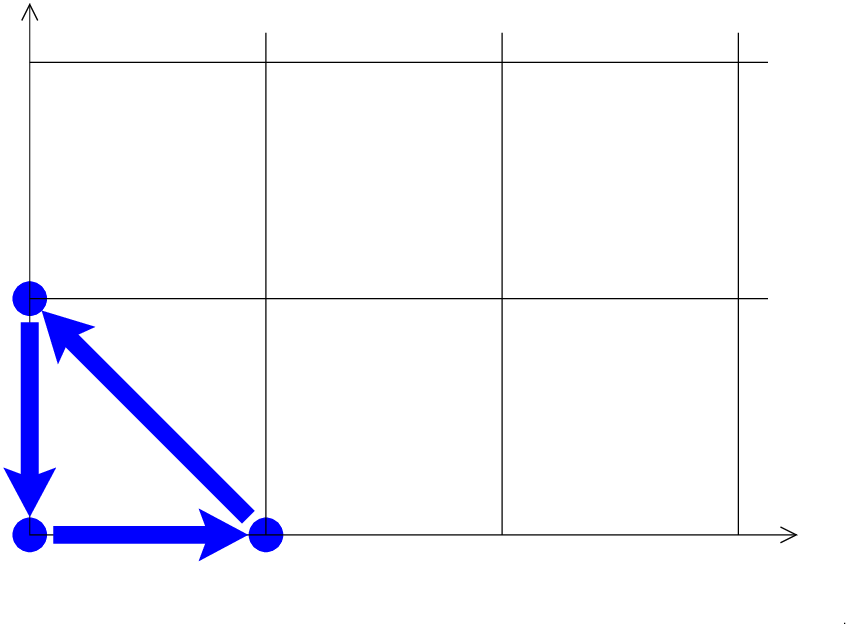} \hskip 7mm \includegraphics[height=3.5cm]{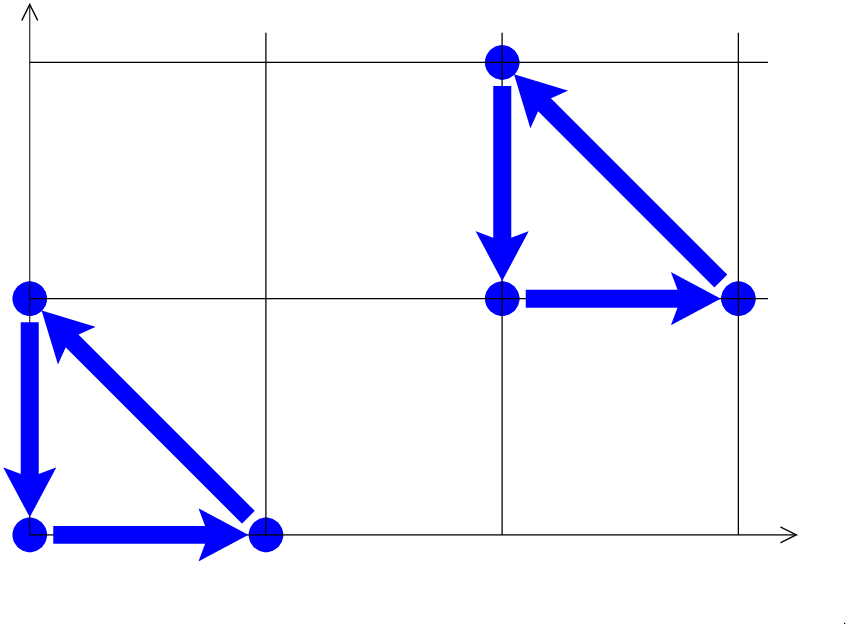}}
\caption{{\it Schematics of planar {\rm E-graphs} associated with} (a)~{\it Lotka-Volterra chemical system~$(\ref{LotkaVolterraCRN})$;} (b)~{\it chemical system~$(\ref{illustrativenetwork1})$;} (c)~{\it chemical system with six reactions given by~$(\ref{illustrativenetwork1})$ and~$(\ref{illustrativenetwork2})$.}}
\label{figure1}
\end{figure}

\noindent
{\bf Example.} Consider the first-order chemical reaction network
\begin{equation}
\emptyset 
\;
\mathop{\longrightarrow}^{1}
\;
X \, ,
\qquad\qquad
X 
\;
\mathop{\longrightarrow}^{1}
\;
Y \, ,
\qquad\qquad
Y
\;
\mathop{\longrightarrow}^{1}
\;
\emptyset \,.
\label{illustrativenetwork1}
\end{equation}
Its associated planar E-graph is visualized in Figure~\ref{figure1}(b), schematically showing three edges $(0,0) \to (1,0),$ $(1,0) \to (0,1)$ and $(0,1) \to (0,0)$ corresponding to the three reactions of chemical system~(\ref{illustrativenetwork1}). Since every edge of the associated planar E-graph is a part of an oriented cycle, chemical reaction network~(\ref{illustrativenetwork1}) is weakly reversible. The ODE system~(\ref{general_chemical_system_x})--(\ref{general_chemical_system_y}) corresponding to the chemical system~(\ref{illustrativenetwork1}) is a planar linear ODE system
\begin{eqnarray}
\frac{\mbox{d}x}{\mbox{d}t}
\,& = &\,
1 \, - \, x \, ,
\label{odelinear1}
\\
\frac{\mbox{d}y}{\mbox{d}t}
\,& = &\,
x \, - \, y \, .
\label{odelinear2}
\end{eqnarray}
Multiplying the right-hand side of the ODE system by polynomial $(1+x^2 y)$, we obtain the ODE system
\begin{eqnarray*}
\frac{\mbox{d}x}{\mbox{d}t}
\,& = &\,
(1 \, + \, x^2 y) (1 \, - \, x) \, ,
\\
\frac{\mbox{d}y}{\mbox{d}t}
\,& = &\,
(1 \, + \, x^2 y) (x \, - \, y) \, ,
\end{eqnarray*}
which are the reaction rate equations for the chemical system consisting of reactions~(\ref{illustrativenetwork1}) and additional three reactions:
\begin{equation}
2 \, X \, + \, Y 
\;
\mathop{\longrightarrow}^{1}
\;
3 \, X \, + \, Y \, ,
\qquad
3 \, X \, + \, Y \, 
\;
\mathop{\longrightarrow}^{1}
\;
2 \, X \, + \, 2 \, Y \, ,
\qquad
2 \, X \, + \, 2 \, Y 
\;
\mathop{\longrightarrow}^{1}
\;
2 \, X \, + \, Y  \,.
\label{illustrativenetwork2}
\end{equation}
The associated planar E-graph is visualized in Figure~\ref{figure1}(c). This example illustrates that the multiplication of the right-hand side of the reaction rate ODEs by a polynomial with positive coefficients results with multiple copies of the original associated E-graph. This will be further used in Section~\ref{sec4}, where we study weakly reversible systems with algebraic limit cycles.
Moreover, the chemical system consisting of six reactions~(\ref{illustrativenetwork1}) and (\ref{illustrativenetwork2}) in Figure~\ref{figure1}(c) provides another example of a weakly reversible system.

\vskip 2mm

\begin{definition}
Let $n \in {\mathbb N}.$ We denote by ${\mathbb W}_{\!n}$ the subset of $\,{\mathbb M}_n$ which corresponds to chemical reaction networks~$(\ref{crn})$ that are weakly reversible.
\label{defwnset}
\end{definition}

\vskip 2mm

\noindent
Considering an arbitrary chemical reaction network, it belongs to ${\mathbb S}_n$ where $n$ is the order given by~(\ref{deforder}). Moreover, the corresponding ODEs satisfy the property that they can be obtained as reaction rate equations of a chemical system in ${\mathbb M}_{n+1}$. We formulate this result as our next lemma.

\vskip 2mm

\begin{lemma}
We have \hskip 2mm
{\rm (a)}
$
{\mathbb W}_{\!n}
\subset
{\mathbb M}_n \subset {\mathbb S}_n
\qquad\quad\mbox{for all} \quad n \in {\mathbb N}
$\,, \hfill\break
\rule{0pt}{1pt} \hskip 34.5mm
{\rm (b)}
$
{\mathbb M}_2 \subset {\mathbb S}_2 \subset {\mathbb M_3} \subset {\mathbb S}_3 \subset {\mathbb M_4} \subset {\mathbb S}_4 \subset  ...
$\,.
\label{setproperties}
\end{lemma}

\begin{proof}
{\rm (a)} This is a direct consequence of Definitions~\ref{defonset}, \ref{defmnset} and~\ref{defwnset}~of ${\mathbb W}_{\!n},$ ${\mathbb M}_n$ and ${\mathbb S_n}.$

\noindent
{\rm (b)} We show that ${\mathbb S}_n \subset {\mathbb M}_{n+1}$, as follows. Recall that in the proof of Lemma~\ref{lemma1} above we have been able to obtain the polynomial right-hand side of a chemical system in ${\mathbb S}_n$ by using only reactions of the form  $iX + jY \to (i-1)X + jY$, $iX + jY \to (i+1)X + jY$, $iX + jY \to iX + (j-1)Y$, or $iX + jY \to iX + (j+1)Y$, with $i+j \le n$. Note now that for these types of reactions the largest stoichiometric coefficients are either $i+1$ and $j$, or $i$ and $j+1$, and we have $i+1+j = i+j+1 \le \ n+1$. This implies that any system that belongs to ${\mathbb S}_n$ also belongs to ${\mathbb M}_{n+1}$. 
\end{proof}

\section{Hilbert number and its analogues for sets ${\mathbb S}_n$, ${\mathbb M}_n$ and ${\mathbb  W}_{\!n}$}

\label{sec3}

Let $H(n)$ be the maximum number of limit cycles for planar {\rm ODE} systems in the form (\ref{odex_general})--(\ref{odey_general}), where $f$ and $g$ are polynomials of degree at most $n$ given by $(\ref{fg_with_coefficients})$.
Then $H(n)$ is often called the Hilbert number, because it can be used to formulate Hilbert's 16th problem~\cite{Christopher:2007:LCD}.
By constructing polynomial ODE systems with a specific number of limit cycles, lower bounds on the Hilbert number $H(n)$ have been obtained:
for example, a quadratic system with 4~limit cycles~\cite{Shi:1980:CEE}, a cubic system with 13 limit cycles~\cite{Li:2009:CST} and a quartic system with 28~limit cycles~\cite{Prohens:2018:NLB} have been constructed in the literature, giving $H(2) \ge 4$, $H(3) \ge 13$ and $H(4) \ge 28$. On the other hand, one can show that quadratic systems which can be written as chemical systems corresponding to bimolecular systems ${\mathbb M}_2$ cannot  have limit cycles~\cite{Pota:1983:TBS,Schuman:2003:LCT}. To formulate our results and put them into context with the literature, we first define the corresponding counterparts of the Hilbert number $H(n)$ for subsets of the polynomial ODE systems corresponding to sets ${\mathbb S}_n$, ${\mathbb M}_n$ and ${\mathbb  W}_{\!n}$.

\vskip 2mm

\begin{definition}
We denote by $S(n)$ the maximum number of limit cycles of {\rm ODEs} in the set $\,{\mathbb S}_n$ of the $n$-th order chemical reaction networks, by $M(n)$ the maximum number of limit cycles of {\rm ODEs} in the set $\,{\mathbb M}_n$ of $n$-molecular chemical reaction networks, and by $W{\hskip -0.3mm}(n)$ the maximum number of limit cycles of {\rm ODEs} in the set ${\mathbb W}_{\!n}$ of weakly reversible chemical reaction networks.
\label{defOMW}
\end{definition}

\begin{table}
\begin{tabular}{|c|c|c|c|c|}
\hline
\raise -1.4mm \hbox{\rule{0pt}{4.6mm}} $\qquad\quad$ {\color{red} $n$} $\qquad\quad\;$ & $\qquad\quad\;$ {\color{red} $W{\hskip -0.2mm}(n)$} $\qquad\quad\;$ & $\qquad\quad\;$ {\color{red} $M(n)$} $\qquad\quad\;$ & $\qquad\quad\;${\color{red} $S(n)$} $\qquad\quad\;$ & $\qquad\quad\;$ {\color{red} $H(n)$} $\qquad\quad\;$
\\
\hline
\hline
\raise -1.4mm \hbox{\rule{0pt}{4.6mm}}{\color{red} 2} & = 0 & = 0 & $\ge$ 3 & $\ge$ 4  \\
\hline
\raise -1.4mm \hbox{\rule{0pt}{4.6mm}}{\color{red} 3} & $\ge$ 3 & $\ge$ 3 & $\ge$ 6 & $\ge$ 13 \\
\hline
\raise -1.4mm \hbox{\rule{0pt}{4.6mm}}{\color{red} 4} & $\ge 6$ & $\ge 6$ & $\ge 13$ & $\ge$ 28 \\
\hline
\hline
\raise -1.4mm \hbox{\rule{0pt}{4.6mm}}{\color{red} as $n \to \infty$} & $\ge {\mathcal O}(n^2 log(n))$ & $\ge {\mathcal O}(n^2 log(n))$ & $\ge {\mathcal O}(n^2 log(n))$ & $\ge {\mathcal O}(n^2 log(n))$  \\
\hline
\end{tabular}
\vskip 2mm
\caption{{\it Some values and estimates from below on numbers $W{\hskip -0.3mm}(n)$, $M(n)$, $S(n)$ and $H(n)$, see Lemmas~$\ref{lemmaM2W2}$, $\ref{lemmaS2S3S4M3M4W3W4}$, $\ref{lemmaHn}$ and~$\ref{lemmaasym}$.}}
\label{table1}
\end{table}

\vskip 2mm

\noindent
Some inequalities between numbers $W{\hskip -0.3mm}(n),$ $M(n)$, $S(n)$ and $H(n)$ are stated as our next lemma.

\vskip 2mm

\begin{lemma}
We have \hskip 2mm
{\rm (a)}
$W{\hskip -0.3mm}(n) \le M(n) \le S(n) \le H(n)$
for all $n \in {\mathbb N}\,$, 
\hfill\break
\rule{0pt}{1pt} \hskip 34.5mm 
{\rm (b)}
$ H(n) \le S(n+1)$ for all $n \in {\mathbb N}\,$,
\hfill\break
\rule{0pt}{1pt} \hskip 34.5mm 
{\rm (c)} \hskip 0.6mm
$S(n) \,\le M(n+1)$ for all $n \in {\mathbb N}\,$,
\hfill\break
\rule{0pt}{1pt} \hskip 34.5mm 
{\rm (d)} 
$H(n) \le W(n+3)$ for all $n \in {\mathbb N}\,$.
\label{WMNineq}
\end{lemma}

\begin{proof} 
{\rm (a)}  
The first two inequalities follow directly from Lemma~\ref{setproperties}(a) and Definition~\ref{defOMW} of numbers $W{\hskip -0.3mm}(n),$ $M(n)$ and $S(n)$. The last inequality is a direct consequence of the definition of the Hilbert number $H(n)$.

\smallskip

\noindent
{\rm (b)} 
First, it is easy to prove a weaker inequality that $H(n) \le S(n+2)$. For this, we observe that any $n$-degree polynomial ODE system~(\ref{odex_general})--(\ref{odey_general}) can be transformed into a chemical system of at most $(n+2)$-th order by shifting all limit cycles to the positive quadrant $[0,\infty)\times[0,\infty)$ and by multiplying both right-hand sides by $xy$. Since this corresponds to a rescaling of time in the original system, both the original system and the transformed system in ${\mathbb S}_{n+2}$ will have the same number of limit cycles. This implies that $H(n) \le S(n+2)$. 

In order to prove that $H(n) \le S(n+1)$ we  rely on the recent work of Plesa~\cite{Plesa:2024:MDS}, and also on the recent work of Gasull and Santana~\cite{Gasull:2024:NHP}.
Consider first the case $H(n) < \infty$. Then, according to Theorem~2 in~\cite{Gasull:2024:NHP}, there  exist ODE systems of degree $n$ with exactly $H(n)$ {\em hyperbolic} limit cycles. Then these limit cycles persist after small perturbations of the functions on the right-hand side of the ODEs; therefore, according to the results in~\cite{Plesa:2024:MDS},  we can transform these ODEs into {chemical} systems of at most $(n+1)$-th order which still have at least $H(n)$ limit cycles. This implies the desired inequality for the case $H(n) < \infty$. If $H(n) = \infty$, then, again according to Theorem~2 in~\cite{Gasull:2024:NHP}, for any $k>0$ there  exist ODE systems of degree $n$ with at least $k$ hyperbolic limit cycles, and, like before, this implies that $S(n+1) = \infty$.  

\smallskip

\noindent
{\rm (c)} This inequality follows directly from Lemma~\ref{setproperties}(b). Note that these inequalities also trivially hold for $n=1$, because there are no limit cycles in linear systems, {\it i.e.,} $W(1)=M(1)=S(1)=H(1)=0$.

\smallskip
 
\noindent
{\rm (d)} Consider an ODE system of degree $n$ that has $H(n)$ limit cycles. Without loss of generality we can assume that the limit cycles lie in the positive quadrant, because we can shift this system via a linear change of variables, and the shift does not affect the degree. Let us now multiply the right-hand side of this system by $xy$ as in the proof of part~{\rm (b)}. The resulting system belongs to ${\mathbb S}_{n+2}$ and has the same $H(n)$ limit cycles. According to results in~\cite{Gasull:2024:NHP}, there exists a small perturbation of this system that has $H(n)$ {\em hyperbolic} limit cycles in the positive orthant, such that the monomials on the right-hand side of the perturbed system are the same as the monomials of the original unperturbed system. Therefore, the perturbed system still belongs to ${\mathbb S}_{n+2}$. Recall now that ${\mathbb S}_{n+2} \subset {\mathbb M}_{n+3}$. Now we do a second small perturbation of this system, in order to bring it from ${\mathbb M}_{n+3}$ to ${\mathbb W}_{n+3}$. In this perturbation we simply make each reaction reversible, but with a very small reaction rate constant for any new reaction that we add to the network. If these rate constants are small enough it follows that we obtain a system in ${\mathbb W}_{n+3}$ that has at least $H(n)$ limit cycles, and therefore $H(n) \le W(n+3)$.
\end{proof}

\medskip

\noindent
 Some lower bounds on numbers $W{\hskip -0.3mm}(n),$ $M(n),$ $S(n)$ and $H(n)$ can also be established. They are summarized in Table~\ref{table1} and stated in lemmas below.

\medskip

\begin{lemma}
The {\rm ODEs} in set ${\mathbb M}_2$ cannot have any limit cycles, {\it i.e.} we have $M(2) = 0$ and $W(2)=0.$ 
\label{lemmaM2W2}
\end{lemma}

\begin{proof}
See  P\'ota~\cite{Pota:1983:TBS}, Tyson and Light~\cite{tyson1973properties}, Schuman and T\'oth~\cite{Schuman:2003:LCT} for the proof of $M(2)=0$. Using Lemma~\ref{WMNineq}, we have $W(2) \le M(2)$, which implies $W(2)=0.$   
\end{proof}

\medskip

\begin{lemma}
We have $S(2) \ge 3,$ $S(3) \ge 6$, $S(4) \ge 13,$ $M(3) \ge 3$, $M(4) \ge 6$, $W(3) \ge 3$ and $W(4) \ge 6.$
\label{lemmaS2S3S4M3M4W3W4}
\end{lemma}

\begin{proof}
The fact that $S(2) \ge 1$ has been known for more than 80 years~\cite{frank-kamenetsky:1943:possibility}, and is considered an important classical example in the mathematical theory of autocatalysis. The improved lower bound $S(2) \ge 3$ is due to Escher~\cite{escher1981bifurcation}. The results $S(3) \ge 6$ and $S(4) \ge 13$ have been established in the literature on Kolmogorov systems~\cite{Lloyd:2002:CKS,Carvalho:2023:NLB}.  A planar cubic (resp. quartic) ODE system with six (resp. thirteen) limit cycles in the positive quadrant has been presented in~\cite{Lloyd:2002:CKS} (resp. \cite{Carvalho:2023:NLB}). Applying Lemma~\ref{lemma1} to these systems, we conclude that there is a chemical system in ${\mathbb S}_3$  with six limit cycles and a chemical system in ${\mathbb S}_4$ with thirteen limit cycles, giving $S(3) \ge 6$ and $S(4) \ge 13$. Using Lemma~\ref{WMNineq}(c), we get $M(3) \ge S(2) \ge 3$ and $M(4) \ge S(3) \ge 3.$ Finally, we can add reactions with very small values of rate constants to make the corresponding systems (weakly) reversible. Such small perturbations will not change the existence of hyperbolic limit cycles~\cite{Smale:1974:DED,Perko:2013:DED}, giving $W(3) \ge 3$ and $W(4) \ge 6.$
\end{proof}

\begin{lemma}
We have $H(2) \ge 4$, $H(3) \ge 13$, $H(4) \ge 28$ and $H(n) \ge {\mathcal O}(n^2 log(n))$ as $n \to \infty.$
\label{lemmaHn}
\end{lemma}

\begin{proof}
See~\cite{Shi:1980:CEE} for $H(2) \ge 4$, \cite{Li:2009:CST} for $H(3) \ge 13$, \cite{Prohens:2018:NLB} for $H(4) \ge 28$ and~\cite{Christopher:2007:LCD} for the asymptotic inequality $H(n) \ge {\mathcal O}(n^2 log(n))$ as $n \to \infty.$
\end{proof}

\medskip

\begin{lemma}
We have $S(n) \ge {\mathcal O}(n^2 log(n))$, 
$M(n) \ge {\mathcal O}(n^2 log(n))$ 
and $W(n) \ge {\mathcal O}(n^2 log(n))$, as $n \to \infty.$
\label{lemmaasym}. 
\end{lemma}

\begin{proof}
Using Lemma~\ref{WMNineq}, we get $S(n) \ge H(n-1)$, $M(n) \ge S(n-1) \ge H(n-2)$ and $W(n) \ge H(n-3).$ The results then follow by applying the asymptotic inequality for $H(n)$ given in Lemma~\ref{lemmaHn}.
\end{proof}

\vskip 1mm

\noindent
The analysis of ODE systems with limit cycles which are used to achieve lower bounds in Table~\ref{table1} often cannot be supported by illustrative numerical simulations, because some parameter values are negligible (beyond the machine precision) when compared to other parameter values. However, there are also chemical systems in the literature with multiple limit cycles, where the phase plane can be computed using standard numerical methods. For example, the third-order system in ${\mathbb S}_3$ with three limit cycles, two stable and one unstable, is presented in~\cite{Plesa:2017:TMS}, and a $3$-molecular chemical system in ${\mathbb M}_3$ with two limit cycles, one stable and one unstable is studied in~\cite{nagy2020two}. 

In the following sections we will restrict our investigations to {\em algebraic} limit cycles, with a counterpart of Table~\ref{table1} for algebraic limit cycles presented in Section~\ref{sec4}. We also introduce a general approach in Theorem~\ref{theorem1} in Section~\ref{sec5} to obtain chemical systems where we will be able to calculate their phase planes with multiple (algebraic) limit cycles and present some illustrative numerical results.

\section{Chemical systems with algebraic limit cycles}

\label{sec4}

The analogues of the Hilbert number $H(n)$ for  polynomial ODE systems corresponding to the sets ${\mathbb S}_n$, ${\mathbb M}_n$ and ${\mathbb  W}_{\!n}$ have been given in Definition~\ref{defOMW}. In this section, we will focus on algebraic limit cycles~\cite{Chavarriga:2004:ALC,Gasull:2023:NLC}. An algebraic limit cycle is not only a closed isolated solution of the ODE system, but it can also be represented as a closed component of an algebraic curve $h(x,y)=0$, where $h$ is a polynomial of degree~$n_h$. Note that, since the flow of the planar ODE system~
(\ref{odex_general})--(\ref{odey_general})
is tangent to the algebraic curve, we have
\begin{equation}
\frac{\partial h}{\partial x}(x,y)
\, f(x,y)
\,
+
\,
\frac{\partial h}{\partial y}
\,
g(x,y)
\, = \,
s(x,y) \, h(x,y) \,,
\label{alglimcyclcond}    
\end{equation}
where $s(x,y)$ is a polynomial, called {\it cofactor} of $h.$ First, we define versions of $S(n),$ $M(n)$ and $W(n)$ for counting only algebraic limit cycles.

\vskip 2mm

\begin{definition}
We denote by $S^a(n)$ the maximum number of algebraic limit cycles for {\rm ODEs} in the set~$\,{\mathbb S}_n$, by $M^a(n)$ the maximum number of algebraic limit cycles for {\rm ODEs} in the set $\,{\mathbb M}_n$ of $n$-molecular chemical reaction networks, and by $W^a(n)$ the maximum number of algebraic limit cycles for {\rm ODEs} in the set ${\mathbb W}_{\!n}$ of weakly reversible chemical reaction networks.
\label{defoamawa}
\end{definition}

\vskip 2mm

\noindent
A counterpart of Lemma~\ref{WMNineq} establishing inequalities between numbers $W{\hskip -0.3mm}(n),$ $M(n)$, $S(n)$ and $H(n)$ can also be formulated for numbers $W^a(n),$ $M^a(n)$, $S^a(n)$ and $H^a(n)$  counting only algebraic limit cycles.

\vskip 2mm

\begin{lemma}
We have \hskip 2mm
{\rm (a)}
$W^a(n) \le M^a(n) \le S^a(n) \le H^a(n) \le S^a(n+2)$
for all $n \in {\mathbb N}\,$, \hfill\break
\rule{0pt}{1pt} \hskip 34.5mm
{\rm (b)}
$S^{(a)}(n) \le M^{(a)}(n+1)$ for all $n \in {\mathbb N}\,$.
\label{WaMaNaineq}
\end{lemma}

\begin{proof}
The proof follows some of the same arguments as in the proof of Lemma~\ref{WMNineq}, where we replace limit cycles by algebraic limit cycles.
\end{proof}

\vskip 2mm

\noindent
Some lower bounds on numbers $W^a(n)$, $M^a(n)$ and $S^a(n)$ are given in Table~\ref{table2} and in Lemmas~$\ref{lemmaSMWalge234}$, $\ref{lemmaWa1}$ and~$\ref{lemmaasyma}$.

\begin{table}
\begin{tabular}{|c|c|c|c|c|}
\hline
\raise -1.4mm \hbox{\rule{0pt}{4.6mm}} $\qquad\;\;\,$ {\color{red} $n$} $\qquad\quad$ & $\qquad\quad$ {\color{red} $W^a{\hskip -0.2mm}(n)$} $\qquad\quad$ & $\qquad\quad$ {\color{red} $M^a(n)$} $\qquad\quad$ & $\qquad\quad${\color{red} $S^a(n)$} $\qquad\quad$ & $\qquad\quad$ {\color{red} $H^a(n)$} $\qquad\quad$
\\
\hline
\hline
\raise -1.4mm \hbox{\rule{0pt}{4.6mm}}{\color{red} 2} & = 0 & = 0 & $\ge$ 1 & $\ge$ 1  \\
\hline
\raise -1.4mm \hbox{\rule{0pt}{4.6mm}}{\color{red} 3} & $\ge$ 0 & $\ge$ 1 & $\ge$ 1 & $\ge$ 2 \\
\hline
\raise -1.4mm \hbox{\rule{0pt}{4.6mm}}{\color{red} 4} & $\ge$ 1 & $\ge$ 1 & $\ge$ 3 & $\ge$ 4 \\
\hline
\hline
\raise -1.4mm \hbox{\rule{0pt}{4.6mm}}{\color{red} as $n \to \infty$} & $\ge {\mathcal O}(n)$ & $\ge {\mathcal O}(n^2$) & $\ge {\mathcal O}(n^2$) & $\ge {\mathcal O}(n^2$)  \\
\hline
\end{tabular}
\vskip 2mm
\caption{{\it Some values and estimates from below on numbers $W^a(n)$, $M^a(n)$, $S^a(n)$ and $H^a(n)$, see Lemma~$\ref{lemmaSMWalge234}$, Lemma~$\ref{lemmaWa1}$, Lemma~$\ref{lemmaasyma}$ and Theorem~$\ref{theoremWam}$}.}
\label{table2}
\end{table}

\vskip 2mm

\begin{lemma}
We have $H^a(2) \ge 1$, $H^{(a)}(3) \ge 2$, $H^a(4) \ge 4,$ $S^a(2) \ge 1$, $S^{(a)}(3) \ge 1$, $S^a(4) \ge 3,$ $M^{(a)}(2) = 0,$ $M^a(3) \ge 1,$ $M^{(a)}(4) \ge 1$ and $W^{(a)}(2) = 0.$
\label{lemmaSMWalge234}
\end{lemma}

\begin{proof}
See~\cite{Llibre:2010:OHP} for $H^a(2) \ge 1$, $H^{(a)}(3) \ge 2$ and $H^a(4) \ge 4.$ Since Lemma~\ref{lemmaM2W2} gives $M(2) = 0$ and $W(2)=0$, we also have $M^{(a)}(2) = 0$ and $W^{(a)}(2) = 0$ when we restrict to algebraic limit cycles in sets~${\mathbb M}_2$ and ${\mathbb W}_2$. To show $S^a(2) \ge 1$, we need to find a quadratic chemical system with an algebraic limit cycle. Consider the quadratic system~\cite{escher1979models}
\begin{eqnarray}
\frac{\mbox{d}x}{\mbox{d}t}
&=&
2x^2 - xy + \frac{3}{2}, 
\label{quadraticescher1}
\\
\frac{\mbox{d}y}{\mbox{d}t}
&=& 
\frac{5}{2}x^2 - xy - y + \frac{17}{4}.
\label{quadraticescher2}
\end{eqnarray}
Using Lemma~\ref{lemma1}, the ODE system~(\ref{quadraticescher1})--(\ref{quadraticescher2}) belongs to ${\mathbb S}_2.$ Moreover, it is easy to verify that it has an algebraic limit cycle in the positive quadrant, which is the ellipse given by~\cite{escher1979models}
\begin{equation}
h(x,y) = 10x^2 - 12xy + 4y^2 + 20x - 16y + 19 = 0
\label{ellipselimitcycle}
\end{equation}
with the cofactor $s(x,y)$ in equation~(\ref{alglimcyclcond}) given as $s(x,y)=x-2.$ Consequently, we have $S^a(2) \ge 1$. Using Lemma~\ref{WaMaNaineq}, we conclude 
\begin{equation}
1 \le S^{(a)}(2) \le M^{(a)}(3) \le S^{(a)}(3) \le M^{(a)}(4).
\label{ineqproof}
\end{equation}
Finally, to show $S^a(4) \ge 3,$ consider quartic algebraic curve $h(x,y)=0$ given by
\begin{equation}
h(x,y) 
\,=\, 
x^{2}y^{2}
\, - \,
\frac{9}{10^3} 
\, \big(
x^{3}y +
xy^{3} 
\big)
\, + \,
\frac{6}{10^4} \, 
\big( x^{3}+y^{3} \big)
\,
+
\,
\frac{2}{50}
\, \big(
x^{2} y + x y^{2} \big)
\, - \, 2 xy 
\, + \, \frac{934}{10^3} \, ,
\label{threeovalcurve}
\end{equation}
which has three ovals in the positive quadrant. We visualize them in Figure~\ref{figure2}(a). Next, we consider the line $\, y  = 7 x \,$, which does not intersect the three ovals of $h(x,y)=0$, as it is shown in Figure~\ref{figure2}(a). Then, we can apply \cite[Theorem 1]{christopher2001polynomial} to deduce that the ODE system 
\begin{eqnarray}
\frac{\mbox{{\rm d}}x}{\mbox{{\rm d}}t}
\, & = &
\,
h(x,y) \, + \, (y - 7 x) \, \frac{\partial h}{\partial y}(x,y) \,, 
\label{christopher1}
\\
\frac{\mbox{{\rm d}}y}{\mbox{{\rm d}}t}
\, & = &
\,
h(x,y) \, + \,  (7 x - y ) \,  \frac{\partial h}{\partial x}(x,y) \,,
\label{christopher2}
\end{eqnarray}
is a polynomial system of degree 4 which has the three ovals of $h(x,y)=0$ in the positive quadrant as hyperbolic algebraic limit cycles. Using Lemma~\ref{lemma1}, we observe that the ODE system~(\ref{christopher1})--(\ref{christopher2}) is in set~${\mathbb S}_4$. Therefore, we conclude that $S^a(4) \ge 3.$ 
\end{proof}

\vskip 2mm
\begin{figure}
\centerline{\hskip 0.2mm \includegraphics[height=5.6cm]{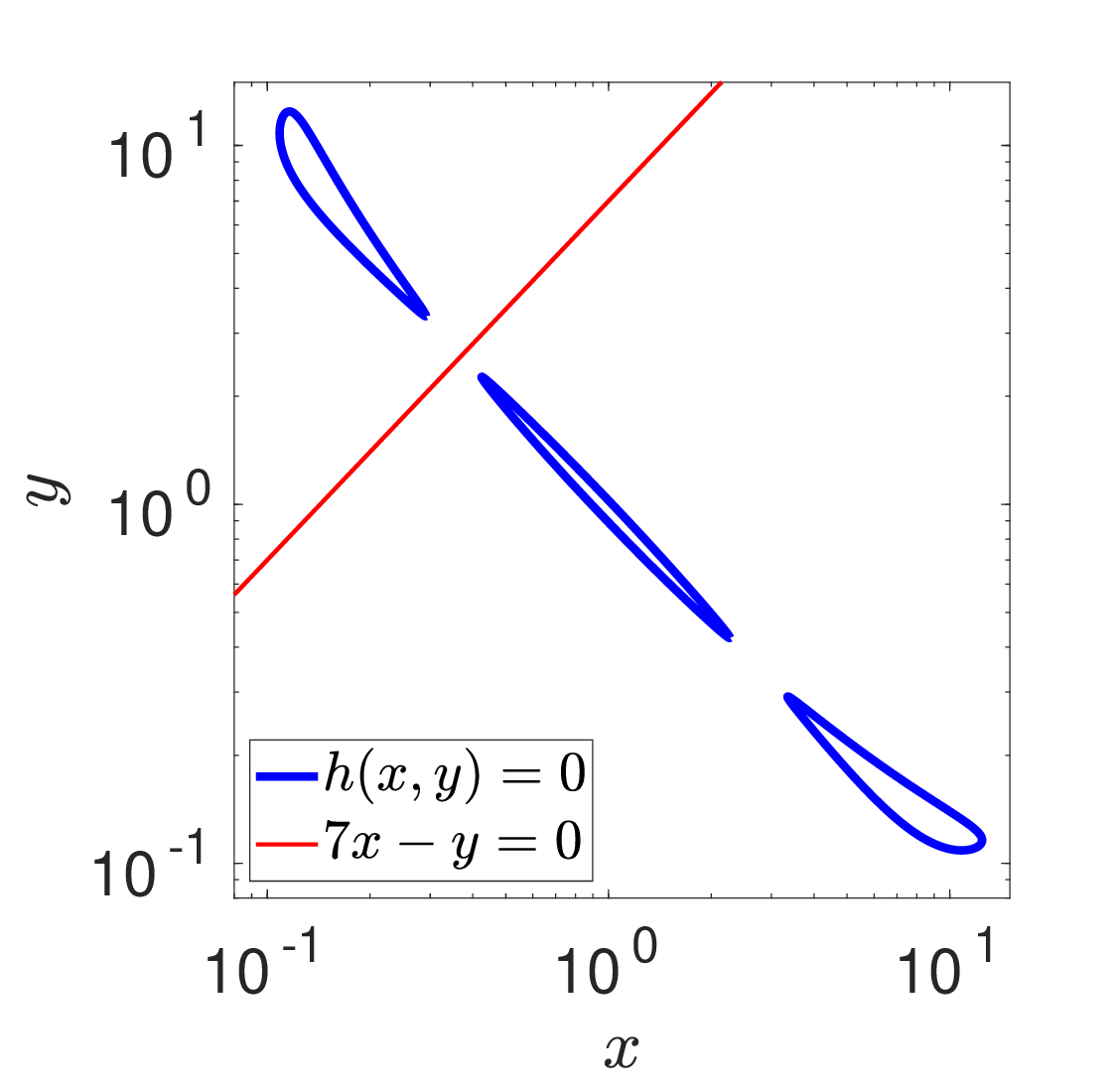}\hskip 0.2mm \includegraphics[height=4.97cm]{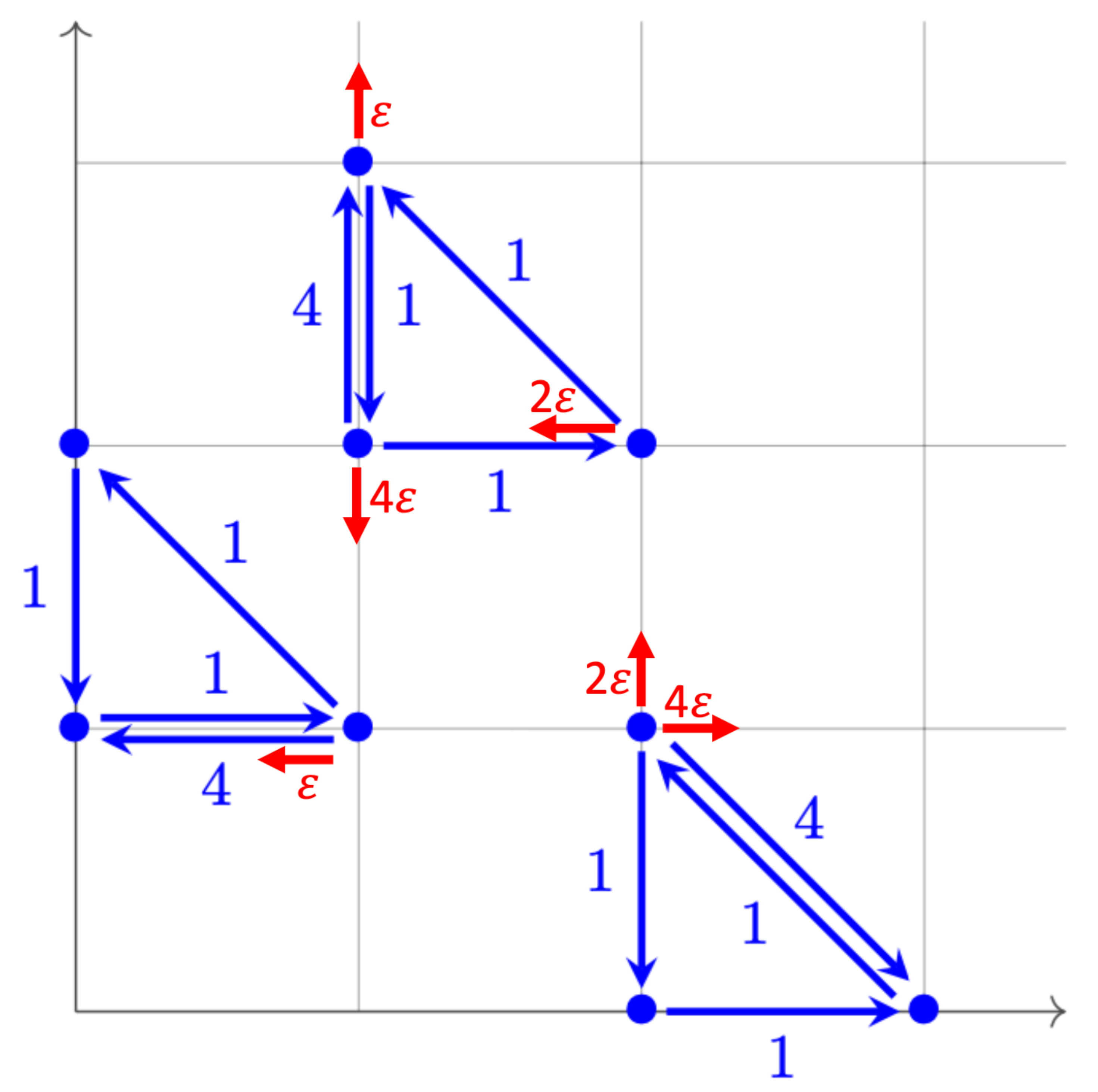} \hskip 0.2mm \includegraphics[height=4.97cm]{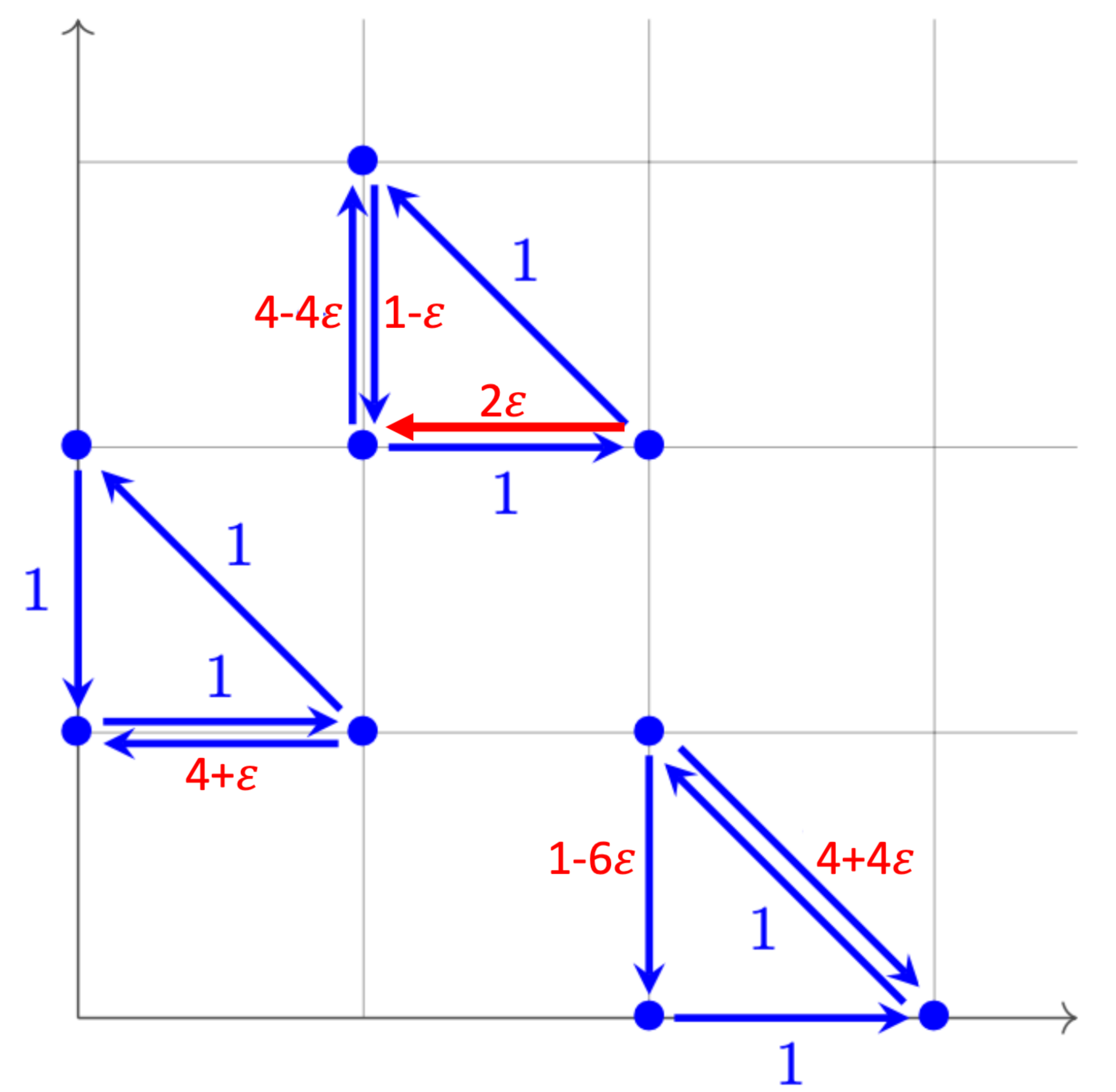}}
\vskip -5.5cm
\leftline{(a) \hskip 5.18cm (b) \hskip 4.56cm (c)}
\vskip 5.2cm
\caption{(a) {\it The quartic algebraic curve $h(x,y)=0$ given by~$(\ref{threeovalcurve})$ has three (blue) ovals in the positive quadrant. The red line is given by $7x-y=0$. The red line does not intersect the (blue) ovals of $h(x,y)=0$ in the positive quadrant. We use log scale on the $x$-axis and $y$-axis.} \hfill\break
(b) {\it  The chemical reaction network corresponding to the blue edges of the planar {\rm E-graph} realizes the `unperturbed' {\rm ODE} system~$(\ref{WR_network_small_x})$--$(\ref{WR_network_small_y})$, while the $\varepsilon$-perturbations are shown by the red edges.} \hfill\break
(c) {\it The weakly reversible network that consists of the blue edges (with some modified rate constants) together with the red edge provides a weakly reversible realization of the perturbed system~$(\ref{WR_network_small_perturbed_x})$--$(\ref{WR_network_small_perturbed_y})$.}}
\label{figure2}
\end{figure}

\noindent
In Lemma~\ref{lemmaSMWalge234}, we have presented a second-order chemical system in ${\mathbb S}_2$ with the algebraic limit cycle given as ellipse~(\ref{ellipselimitcycle}). The corresponding ODE system~(\ref{quadraticescher1})--(\ref{quadraticescher2}) has quadratic polynomials on the right-hand side. Another quadratic polynomial dynamical system with an algebraic limit cycle is~\cite{Chavarriga:2004:ALC}
\begin{eqnarray}
\frac{\mbox{d}x}{\mbox{d}t}
\,
&=&
\,
2 \, \big(1 \, + \,2\,x \, - \, 2 \, c \, x^2 \, + \, 6 \,x\,y \big)\,, 
\label{chavar1}
\\
\frac{\mbox{d}y}{\mbox{d}t}
\,
&=& 
\, 8 \, - \, 3 \, c \, - \, 14\,c\,x 
\, - \, 2 \, c \, x\,y \, - \, 8 \, y^2 \, .
\label{chavar2}
\end{eqnarray}
As explained in~\cite{Chavarriga:2004:ALC}, for any $c \in (0,1/4)$ the ODE system~(\ref{chavar1})--(\ref{chavar2}) has an invariant algebraic curve, which is a limit cycle~\cite{Gasull:2023:NLC}. The ODE system~(\ref{chavar1})--(\ref{chavar2}) is not a chemical system because of the term $-14 {\hskip 0.2mm} c {\hskip 0.2mm} x$ in equation~(\ref{chavar2}). However, we can use this quadratic system for $c = 1/8$ to construct a cubic mass-action system that has an algebraic limit cycle. For this we note that the limit cycle of the quadratic system~(\ref{chavar1})--(\ref{chavar2}) is contained in the set $(0, \infty) \times (-1, \infty)$, and therefore if we shift this system one unit in the $y$-direction, and then  multiply its right-hand side by a factor of $y$, we obtain
\begin{eqnarray*}
\frac{\mbox{d}x}{\mbox{d}t}
&=&
2 \, \big(1 \, + \,2\,x \, - \, 2 \, c \, x^2 \, + \, 6 \,x \, (y-1) \big) \, y, \\
\frac{\mbox{d}y}{\mbox{d}t}
&=& 
\, \big(
8 \, - \, 3 \, c \, - \, 14\,c\,x 
\, - \, 2 \, c \, x \, (y-1) \, - \, 8 \, (y-1)^2 \big) \, y \, .
\end{eqnarray*}
Then, the cubic system above is a mass-action system; moreover, it has an algebraic limit cycle, because its trajectory curves are the shifted versions of the trajectory curves of the system~(\ref{chavar1})--(\ref{chavar2}).
Such a construction provides an alternative way for us to show that $S^{(a)}(3) \ge 1$, which we have previously established in the proof of Lemma~\ref{lemmaSMWalge234} by using inequalities~(\ref{ineqproof}).

In Section~\ref{sec3}, it has been relatively straightforward to establish that weakly reversible chemical systems can give rise to limit cycles by using small perturbations of non-reversible chemical systems. This is not the case, when we consider  algebraic limit cycles, because a small perturbation can change an algebraic limit cycle to a non-algebraic one.
In our next lemma, we show that weakly reversible chemical systems can have algebraic limit cycles by constructing a quartic weakly reversible two-species system. The question of whether cubic weakly reversible two-species systems can give rise to algebraic limit cycles remains open, leaving us with inequality $W^{(a)}(3) \ge 0$ in Table~\ref{table2}. In our construction we rely on a general approach for constructing weakly reversible systems that have a curve of equilibria; this approach has been introduced in~\cite{Boros:2020:WRM}. 

\vskip 2mm

\begin{lemma}
We have $W^a(4) \ge 1.$
\label{lemmaWa1}
\end{lemma} 

\begin{proof} We will construct a weakly reversible system that has an algebraic limit cycle, as follows. 
We start with a weakly reversible system constructed in~\cite[Example 4.3]{Boros:2020:WRM}, given by
\begin{eqnarray}
\label{WR_network_small_x}
\frac{\mbox{d}x}{\mbox{d}t} 
&=& 
\big(x^2+xy^2+y-4xy\big) \, (1-x), \\
\label{WR_network_small_y}
\frac{\mbox{d}y}{\mbox{d}t}
&=& 
\big(x^2+xy^2+y-4xy\big)\,(x-y).
\end{eqnarray}
The common factor 
\begin{equation}
h(x,y) = x^2+xy^2+y-4xy  
\label{commfach}
\end{equation}
results in a curve of positive steady states within the positive quadrant. Note that the polynomials on the right-hand side of the ODE system~(\ref{WR_network_small_x})--(\ref{WR_network_small_y}) have degree $4$. They can be obtained by multiplying the linear system~(\ref{odelinear1})--(\ref{odelinear2}) by~(\ref{commfach}). In Figure~\ref{figure1}(c), we illustrated that the multiplication of the linear system~(\ref{odelinear1})--(\ref{odelinear2}) by positive monomials results in copies of the network. Generalizing this observation to the ODE system~(\ref{WR_network_small_x})--(\ref{WR_network_small_y}), we can realize it by the chemical reaction network shown in blue in Figure~\ref{figure2}(b) or Figure~\ref{figure2}(c).
We now consider a perturbed version of the ODE system~(\ref{WR_network_small_x})--(\ref{WR_network_small_y}), also of degree 4, given by
\begin{eqnarray}
\frac{\mbox{d}x}{\mbox{d}t} 
&=& 
\big(x^2+xy^2+y-4xy\big)\,(1-x)
\, - \, 
\varepsilon \, x \, y
\, \frac{\partial h}{\partial y}(x,y) \,, 
\label{ex1x} \\
\frac{\mbox{d}y}{\mbox{d}t}
&=& 
\big(x^2+xy^2+y-4xy\big)\,(x-y)
\, + \, 
\varepsilon \, x \, y
\, \frac{\partial h}{\partial x}(x,y) \,,
\label{ex1y}
\end{eqnarray}
which implies
\begin{eqnarray}
\label{WR_network_small_perturbed_x}
\frac{\mbox{d}x}{\mbox{d}t} 
&=& 
\big(x^2+xy^2+y-4xy\big)\,(1-x)  
\, - \, \varepsilon \, x \, y \, 
\big(2xy +1 - 4x\big)\,,
\\ 
\label{WR_network_small_perturbed_y}
\frac{\mbox{d}y}{\mbox{d}t}
&=& 
\big(x^2+xy^2+y-4xy\big)\,(x-y)
\, + \, \varepsilon \, x \,y  \, 
\big(2x + y^2 -4y\big)\,.
\end{eqnarray}
The ODE system~(\ref{WR_network_small_perturbed_x})--(\ref{WR_network_small_perturbed_y}) has an algebraic limit cycle given by $h(x,y)=0$. One possible way to check that the periodic solution that lies along the curve $h(x,y)=0$ is indeed a limit cycle is to look at a more general setting which is discussed in depth in Section~\ref{sec5}; specifically, it is easy to check that the transversality  condition~(\ref{transcondition}) holds for the ODE system~(\ref{WR_network_small_perturbed_x})--(\ref{WR_network_small_perturbed_y}).

The red edges in Figure~\ref{figure2}(b) suggest a realization of the ODE system (\ref{WR_network_small_perturbed_x})--(\ref{WR_network_small_perturbed_y}) as a chemical reaction network in ${\mathbb S}_4$ for any  $\varepsilon > 0$, but this particular realization is not weakly reversible. However, reaction network realizations of polynomial dynamical systems are not unique in general~\cite{Craciun_Pantea_2008,Plesa:2018:NCM,Craciun:2020:ECC}. If~$\varepsilon \in (0, 1/6)$, then there do exist weakly reversible realizations of the ODE system~(\ref{WR_network_small_perturbed_x})--(\ref{WR_network_small_perturbed_y}); one such realization is shown in Figure~\ref{figure2}(c). Therefore, the ODE system~(\ref{WR_network_small_perturbed_x})--(\ref{WR_network_small_perturbed_y}) is in ${\mathbb W}_4$ for $\varepsilon \in (0, 1/6)$, giving $W^a(4) \ge 1.$
\end{proof}

\begin{lemma}
We have $H^a(n) \ge {\mathcal O}(n^2)$, $S^a(n) \ge {\mathcal O}(n^2)$ and
$M^a(n) \ge {\mathcal O}(n^2)$, as $n \to \infty.$
\label{lemmaasyma}
\end{lemma}

\begin{proof}
See~\cite{Llibre:2010:OHP} for $H^a(n) \ge {\mathcal O}(n^2)$ as $n \to \infty.$
The next two asymptotic inequalities follow from Lemma~\ref{WaMaNaineq}.
\end{proof}

\noindent
Lemmas~$\ref{lemmaSMWalge234}$, $\ref{lemmaWa1}$ and~$\ref{lemmaasyma}$ have justified all lower bounds in Table~\ref{table2}, except of the asymptotic inequality $W^a(n) \ge {\mathcal O}(n)$ as $n \to \infty.$ We will show this inequality in the next subsection by considering reversible chemical systems, which is even more restrictive class of chemical reaction networks than weakly reversible systems. In particular, we will show that reversible chemical systems can have (multiple) algebraic limit cycles.

\subsection{Algebraic limit cycles for reversible chemical systems}
\label{reversible}

In this section we describe a general approach for constructing reversible systems with algebraic limit cycles. We start with the simple chemical reaction network shown in Figure~\ref{figure3}(a). If we choose all the reaction rate constants to be equal to 1, the corresponding reaction rate equations are given by the ODE system
\begin{eqnarray}
\frac{\mbox{d}x}{\mbox{d}t} 
&= 
1-x+y-xy,
\label{unit_square_system_x} \\
\frac{\mbox{d}y}{\mbox{d}t}
&= 
1+x-y-xy,
\label{unit_square_system_y}
\end{eqnarray}
which has a globally attracting point at $(x,y)=(1,1)$. 
Next, we consider the algebraic curve $h(x,y)=0$ of degree $n_h=4$ given by
\begin{equation}
h(x,y) = x^2y^2+x^2y+xy^2+x^2+y^2+x+y+1-9xy\, , 
\label{h_x_y_9}
\end{equation}
then, within the positive quadrant, the equation $h(x,y)=0$ is satisfied along a simple closed  curve. Indeed, we can rewrite $h(x,y)=0$ as
$$
\left(x + 1 + \frac{1}{x} \right)
\left(y + 1 + \frac{1}{y} \right)
=
10\,,
$$
which has two solutions $y$ for each $x$ satisfying $(7-\sqrt{13})/6 < x < (7+\sqrt{13})/6$. Plotting the values of $y$ as functions of $x$ in Figure~\ref{figure4}(a), we obtain the two branches of the closed curve visualized as the blue line. A geometric representation of the monomials of $h(x,y)$ in~(\ref{h_x_y_9}) is shown in Figure~\ref{figure3}(b).
Multiplying the right-hand side of the ODE system~(\ref{unit_square_system_x})--(\ref{unit_square_system_y}) by $h(x,y)$, we obtain the ODE system
\begin{eqnarray}
\frac{\mbox{d}x}{\mbox{d}t} 
&=&
\big(x^2y^2+x^2y+xy^2+x^2+y^2+x+y+1-9xy\big)
\, \big(1-x+y-xy\big)\,,
\label{min_reversible_system_x} \\
\frac{\mbox{d}y}{\mbox{d}t}
&=& 
\big(x^2y^2+x^2y+xy^2+x^2+y^2+x+y+1-9xy\big)
\, \big(1+x-y-xy\big)\,,
\label{min_reversible_system_y}
\end{eqnarray}
which has polynomials of degree 6 on the right hand side. The ODE system~(\ref{min_reversible_system_x})--(\ref{min_reversible_system_y}) has a curve of equilibria~\cite{Boros:2020:WRM} given by $h(x,y)=0$. Moreover, since the ODE system~(\ref{min_reversible_system_x})--(\ref{min_reversible_system_y}) has been obtained by multiplying the ODE system~(\ref{unit_square_system_x})--(\ref{unit_square_system_y}) by a polynomial, the corresponding planar E-graph representation will consists of shifted copies of the reaction network in Figure~\ref{figure3}(a) for each multiplication by a positive monomial, as we have already observed in Figure~\ref{figure1}(c).

\begin{figure}
\leftline{
\hskip 1mm
(a) 
\hskip 4.65cm 
(b)
\hskip 4.7cm 
(c)
}
\centerline{\hskip 2mm \includegraphics[height=4.5cm]{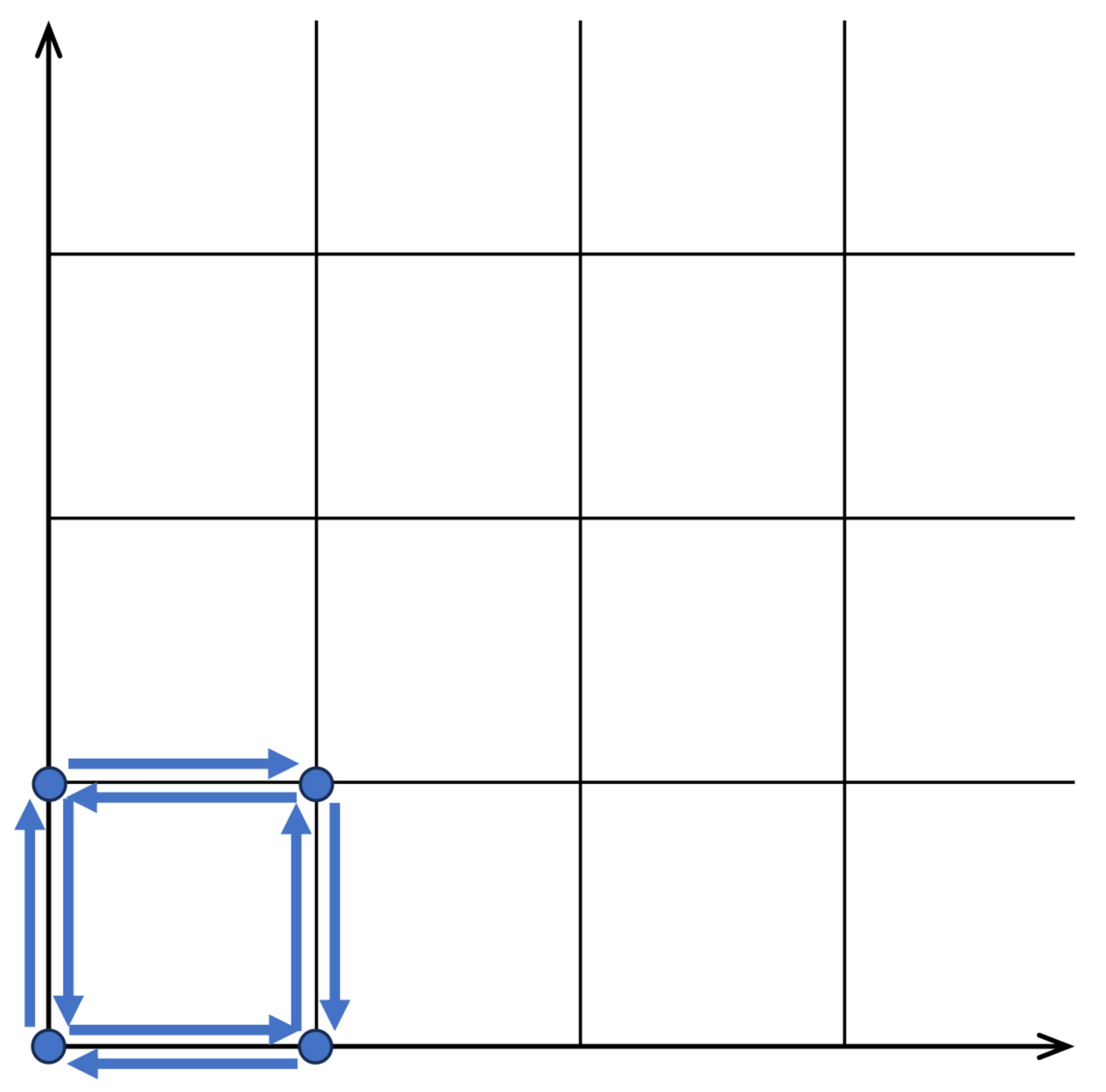}
\hskip 6mm
\includegraphics[height=4.5cm]{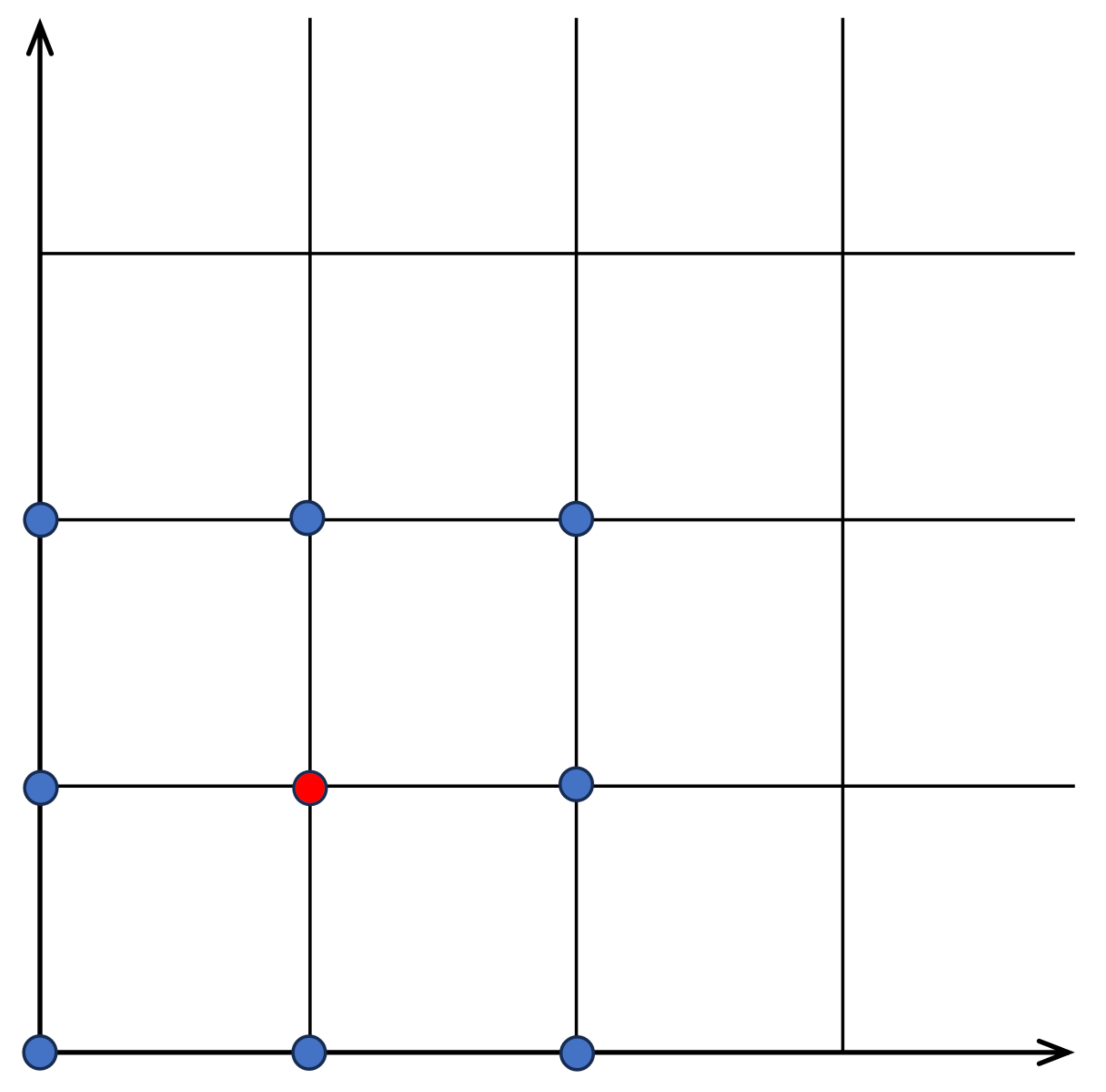}
\hskip 6mm
\includegraphics[height=4.5cm]{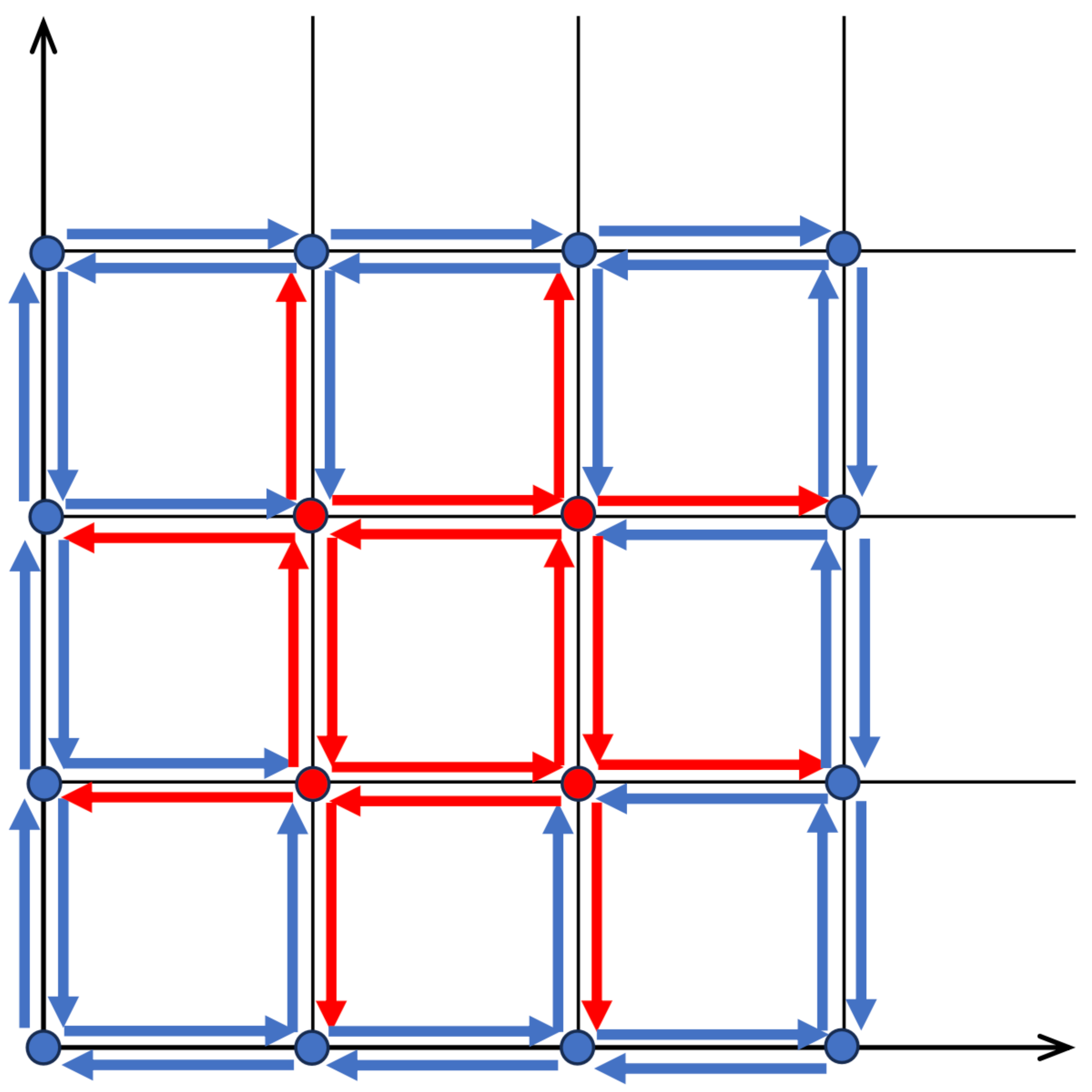}}
\caption{(a) {\it Planar {\rm E-graph} of a reversible chemical reaction network corresponding to the {\rm ODE} system~$(\ref{unit_square_system_x})$--$(\ref{unit_square_system_y})$ as its reaction rate equations with all reaction rate constants equal to $1.$} \hfill\break
(b) {\it A geometric representation of the monomials of the polynomial $h(x,y)$ given by~$(\ref{h_x_y_9})$. The blue points represent the monomials with positive coefficients and the red point represents its negative monomial $-9xy$.} \hfill\break
(c) {\it The dynamical system obtained by multiplying the network shown in {\rm (a)} by the factor $h(x,y)$ shown in {\rm (b)} $($which gives rise to the equations~$(\ref{min_reversible_system_x})$--$(\ref{min_reversible_system_y}))$ can be realized by this reversible `full-grid' network.}}
\label{figure3}
\end{figure}

We argue that the ODE system~(\ref{min_reversible_system_x})--(\ref{min_reversible_system_y}) has a weakly reversible realization given by the network in Figure~\ref{figure3}(c), and, moreover, this realization can be chosen such that all reactions have reaction rate constants $\ge 1$. 
For this, we first note that the reactions shown in {\em blue} in Figure~\ref{figure3}(c) can all be chosen to have reaction rate constants equal to 1, because these are obtained from reactions in Figure~\ref{figure3}(a) after multiplying with one of the {\em positive} monomials of $h(x,y)$. 

On the other hand, the reactions shown in {\em red} in Figure~\ref{figure3}(c) may have rate constants that are impacted by multiplication with some positive and some negative monomials of $h(x,y)$, so their size (and even their sign) are not immediately clear. Nevertheless, note that, no matter what values these rate constants have to begin with, if we increase all of them by an arbitrarily chosen constant, then {\em the effect of all these increases cancels out.} This is due to the fact that the red reactions can be partitioned into pairs, such that each pair of reactions originates at the same red node, and the two reactions within each such pair point exactly opposite from each other.
Therefore, we conclude that the system (\ref{min_reversible_system_x})--(\ref{min_reversible_system_y}) can be realized by the network shown in Figure~\ref{figure3}(c). 
Consider now a perturbed version of this system, also of degree 6, given by
\begin{eqnarray}
\frac{\mbox{d}x}{\mbox{d}t} 
&=&
\big(x^2y^2+x^2y+xy^2+x^2+y^2+x+y+1-9xy\big)\,\big(1-x+y-xy\big) \, - \, 
\varepsilon \, x \, y
\, \frac{\partial h}{\partial y}(x,y) \,,
\qquad
\label{ex2x}
\\
\frac{\mbox{d}y}{\mbox{d}t}
&=& 
\big(x^2y^2+x^2y+xy^2+x^2+y^2+x+y+1-9xy\big)\, \big(1+x-y-xy\big) \, + \, 
\varepsilon \, x \, y
\, \frac{\partial h}{\partial x}(x,y) \,,
\qquad
\label{ex2y}
\end{eqnarray}
which implies
\begin{eqnarray}
\frac{\mbox{d}x}{\mbox{d}t} 
&=&
\big(x^2y^2+x^2y+xy^2+x^2+y^2+x+y+1-9xy\big)\, \big(1-x+y-xy\big) \nonumber \\ 
&& \qquad - \, \varepsilon \, x \, y \big(2x^2y +
x^2+2xy+2y+1-9x\big) \,,
\label{min_reversible_system_perturbed_x}
\\
\frac{\mbox{d}y}{\mbox{d}t}
&=& 
\big(x^2y^2+x^2y+xy^2+x^2+y^2+x+y+1-9xy\big)\,\big(1+x-y-xy\big)
\nonumber \\
&& \qquad + \, \varepsilon \, x \, y \, \big(2xy^2+
y^2+2xy+2x+1-9y\big) \, .
\label{min_reversible_system_perturbed_y}
\end{eqnarray}
The ODE system~(\ref{min_reversible_system_perturbed_x})--(\ref{min_reversible_system_perturbed_y}) has been constructed in a similar way as the ODE system~(\ref{WR_network_small_perturbed_x})--(\ref{WR_network_small_perturbed_y}). 
Like in that example, it is easy to check that the transversality condition~(\ref{transcondition}) also holds in this case.
We therefore conclude that the ODE system~(\ref{min_reversible_system_perturbed_x})--(\ref{min_reversible_system_perturbed_y}) has an algebraic limit cycle, plotted as the blue line in Figure~\ref{figure4}(a). 

\begin{figure}
\leftline{\includegraphics[height=5.3cm]{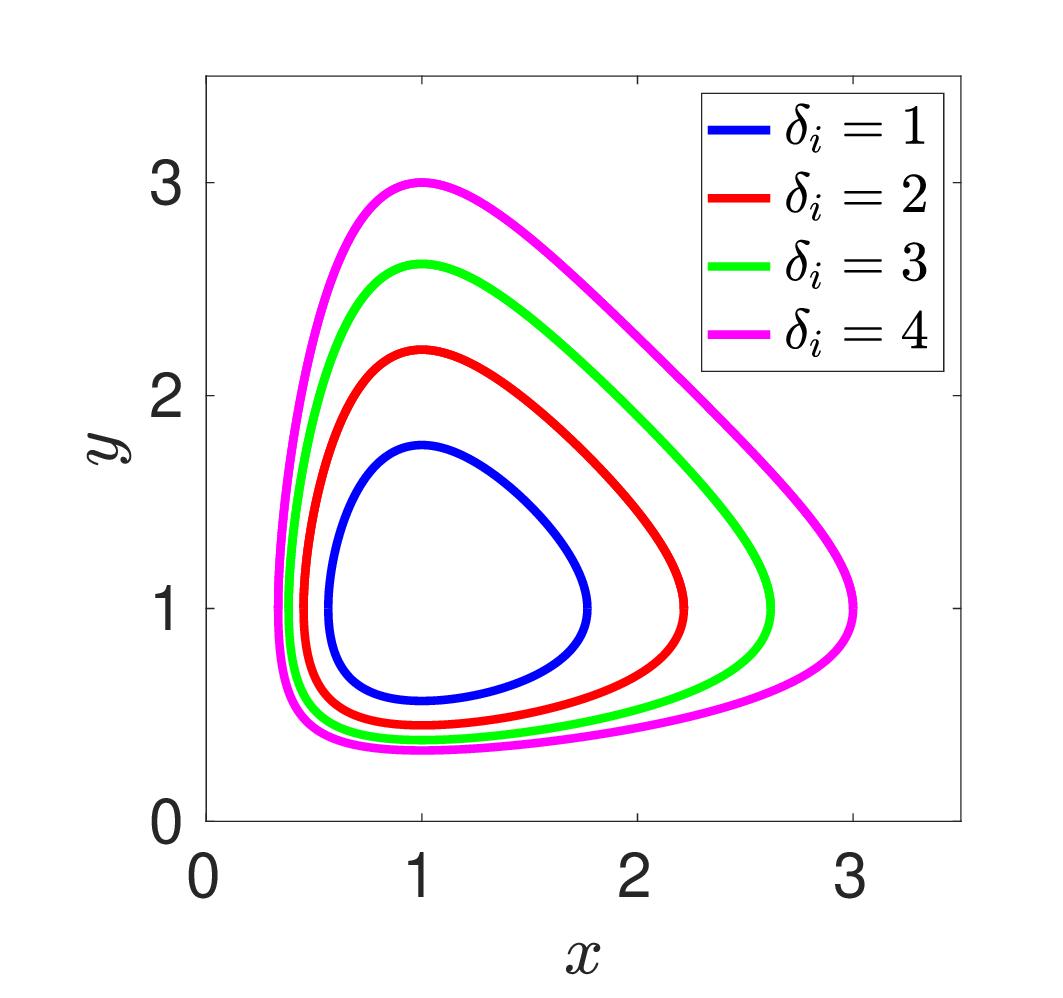}\hskip 1mm \raise 3.5mm
\hbox{\includegraphics[height=4.6cm]{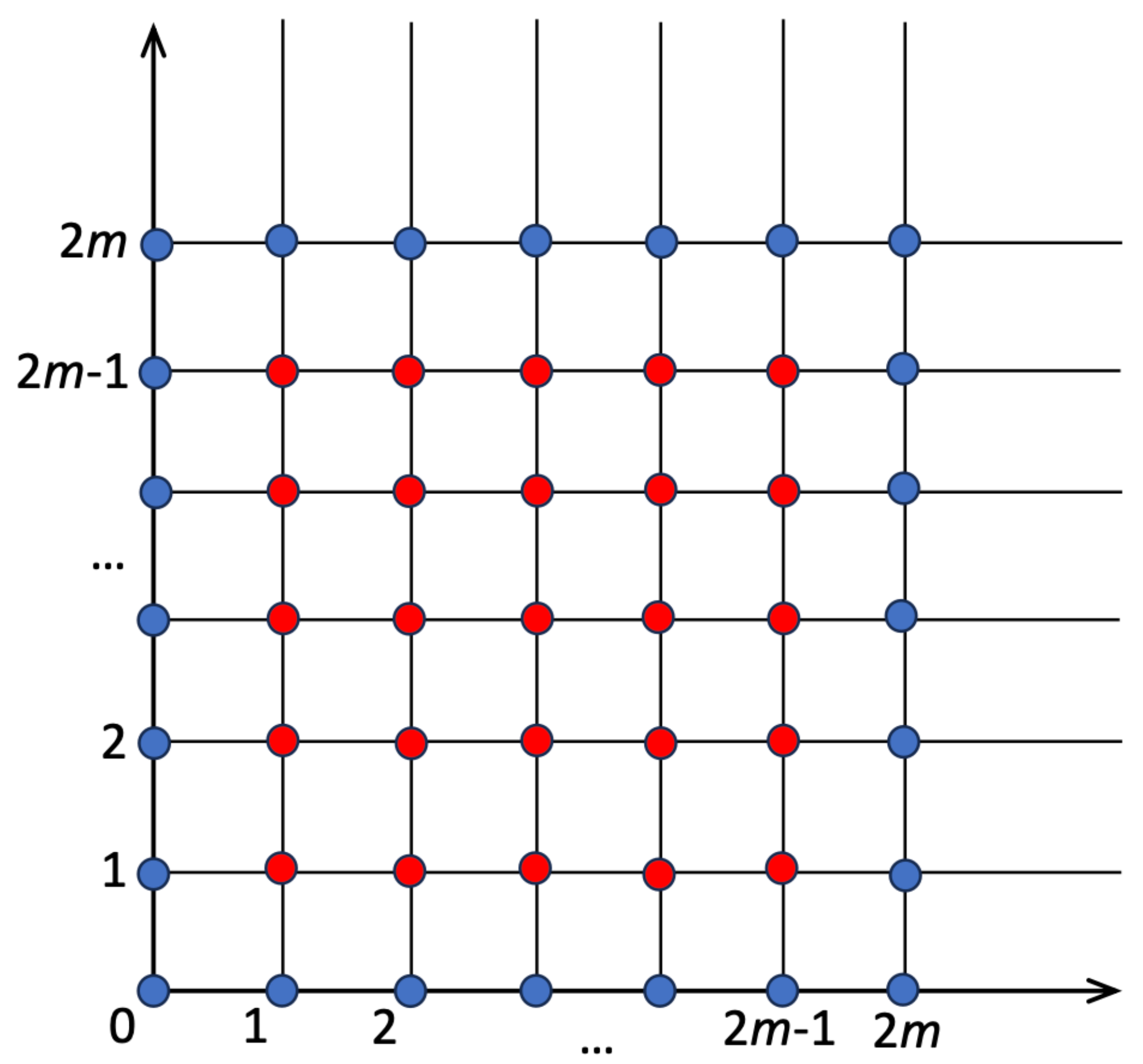}} \hskip 1.5mm \raise 3.5mm \hbox{\includegraphics[height=4.6cm]{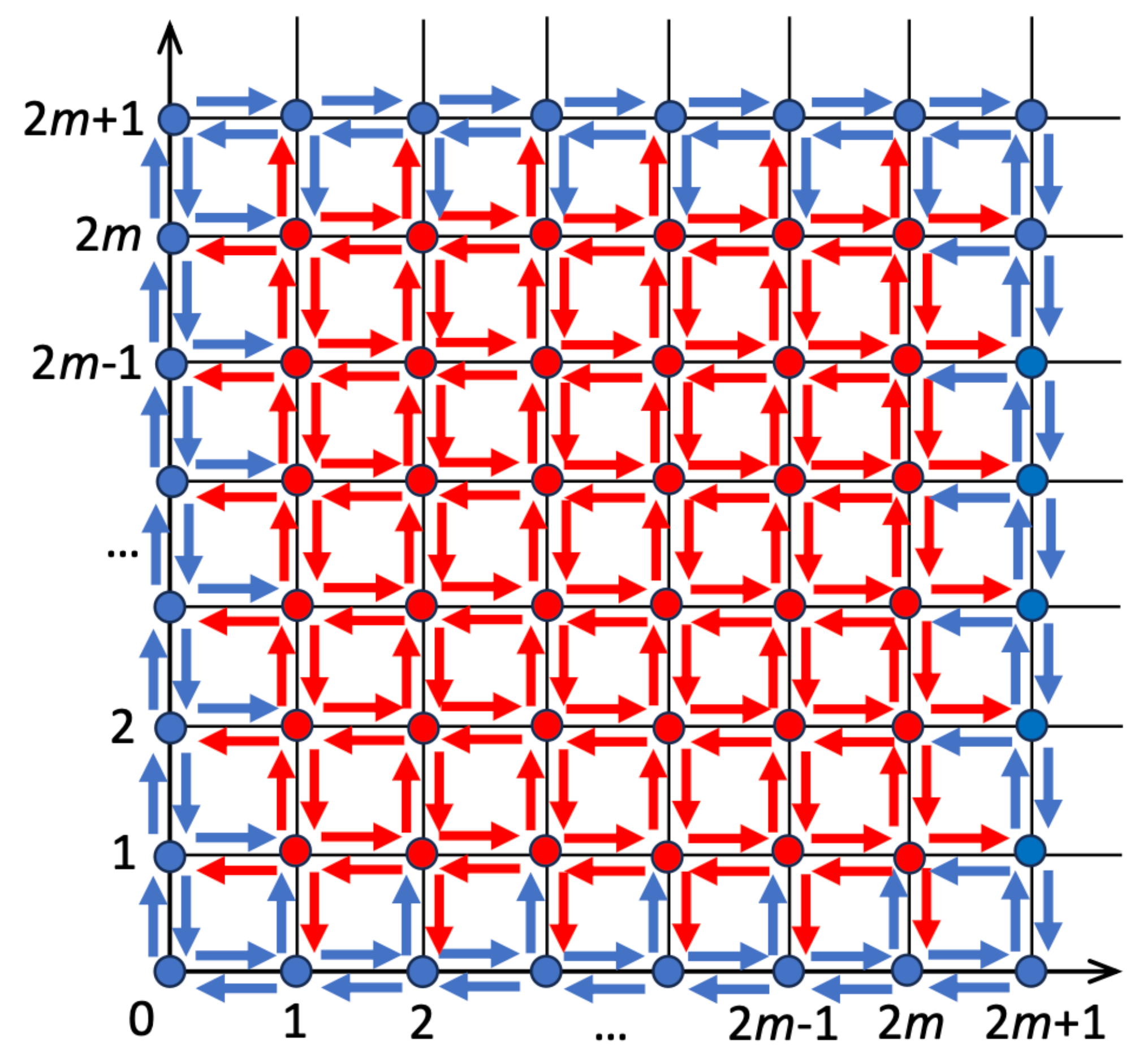}}}
\vskip -5.5cm
\leftline{(a) \hskip 5.18cm (b) \hskip 4.56cm (c)}
\vskip 5.1cm
\caption{(a) {\it The algebraic curve $h(x,y)=0$ given by~$(\ref{h_x_y_9})$ is plotted as the blue line, together with algebraic curves $h_i(x,y)=0$ given by $(\ref{h_i})$ for $\delta_i = 2$ (red line), $\delta_i=3$ (green line) and $\delta_i = 4$ (magenta line).} \hfill\break (b) {\it The negative coefficients of product $h_0$ given by~$(\ref{h_0})$ correspond to $($a subset of$)$ monomials that are represented by the red points, while the coefficients of the monomials that are represented by the blue points are all positive.} \hfill\break
(c) {\it A reversible chemical reaction network which can be modelled by the reaction rate equations written in the form of the {\rm ODE} system~$(\ref{reversible_system_perturbed_prod_hi_x})$--$(\ref{reversible_system_perturbed_prod_hi_y})$, which has $N$ algebraic limit cycles.}}
\label{figure4}
\end{figure}

We now explain why this system also has a weakly reversible realization given by the network in Figure~\ref{figure3}(c). Recall that we have observed above that the ODE system~(\ref{min_reversible_system_x})--(\ref{min_reversible_system_y}) has a realization that uses the reaction network in Figure~\ref{figure3}(c) with all reactions having reaction rate constants $\ge 1$. Note now that all the monomials in~(\ref{min_reversible_system_perturbed_x})--(\ref{min_reversible_system_perturbed_y}) that contain a factor of $\varepsilon$ already appear among the monomials of the ODE system~(\ref{min_reversible_system_x})--(\ref{min_reversible_system_y}), and if we choose $\varepsilon$ sufficiently small, the ODE system~(\ref{min_reversible_system_perturbed_x})--(\ref{min_reversible_system_perturbed_y}) can then be realized by the same reaction network as the ODE system~(\ref{min_reversible_system_x})--(\ref{min_reversible_system_y}). 
Therefore, the ODE system~(\ref{min_reversible_system_perturbed_x})--(\ref{min_reversible_system_perturbed_y}) can be realized by the reversible network shown in Figure~\ref{figure3}(c).

Our construction of a reversible chemical system~(\ref{min_reversible_system_perturbed_x})--(\ref{min_reversible_system_perturbed_y}) with an algebraic limit cycle can be generalized to obtain reversible systems with multiple limit cycles. To do that, we replace a single factor $h(x,y)$ by a product of several such factors. We construct reversible systems with several algebraic limit cycles in our next theorem.

\vskip 2mm

\begin{theorem}
There exists a reversible chemical system of order $4N+2$ that has $N$ algebraic limit cycles for all $N \in {\mathbb N}.$ In particular, we have $W^a(4N+2) \ge N$. 
\label{theoremWam}
\end{theorem}

\begin{proof}
We define
\begin{equation}
h_i(x,y) = x^2y^2+x^2y+xy^2+x^2+y^2+x+y+1-(8+\delta_i) \, x \, y \, ,
\qquad \mbox{for} \quad i=1,2,\dots, N \,,
\label{h_i}
\end{equation}
and for some mutually distinct real positive numbers $\delta_1,$ $\delta_2,$ $\dots,$ $\delta_N$. Then the equation $h_i(x,y) = 0$ can be rewritten as
$$
\left(x + 1 + \frac{1}{x} \right)
\left(y + 1 + \frac{1}{y} \right)
=
\delta_i + 9 \,,
$$
which has two solutions for 
$$
1 \, + \,
\frac{\delta_i
\,- \,\sqrt{\delta_i (12 + \delta_i)}}{6}
\,<\,
x
\,<\,
1 \, + \,
\frac{\delta_i
\, + \,\sqrt{\delta_i (12 + \delta_i)}}{6}
$$ 
giving a simple closed curve in the positive quadrant for each $\delta_i>0$. As an illustration, we visualize four such curves in Figure~\ref{figure4}(a) for
$\delta_i=1$, $\delta_i=2$, $\delta_i=3$ and $\delta_i=4.$ We note that $h_i$ given by~(\ref{h_i}) is equal to $h$ given by (\ref{h_x_y_9}) for $\delta_i=1$, which is plotted as the blue line in Figure~\ref{figure4}(a). In particular, we note that $h_i(x,y)=0$ in equation~(\ref{h_i}) for mutually distinct real positive numbers $\delta_1,$ $\delta_2,$ $\dots,$ $\delta_N$ give rise to $N$ disjoint algebraic curves in the positive quadrant. Now denote 
\begin{align}
    h_0 = \prod_{i=1}^N h_i 
\label{h_0}
\end{align}

\vskip -3.5mm

\noindent
and consider the system

\vskip -4.5mm

\begin{eqnarray}
\label{reversible_system_perturbed_prod_hi_x}
\frac{\mbox{d}x}{\mbox{d}t} 
&=&
h_0(x,y) \, \big(1-x+y-xy\big) 
\, - \, \varepsilon \, x \, y \, \frac{\partial h_0}{\partial y}\,,
\\ 
\label{reversible_system_perturbed_prod_hi_y}
\frac{\mbox{d}y}{\mbox{d}t}
&=& 
h_0(x,y) \, \big(1+x-y-xy\big)
\, + \, \varepsilon \, x \, y \, \frac{\partial h_0}{\partial x}\,.
\end{eqnarray}
The curves of the form $h_i(x,y) = 0$ lie along periodic trajectories of the system (\ref{reversible_system_perturbed_prod_hi_x})--(\ref{reversible_system_perturbed_prod_hi_y}), and a quick way to ensure that each one of the curves $h_i(x,y) = 0$ is actually a limit cycle is to check the transversality condition~(\ref{transcondition}) in Theorem~\ref{theorgen}. This can be done without additional calculations (by relying on the case of the ODE system~(\ref{min_reversible_system_perturbed_x})--(\ref{min_reversible_system_perturbed_y})) if we assume that we have chosen all the $\delta_i$ to be close enough to 1. 

The polynomial $h_0$ has degree $4N$. From the definition of $h_0$ we conclude that if a monomial $x^p y^q$ of $h_0$ has a {\em negative} coefficient, then we have $1 \le p \le 4N-1$ and $1 \le q \le 4N-1$. This situation is illustrated in Figure~\ref{figure4}(b), as follows: if the points in Figure~\ref{figure4}(b) represent monomials of $h_0$, then all the negative monomials are among the red points, and all the blue points correspond to {\em positive} monomials. Using the same argument as in Figures~\ref{figure3}(b) and~\ref{figure3}(c), we conclude that the ODE system~(\ref{reversible_system_perturbed_prod_hi_x})--(\ref{reversible_system_perturbed_prod_hi_y}) has a reversible realization which uses the reaction network illustrated in Figure~\ref{figure4}(c), and the ODE system~(\ref{reversible_system_perturbed_prod_hi_x})--(\ref{reversible_system_perturbed_prod_hi_y}) has at least $N$ algebraic limit cycles in the positive quadrant, given by the equations $h_i(x,y) = 0$. Therefore we have $W^a(4N+2) \ge N$. 
\end{proof}

\section{Robust limit cycles}

\label{sec5}

The reaction rate equations~(\ref{ex1x})--(\ref{ex1y}), (\ref{ex2x})--(\ref{ex2y}) and~(\ref{reversible_system_perturbed_prod_hi_x})--(\ref{reversible_system_perturbed_prod_hi_y}) can be written in the following general form

\vskip -6mm

\begin{eqnarray}
\frac{\mbox{d}x}{\mbox{d}t}
& = &
h(x,y) \, f_0(x,y) \, - \, \varepsilon \, x \, y \, \frac{\partial h}{\partial y}(x,y) \,, 
\label{odex}
\\
\frac{\mbox{d}y}{\mbox{d}t}
& = &
h(x,y) \, g_0(x,y) \, + \, \varepsilon \, x \, y \, \frac{\partial h}{\partial x}(x,y) \,,
\label{odey}
\end{eqnarray}
where $f_0(x,y),$ $g_0(x,y)$ and $h(x,y)$ are polynomials. For example, the ODE system~(\ref{ex1x})--(\ref{ex1y}) is given in the general form~(\ref{odex})--(\ref{odey}) for
\begin{equation}
h(x,y) = x^2 + x y^2 + y - 4xy,
\qquad
f_0(x,y) = 1-x
\qquad
\mbox{and}
\qquad
g_0(x,y) = x-y.
\label{ex43boros}
\end{equation}
Substituting $\varepsilon=0$, we get the ODE system~(\ref{WR_network_small_x})--(\ref{WR_network_small_y}), which can be realized as a chemical system and has a continuum of stable steady states, given by $h(x,y)=0$. We illustrate this in Figure~\ref{figure5}(a), where we plot the set $h(x,y)=0$ as the black dashed line together with fifteen illustrative trajectories starting at the boundary of the visualized domain $[0,4] \times [0,4]$ and three illustrative trajectories starting inside the oval $h(x,y)=0$. We observe that all calculated trajectories approach an equilibrium point inside the set $h(x,y)=0$ as $t \to \infty.$ 

In the proof of Lemma~\ref{lemmaWa1}, we have found a weakly reversible realization of the ODE system~(\ref{WR_network_small_perturbed_x})--(\ref{WR_network_small_perturbed_y}) for $\varepsilon \in (0, 1/6)$. However, the ODE system~(\ref{WR_network_small_perturbed_x})--(\ref{WR_network_small_perturbed_y}) 
can be realized as a chemical system for all $\varepsilon \ge 0.$ For example, if $\varepsilon=1$, it simplifies to
\begin{eqnarray}
\frac{\mbox{d}x}{\mbox{d}t}
& = &
x^2 + x y^2 + y - 6xy + 8x^2y - 3 x^2y^2 -x^3 ,
\label{odex43b}
\\
\frac{\mbox{d}y}{\mbox{d}t}
& = &
x^3 + x^2 y^2 + x y - 3x^2y - y^2.
\label{odey43b}
\end{eqnarray}
This ODE system has one stable limit cycle as illustrated in Figure~\ref{figure5}(b), where we calculate trajectories for the same initial conditions as in Figure~\ref{figure5}(a). We observe that all calculated trajectories approach the stable limit cycle $h(x,y)=0.$ The existence of a stable limit cycle can also be established for the general system~(\ref{odex})--(\ref{odey}). We formulate it as our next theorem.

\begin{figure}
\leftline{(a) \hskip 7.75cm (b)}
\centerline{\hskip 3mm \includegraphics[height=6.9cm]{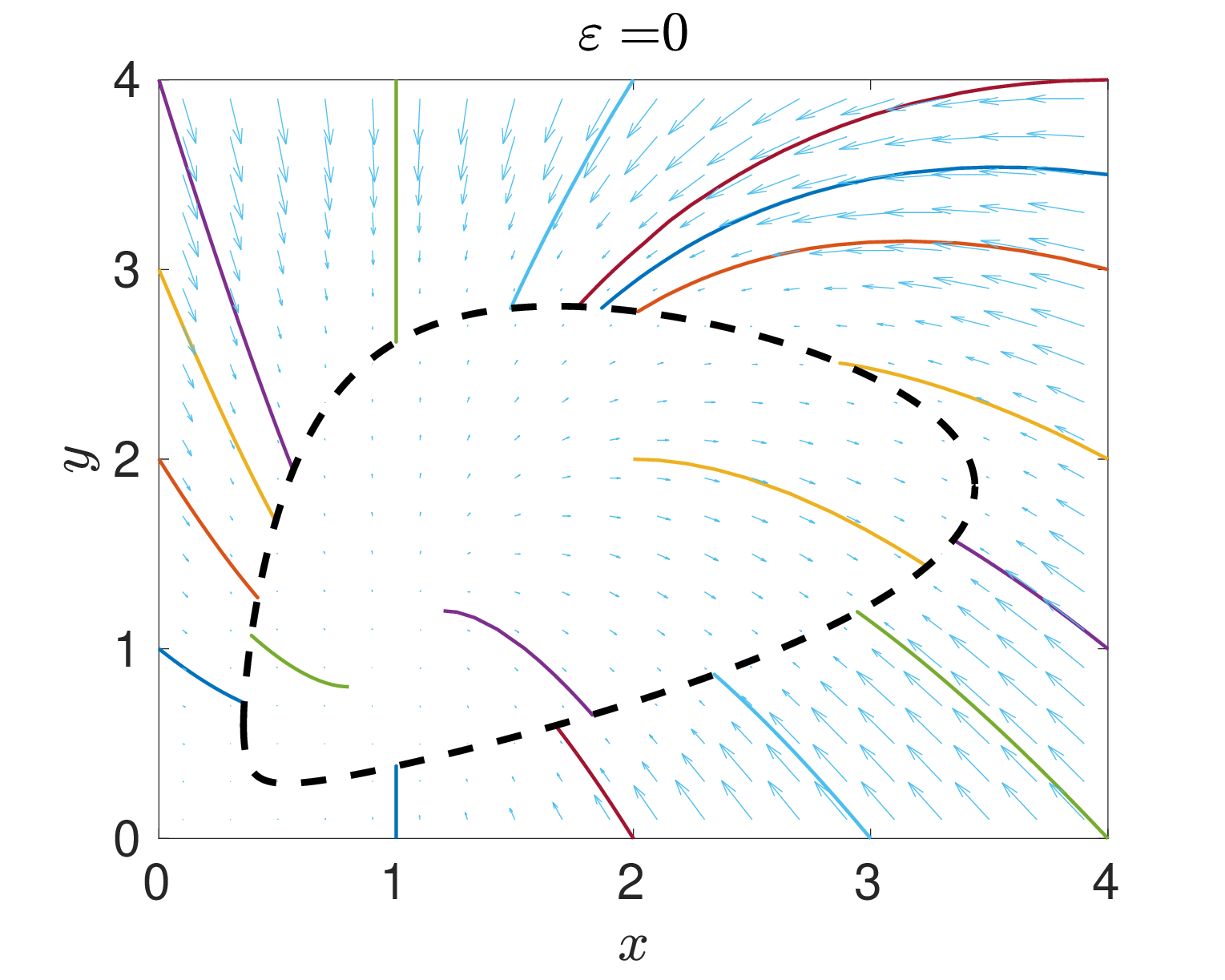}\hskip -3mm \includegraphics[height=6.9cm]{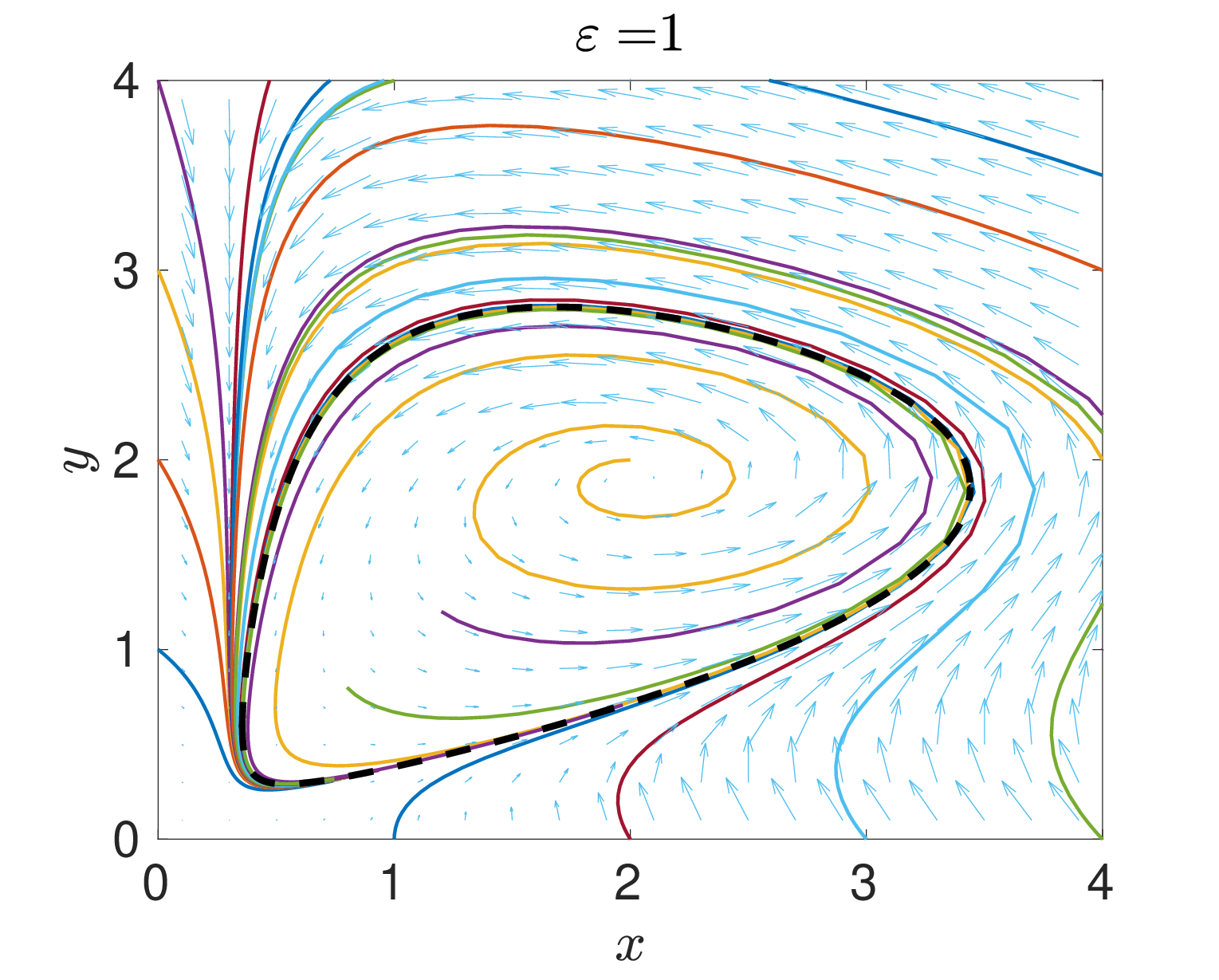}}
\caption{(a) {\it The phase plane of the {\rm ODE} system~$(\ref{WR_network_small_x})$--$(\ref{WR_network_small_y})$, i.e. the {\rm ODE} system~$(\ref{WR_network_small_perturbed_x})$--$(\ref{WR_network_small_perturbed_y})$ for $\varepsilon=0$. We plot
the algebraic curve $h(x,y)\equiv x^2 + x y^2 + y - 4xy = 0$ (black dashed line) together with some illustrative 
trajectories starting at the boundary of the visualized square. All trajectories converge to a stable steady state
on the curve $h(x,y)=0.$ } \hfill\break
(b) {\it The phase plane of the {\rm ODE} system~$(\ref{odex43b})$--$(\ref{odey43b})$, i.e. the {\rm ODE} system~$(\ref{WR_network_small_perturbed_x})$--$(\ref{WR_network_small_perturbed_y})$ for $\varepsilon=1$. The algebraic curve
$h(x,y)=0$ becomes a stable limit cycle for $\varepsilon>0$.}}
\label{figure5}
\end{figure}

\vskip 2mm

\begin{theorem}
Consider the {\rm ODE} system~$(\ref{odex})$--$(\ref{odey})$, where $f_0(x,y),$ $g_0(x,y)$ and $h(x,y)$ are real polynomials. Assume that the set where $h(x,y) = 0$ contains $N$ isolated simple closed curves, and also assume that the transversality condition
\begin{equation}
f_0(x,y) \, \frac{\partial h}{\partial x}(x,y)
\, + \,
g_0(x,y) \, \frac{\partial h}{\partial y}(x,y)
\, \neq \,
0
\label{transcondition}
\end{equation}
holds along all these $N$ isolated closed curves. 

Then any simple closed curve where $h(x,y) = 0$ and 
\begin{equation}
f_0(x,y) \, \frac{\partial h}{\partial x}(x,y)
\, + \,
g_0(x,y) \, \frac{\partial h}{\partial y}(x,y)
\, < \,
0
\label{transcondition_minus}
\end{equation}
is a stable algebraic limit cycle of the {\rm ODE} system~$(\ref{odex})$--$(\ref{odey})$, for all values of $\varepsilon > 0$. Similarly, any simple closed curve where $h(x,y) = 0$ and 
\begin{equation}
f_0(x,y) \, \frac{\partial h}{\partial x}(x,y)
\, + \,
g_0(x,y) \, \frac{\partial h}{\partial y}(x,y)
\, > \,
0
\label{transcondition_plus}
\end{equation}
is an unstable algebraic limit cycle of the {\rm ODE} system~$(\ref{odex})$--$(\ref{odey})$, for all values of $\varepsilon > 0$. \hfill\break In particular, the {\rm ODE} system~$(\ref{odex})$--$(\ref{odey})$ has $N$ algebraic limit cycles, for all values of $\varepsilon > 0$.
\label{theorgen}
\end{theorem}

\begin{proof}
The transversality condition~(\ref{transcondition}) implies that the gradient of $h$ does not vanish along the curve $h(x,y) = 0$. In particular, any simple closed curve where $h(x,y) = 0$ is a smooth curve. 
Note that the vector field $$
(-h_y,h_x) \equiv \left(- \frac{\partial h}{\partial y},
\frac{\partial h}{\partial x}
\right)
$$ 
always points {\em along} any curve of the form $h(x,y) = \alpha$, {\it i.e.} never points {\em across} it; the same is true of the vector field $\varepsilon {\hskip 0.2mm} x {\hskip 0.2mm} y {\hskip 0.2mm} (-h_y, h_x)$. Therefore, the dynamics of the system~(\ref{odex})--(\ref{odey}) across curves of the form $h(x,y) = \alpha$ is determined by the vector field $(f_0, g_0)$. 

Let us focus on the case where the  condition~(\ref{transcondition_minus}) is satisfied along one such curve $\mathcal C$. (The other case is completely analogous.) Then there exists an annular neighborhood of $\mathcal C$ denoted  $\mathcal A_{\mathcal C}(\delta)$, which is delimited by two curves where $h(x,y) = \pm \delta$, for some small number $\delta >0$, such that $\mathcal A_{\mathcal C}(\delta)$ is forward invariant for the system~(\ref{odex})--(\ref{odey}). To show this, we observe that the condition (\ref{transcondition_minus}) implies that, for $\delta$ small enough, the two boundary curves of $\mathcal A_{\mathcal C}(\delta)$ where $h(x,y) = \pm \delta$ are smooth and (\ref{transcondition_minus}) holds along them. Therefore, along the two boundary curves of $\mathcal A_{\mathcal C}(\delta)$ the vector field~(\ref{odex})--(\ref{odey}) points towards the interior of $\mathcal A_{\mathcal C}(\delta)$. 

Moreover, if we fix some $\delta_0 >0$ such that $A_{\mathcal C}(\delta)$ is forward invariant for the system~(\ref{odex})--(\ref{odey}) for all $\delta \in (0,\delta_0]$, then it follows that $\mathcal A_{\mathcal C}(\delta_0)$ cannot contain any periodic orbit other than $\mathcal C$, and cannot contain any fixed point. Therefore, all the forward trajectories that start within $\mathcal A_{\mathcal C}(\delta_0)$ must converge to $\mathcal C$, which implies that $\mathcal C$ is a stable limit cycle of the system~(\ref{odex})--(\ref{odey}).
\end{proof}

\smallskip

\noindent
The vector field $(f_0,g_0)$ given by~(\ref{ex43boros}) has a single critical point $(1,1)$ inside the oval $h(x,y)=0$, and the transversality condition~(\ref{transcondition}) is satisfied in our example in Figure~\ref{figure5}(b). Such an approach is also used in our proof of Theorem~\ref{theoremWam} to obtain the ODE system~(\ref{reversible_system_perturbed_prod_hi_x})--(\ref{reversible_system_perturbed_prod_hi_y}),
where we have
\begin{equation}
f_0(x,y)
\, = \, 1-x+y-xy \,
\qquad
\mbox{and}
\qquad
g_0(x,y)
\, = \,
1+x-y-xy.
\label{multiplecyclesf0g0}
\end{equation}
Then the vector field $(f_0,g_0)$ has one critical point at $(x,y)=(1,1)$, which is inside the ovals~(\ref{h_i}) for any $\delta_i > 0.$ Consider $h_0(x,y)$ in the ODE system~(\ref{reversible_system_perturbed_prod_hi_x})--(\ref{reversible_system_perturbed_prod_hi_y}) in the product form~(\ref{h_0}) where $N=4$, $\delta_1=1$, $\delta_2=2$, $\delta_3=3$ and $\delta_4=4.$ Such curves have been visualized 
in Figure~\ref{figure4}(a). They are algebraic limit cycles of the ODE system~(\ref{reversible_system_perturbed_prod_hi_x})--(\ref{reversible_system_perturbed_prod_hi_y}). In Figure~\ref{figure6}(a), we plot ten illustrative trajectories of the ODE system~(\ref{reversible_system_perturbed_prod_hi_x})--(\ref{reversible_system_perturbed_prod_hi_y}). We observe that the trajectories starting at the corners of our visualized domain $[0,3.5] \times [0,3.5]$ approach the outer limit cycle corresponding to $\delta_4=4,$ while trajectories starting inside this oval converge either to it, or to the limit cycle corresponding to $\delta_2=2$ or to a fixed point. The limit cycles corresponding to $\delta_2=2$ and $\delta_4=4$ are stable and they satisfy our transversality condition~(\ref{transcondition_minus}). This is also confirmed in Figure~\ref{figure6}(b), where we visualize the subdomains
\begin{eqnarray}
\Omega_s
&=&
\left\{ (x,y) \in [0,\infty)^2 \; \bigg|
\;
f_0(x,y) \, \frac{\partial h}{\partial x}(x,y)
\, + \,
g_0(x,y) \, \frac{\partial h}{\partial y}(x,y)
\, < \,
0
\right\}
\label{transsubdomains}
\\
\Omega_u
&=&
\left\{ (x,y) \in [0,\infty)^2 \; \bigg|
\;
f_0(x,y) \, \frac{\partial h}{\partial x}(x,y)
\, + \,
g_0(x,y) \, \frac{\partial h}{\partial y}(x,y)
\, > \,
0
\right\}
\rule{0pt}{7mm}
\label{transsubdomainu}
\end{eqnarray}
using magenta and white shading, respectively. The limit cycles corresponding to $\delta_2=2$ and $\delta_4=4$ are inside the domain~$\Omega_s$, while the limit cycles corresponding to $\delta_1=1$ and $\delta_3=3$ are inside the domain~$\Omega_u$ and they are unstable, satisfying the transversality condition~(\ref{transcondition_plus}). 

\begin{figure}
\leftline{(a) \hskip 7.6cm (b)}
\centerline{\hskip 3mm \includegraphics[height=6.7cm]{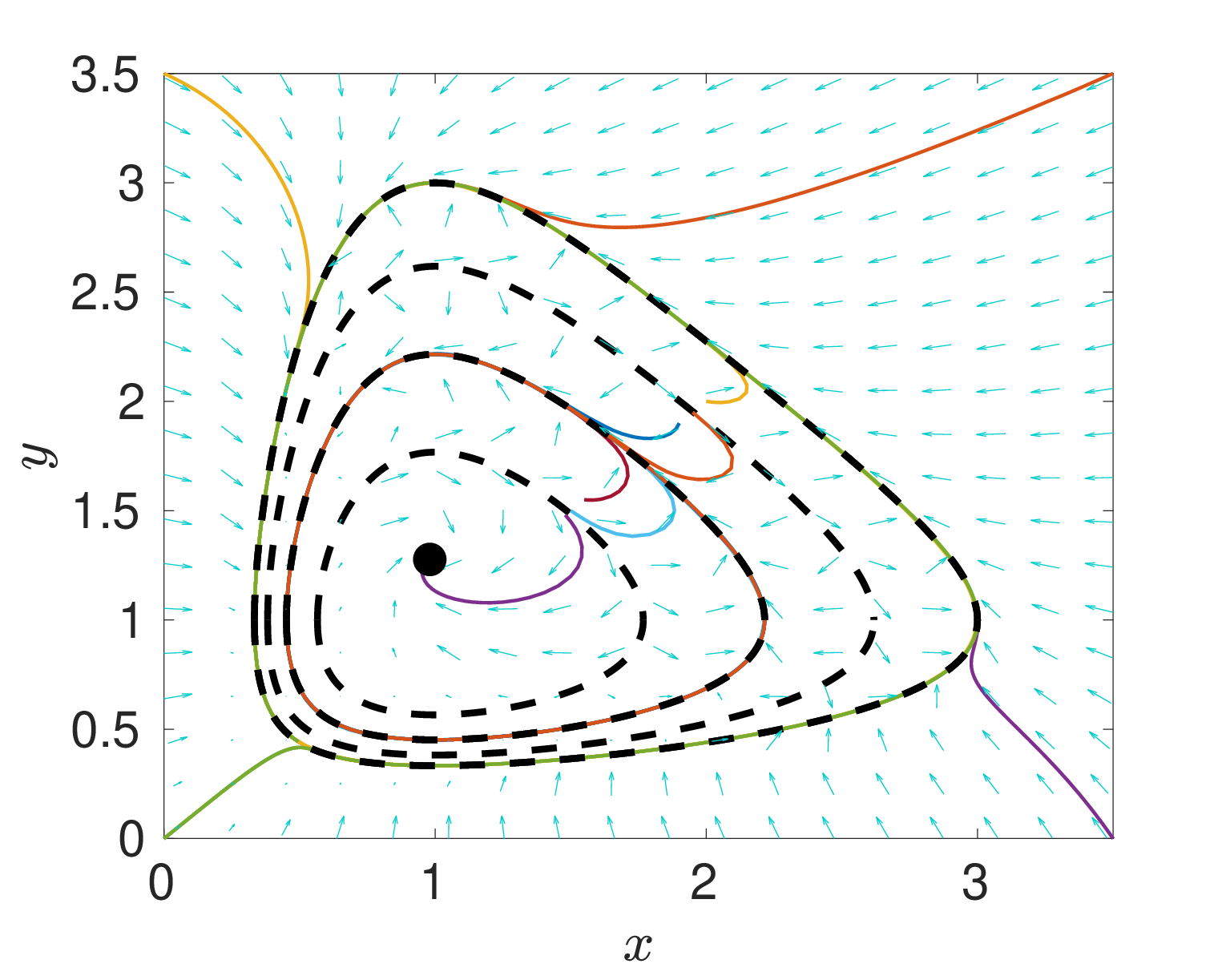}\hskip -3mm \includegraphics[height=6.7cm]{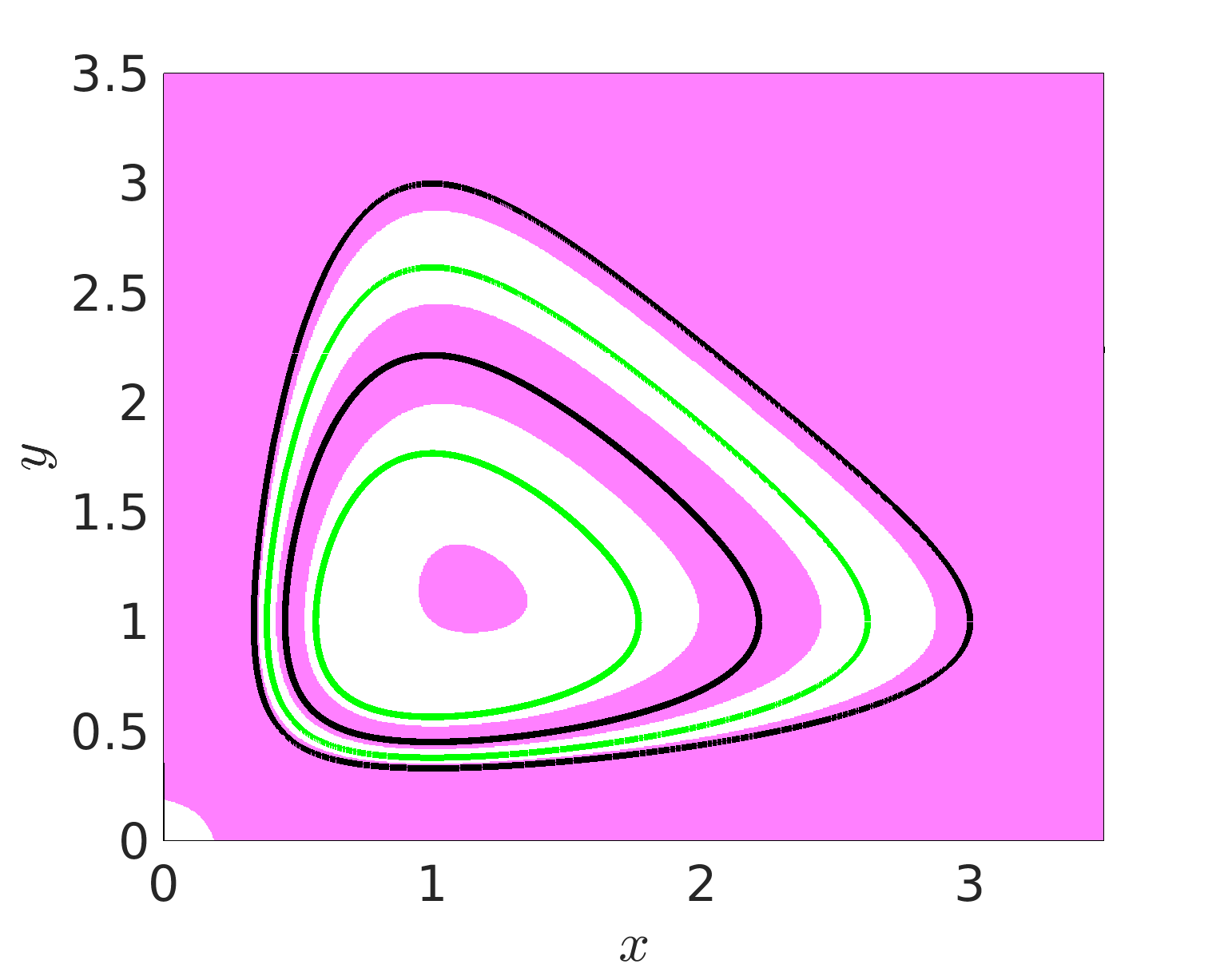}}
\caption{(a) {\it The phase plane of the {\rm ODE} system~$(\ref{reversible_system_perturbed_prod_hi_x})$--$(\ref{reversible_system_perturbed_prod_hi_y})$ for $\varepsilon=0.1$. We plot
the algebraic curve $h_0(x,y)$ (black dashed line) together with ten illustrative 
trajectories. } \hfill\break
(b) {\it The visualization of domains $\Omega_s$ $($magenta shading$)$ and $\Omega_u$ $($white shading$)$ given by~$(\ref{transsubdomains})$ and $(\ref{transsubdomainu})$, respectively. The stable algebraic limit cycles, corresponding to $\delta_2=2$ and $\delta_4=4$, are plotted as the black lines, while the unstable algebraic limit cycles, corresponding to $\delta_1=1$ and $\delta_3=3$, are plotted as the green lines.}}
\label{figure6}
\end{figure}

The systems with limit cycles which are used to achieve lower bounds in Tables~\ref{table1} and~\ref{table2} have been constructed using the standard definition of the limit cycle as an isolated closed trajectory. While such limit cycles can be stable, they are sometimes difficult to observe in numerical simulations. For example, consider the ODE system~(\ref{christopher1})--(\ref{christopher2}) which is a polynomial system of degree 4 with three hyperbolic algebraic limit cycles in the positive quadrant. The ODE system~(\ref{christopher1})--(\ref{christopher2}) shares some similarities with our general form~(\ref{odex})--(\ref{odey}) for $(f_0,g_0)=(1,1)$ if factor $\varepsilon {\hskip 0.2mm} x {\hskip 0.2mm} y$ is replaced by $7x-y.$ However, if we define the subdomains $\Omega_s$ and $\Omega_u$ by~(\ref{transsubdomains})--(\ref{transsubdomainu}) for $(f_0,g_0)=(1,1)$, then we observe that some parts of each limit cycle of the ODE system~(\ref{christopher1})--(\ref{christopher2})
are in $\Omega_s$ and some parts are in $\Omega_u.$
While the application of~\cite[Theorem 1]{christopher2001polynomial} can help us to deduce that each limit cycle is hyperbolic, the trajectories are attracted by parts of the limit cycle in $\Omega_s$ and repelled by parts of the limit cycle in $\Omega_u$. In particular, numerical errors can make it impossible to observe a trajectory which would for long positive times (resp. long negative times) approach the (theoretically) stable (resp. unstable) limit cycle in computational studies of such systems. However, if we do not attempt to minimize the degree of the polynomial on the right hand side of the ODE system~(\ref{odex})--(\ref{odey}), then it is possible to find $f_0$ and $g_0$ such that all limit cycles are fully in $\Omega_s$. We state this result as our next theorem.

\vskip 2mm

\begin{theorem}Let $h: {\mathbb R}^2 \to {\mathbb R}$ be a polynomial of degree $n_h$ and let the real algebraic curve $h(x,y)=0$ contain $N \in {\mathbb N}$ ovals in the $($strictly$)$ positive quadrant $(0,\infty) \times (0,\infty)$. Assume that 
\begin{equation}
\nabla h(x,y) 
=
\begin{pmatrix}
\displaystyle
\frac{\partial h}{\partial x}(x,y) \\
\displaystyle
\rule{0pt}{7mm} \frac{\partial h}{\partial y}(x,y)
\end{pmatrix}
\,\ne\,
\begin{pmatrix}
0 \\ 0    
\end{pmatrix}
\qquad
\mbox{for all $(x,y)$ satisfying $h(x,y)=0.$}
\label{nonzerograd}
\end{equation}
Then the {\rm ODE} system
\begin{eqnarray}
\frac{\mbox{{\rm d}}x}{\mbox{{\rm d}}t}
& = &
- \, x \, y \,
h(x,y) \, \frac{\partial h}{\partial x}(x,y) \, - \, \varepsilon \, x \, y \, \frac{\partial h}{\partial y}(x,y) \,, 
\label{odexchem}
\\
\frac{\mbox{{\rm d}}y}{\mbox{{\rm d}}t}
& = &
- \, x \, y \,
h(x,y) \, \frac{\partial h}{\partial y}(x,y) \, + \, \varepsilon \, x \, y \, \frac{\partial h}{\partial x}(x,y) \,,
\label{odeychem}
\end{eqnarray}
is a polynomial {\rm ODE} system of degree $n=2\,n_h+1$ which can be realized as a chemical reaction network under mass-action kinetics $($for any value of parameter $\varepsilon)$. The chemical system~$(\ref{odexchem})$--$(\ref{odeychem})$ has $N$~stable algebraic limit cycles contained in the components of the curve $h(x,y)=0$ for all $\varepsilon>0$ and the cofactor, defined by~$(\ref{alglimcyclcond})$, is equal to
$s(x,y) \,=\,
- \, x \, y 
\parallel \! \nabla h(x,y) \!\parallel^2.$ 
\label{theorem1}
\end{theorem}

\begin{proof}
Consider the ODE system~(\ref{odex})--(\ref{odey})
with
\begin{equation}
f_0(x,y) \, = \, - \, x \, y \, \frac{\partial h}{\partial x}(x,y)
\qquad \mbox{and} \qquad
g_0(x,y) \, = \, - \, x \, y \, \frac{\partial h}{\partial y}(x,y). 
\label{fgform}
\end{equation}
Then the ODE system~(\ref{odex})--(\ref{odey}) becomes the ODE system~(\ref{odexchem})--(\ref{odeychem}),
where the right-hand side contains polynomials of degree at most $2 n_h  + 1.$ Moreover, the assumption~(\ref{nonzerograd}) implies the transversality condition~(\ref{transcondition_minus}). Using Theorem~\ref{theorgen}, we conclude the existence of $N$ stable algebraic limit cycles contained in the components of the curve $h(x,y)=0$ for all $\varepsilon>0$. Differentiating $h(x,y)$ with respect of time, we obtain
\begin{equation*}
\frac{\mbox{d}}{\mbox{d}t}
h(x,y)
\,=\,
\frac{\partial h}{\partial x}(x,y) 
\,
\frac{\mbox{d}x}{\mbox{d}t}
+
\frac{\partial h}{\partial y}(x,y) 
\,
\frac{\mbox{d}y}{\mbox{d}t}
\,=\,
- \,
x \, y \, \left[
\left( \frac{\partial h}{\partial x}(x,y) \right)^{\!\!2}
+
\left(
\frac{\partial h}{\partial y}(x,y)  
\right)^{\!\!2}
\right]
h(x,y) 
\end{equation*}
which implies that the cofactor~(\ref{alglimcyclcond}) is
a polynomial of degree at most $2 n_h$ given by
$$
s(x,y)
\,=\,
- \, x \, y \,
\left( \frac{\partial h}{\partial x}(x,y) \right)^{\!\!2}
- \, x \, y \,
\left(
\frac{\partial h}{\partial y}(x,y)  
\right)^{\!\!2}
\,=\,
- \, x \, y 
\parallel \! \nabla h(x,y) \!\parallel^2.
$$
\end{proof}

\noindent
We note that the ODE system~(\ref{odexchem})--(\ref{odeychem}) can also be written in the matrix form as
\begin{equation}
\frac{\mbox{d}}{\mbox{d}t}
\begin{pmatrix}
x \\
y 
\end{pmatrix}
\, = \,
x \, y \,
\begin{pmatrix}
-h(x,y) & - \varepsilon \\
\varepsilon & -h(x,y)
\end{pmatrix}
\nabla h(x,y) \,.
\label{matrixform}
\end{equation}
This ODE system can be used to construct chemical systems with multiple stable algebraic limit cycles, provided that the ovals of $h(x,y)=0$ are contained in the (strictly) positive quadrant $(0,\infty)^2$, as we illustrate using examples with quartic planar curves ({\it i.e.} using $n_h=4$) in the next section.

\subsection{A chemical system with multiple stable algebraic limit cycles}

\label{sec61}

We consider quartic polynomial $q(x,y)$ in the following form
\begin{equation}
q(x,y) \,= \, 16 \, (x^4+y^4)
\, - \,
25 \, (x^2+y^2)
\, + \, 
\mu \, x^2 y^2 \, + \, 9 \,,
\label{mucurve4}
\end{equation}
where $\mu \in {\mathbb R}$ is a parameter. Since the degree of the polynomial~(\ref{mucurve4}) is $4$ for all $\mu \in {\mathbb R}$, Harnack's curve theorem implies that the maximum number of connected components of the algebraic curve $q(x,y)=0$ is $4.$ Depending on the value of parameter $\mu \in {\mathbb R},$ the algebraic curve $q(x,y)=0$ contains one, two or four ovals, as we show in our next lemma and illustrate in Figure~\ref{figure7}.

\vskip 2mm

\begin{lemma}
Let $\mu \in {\mathbb R}$ and let $q(x,y)$ be given by~$(\ref{mucurve4})$. Then we have:
{\par \parindent -5mm \leftskip 7mm
{\rm (i)} The set of solutions to equation $q(x,y)=0$ contains points
\begin{equation}
[-1,0], \quad
[-3/4,0], \quad
[3/4,0], \quad
[1,0], \quad
[0,-1], \quad
[0,-3/4], \quad
[0, 3/4], \quad
\mbox{and} \quad
[0,1].
\label{alwayspoints}
\end{equation}
Points~$(\ref{alwayspoints})$ are the only intersections of the algebraic curve $q(x,y)=0$ with $x$-axis and $y$-axis.
\par \parindent -6mm {\rm (ii)}  If $\mu \le -32,$ then the set of solutions to equation $q(x,y)=0$ contains one oval. 
\par \parindent -7mm {\rm (iii)} If $-32 < \mu < 337/9,$ then the set of solutions to equation $q(x,y)=0$ contains two ovals. 
\par \parindent -7mm {\rm (iv)} If $\mu = 32,$ then the algebraic curve $q(x,y)=0$ are two concentric circles with radii $3/4$ and $1$.
\par \parindent -6mm {\rm (v)} If $\mu=337/9$, then the set of solutions to equation $q(x,y)=0$ is connected and contains four ordinary double points $(${\hskip -0.4mm}crunodes$)$ at $[3/5,3/5]$, $[3/5,-3/5]$, $[-3/5,3/5]$ and $[-3/5,-3/5]$. 
\par \parindent -7mm {\rm (vi)} If $\mu > 337/9,$ then the set of solutions to equation $q(x,y)=0$ contains four ovals. \hfill\break 
In particular, $q(x,y)=0$ is an M-curve containing four connected components. 
\par}
\label{lemmaquartic}
\end{lemma}

\begin{proof}
(i) If $y=0$, then $q(x,y)=0$ simplifies to
$16 x^4
\, - \,
25 x^2
\, + \, 9
\, = \, 0,$
which is solved by $x=\pm 1$ and $x=\pm 3/4.$ Using symmetry, equation $q(x,y)=0$ is solved for $x=0$ by $y=\pm 1$ and $y=\pm 3/4.$ \par

\vskip 1mm

\noindent
(ii) Using~(\ref{mucurve4}), we have 
$$
q(x,x)
\,= \, (32+\mu) \, x^4
\, - \,
50 \, x^2
\, + \, 9 \,.
$$
If $\mu \le -32$, then equation $q(x,x)=0$ has exactly two real solutions and the algebraic curve $q(x,y)=0$ contains one oval. For example,  
if $\mu=-32$, then the two real solutions to $q(x,x)=0$ are given as $x=\pm 3/(5\sqrt{2}) \approx 0.424$ and the algebraic curve $q(x,y)=0$ contains one oval which goes clockwise through the points $[-3/4,0],$ $[-3/(5\sqrt{2}),3/(5\sqrt{2})],$ $[0, 3/4],$ $[3/(5\sqrt{2}),3/(5\sqrt{2})],$
$[3/4,0],$ $[3/(5\sqrt{2}),-3/(5\sqrt{2})],$ $[0,-3/4]$ and $[-3/(5\sqrt{2}),3/(5\sqrt{2})],
$ see Figure~\ref{figure7}(b).

\vskip 1mm

\noindent
(iii) If $-32 < \mu < 337/9,$ then there are four real solutions to $q(x,x)=0$ given by
$$
\pm \sqrt{\frac{25 + \sqrt{337 - 9 \, \mu}}{ 32+\mu}}
\quad
\mbox{and}
\quad
\pm \sqrt{\frac{25 - \sqrt{337 - 9 \, \mu}}{32+\mu}}
$$
and the set of solutions to equation $q(x,y)=0$ contains two concentric ovals, see Figures~\ref{figure7}(c), \ref{figure7}(d), \ref{figure7}(e) and~\ref{figure7}(f).

\vskip 1mm

\noindent
(iv) If $\mu=32$, then the formula~(\ref{mucurve4}) can be rewritten as
$$
q(x,y) \,= \, 16 \, (x^2+y^2)^2
\, - \,
25 \, (x^2+y^2)
\, + \, 
9
\,=\,
16 \, r^4 - \, 25 \, r^2 + \, 9 \,,
$$
where $r^2=x^2+y^2.$ Solving $q(x,y)=0$ for $r$, we obtain $r=1$ or $r=3/4,$ see Figure~\ref{figure7}(f).

\vskip 1mm

\noindent
(v) If $\mu=337/9$, then there are two solutions to $q(x,x)=0$ given as $\pm 3/5.$ They correspond to ordinary double points (crunodes) at $[3/5,3/5]$, $[3/5,-3/5]$, $[-3/5,3/5]$ and $[-3/5,-3/5]$, where the curves intersects itself so that two branches of the curve have distinct tangent lines, see Figure~\ref{figure7}(g).

\vskip 1mm

\noindent
(vi) If $\mu > 337/9,$ then there are no solutions to $q(x,x)=0$. In particular, we have four regions separated by lines $y=x$ and $y=-x$ each containing one oval, see Figures~\ref{figure7}(h) and~\ref{figure7}(i).
\end{proof}

\begin{figure}
\leftline{(a) \hskip 5.1cm (b) \hskip 5.1cm (c)}
\centerline{\hskip 3mm \includegraphics[height=4.1cm]{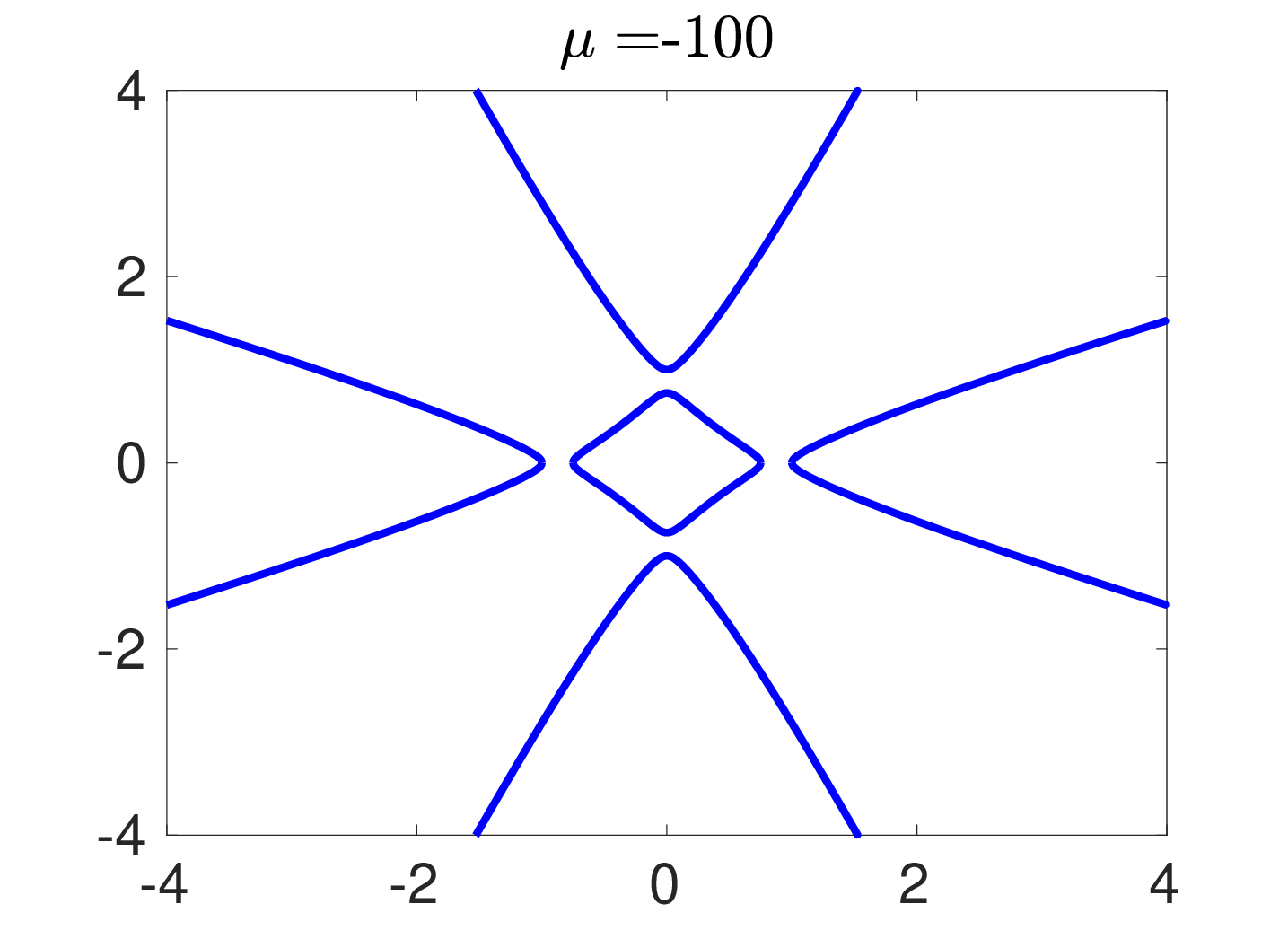}\hskip -1mm \includegraphics[height=4.1cm]{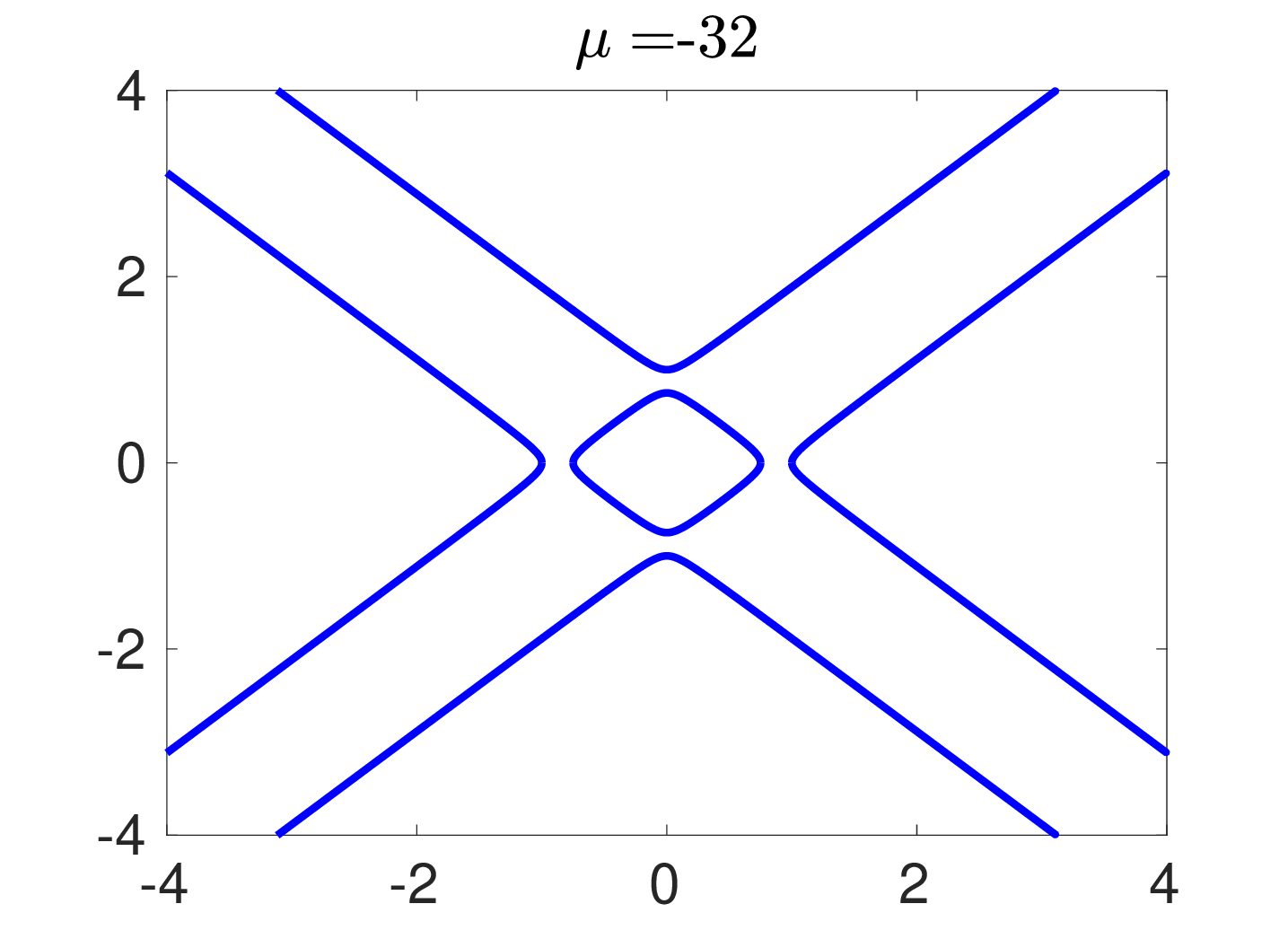} \hskip -1mm \includegraphics[height=4.1cm]{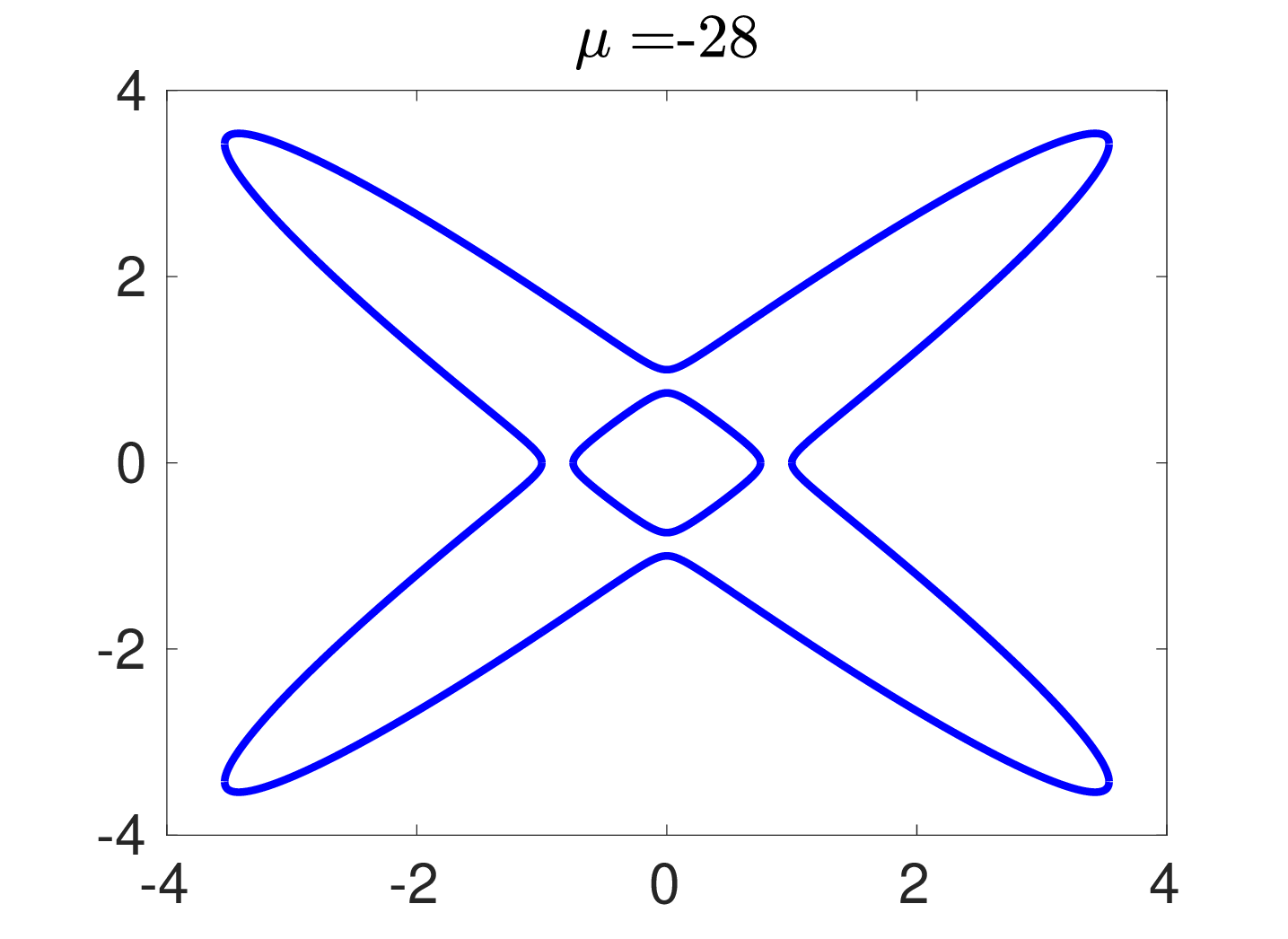}}
\leftline{(d) \hskip 5.1cm (e) \hskip 5.1cm (f)}
\centerline{\hskip 3mm \includegraphics[height=4.1cm]{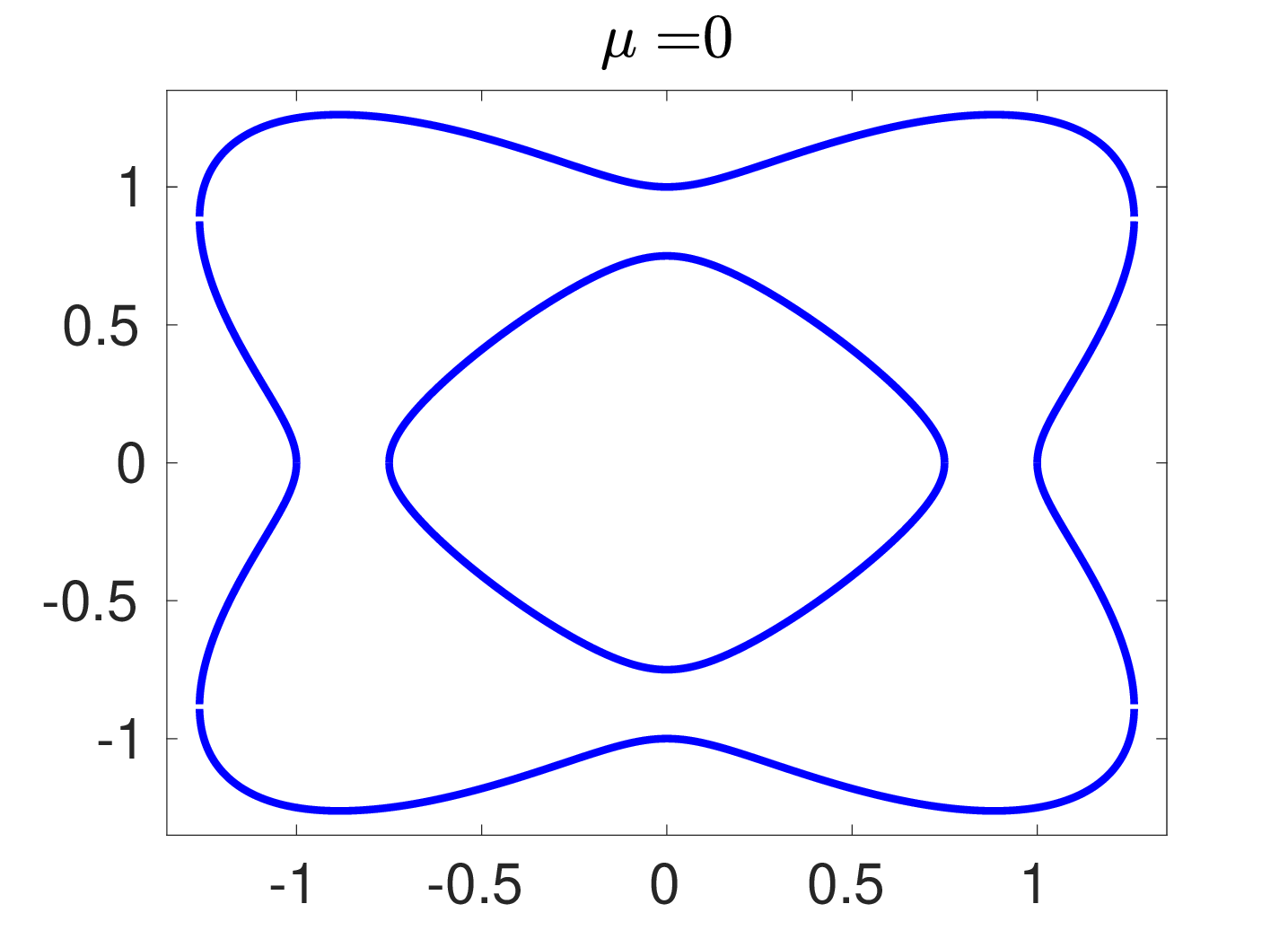}\hskip -1mm \includegraphics[height=4.1cm]{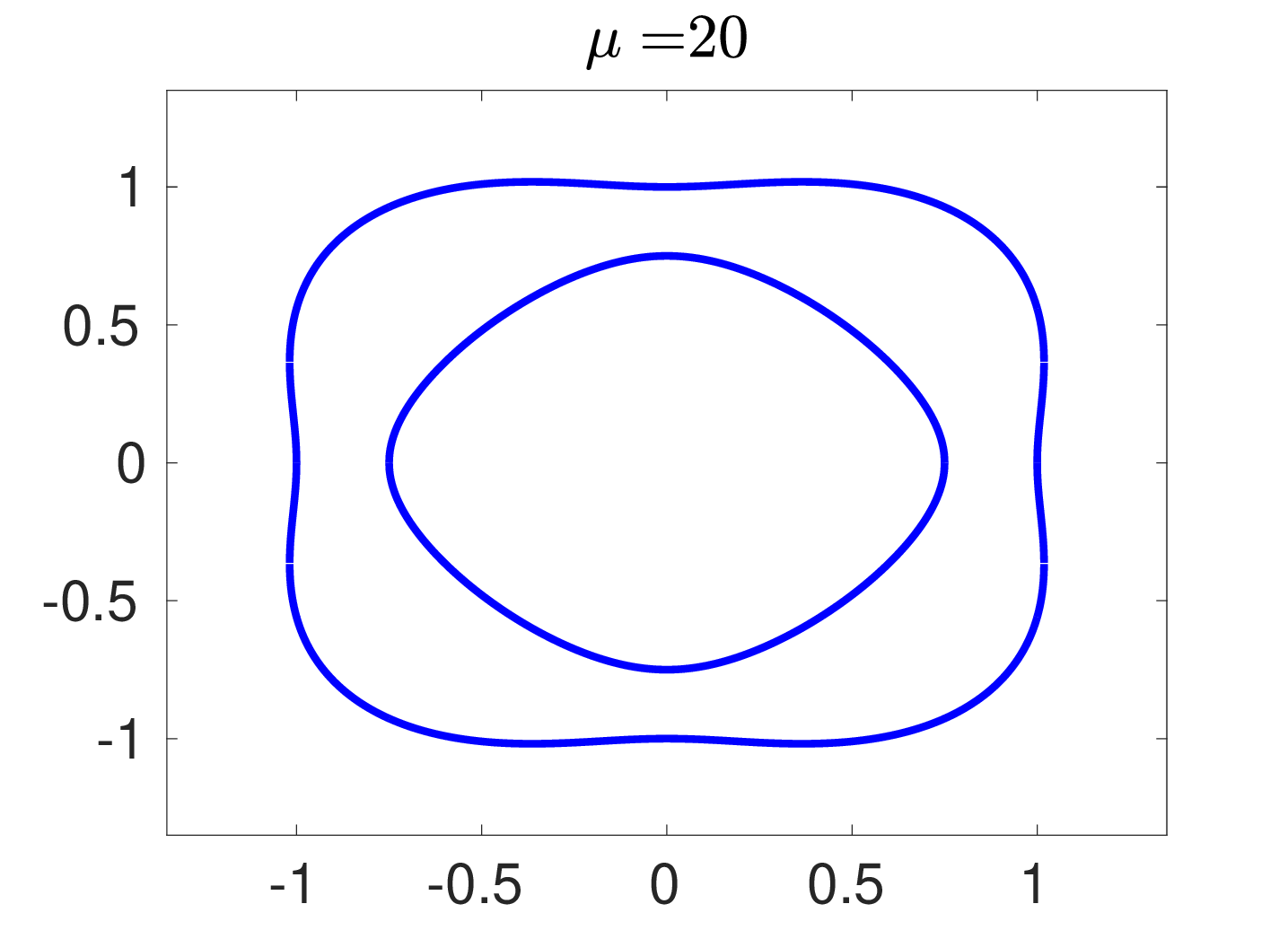} \hskip -1mm \includegraphics[height=4.1cm]{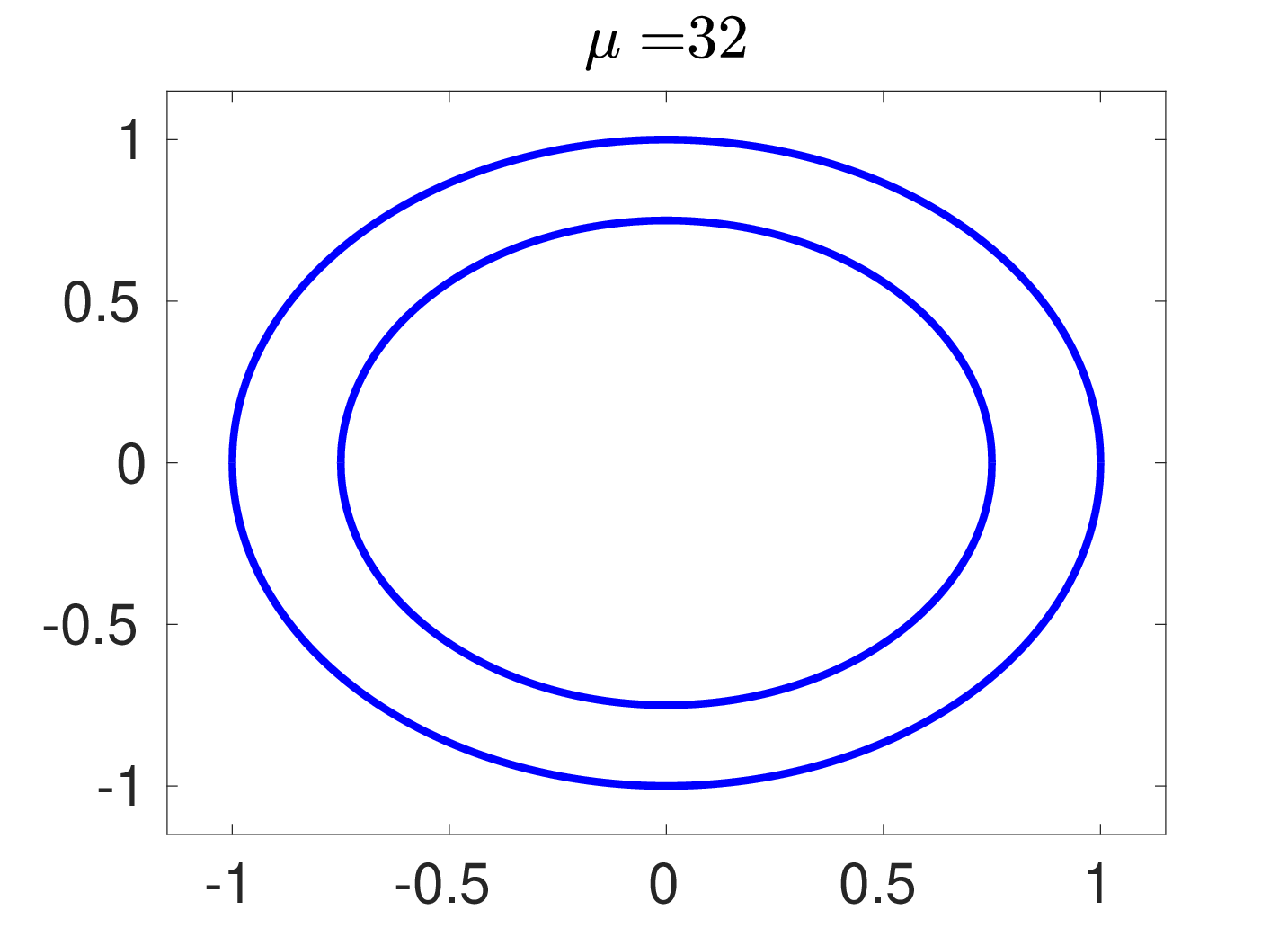}}
\leftline{(g) \hskip 5.1cm (h) \hskip 5.1cm (i)}
\centerline{\hskip 3mm \includegraphics[height=4.1cm]{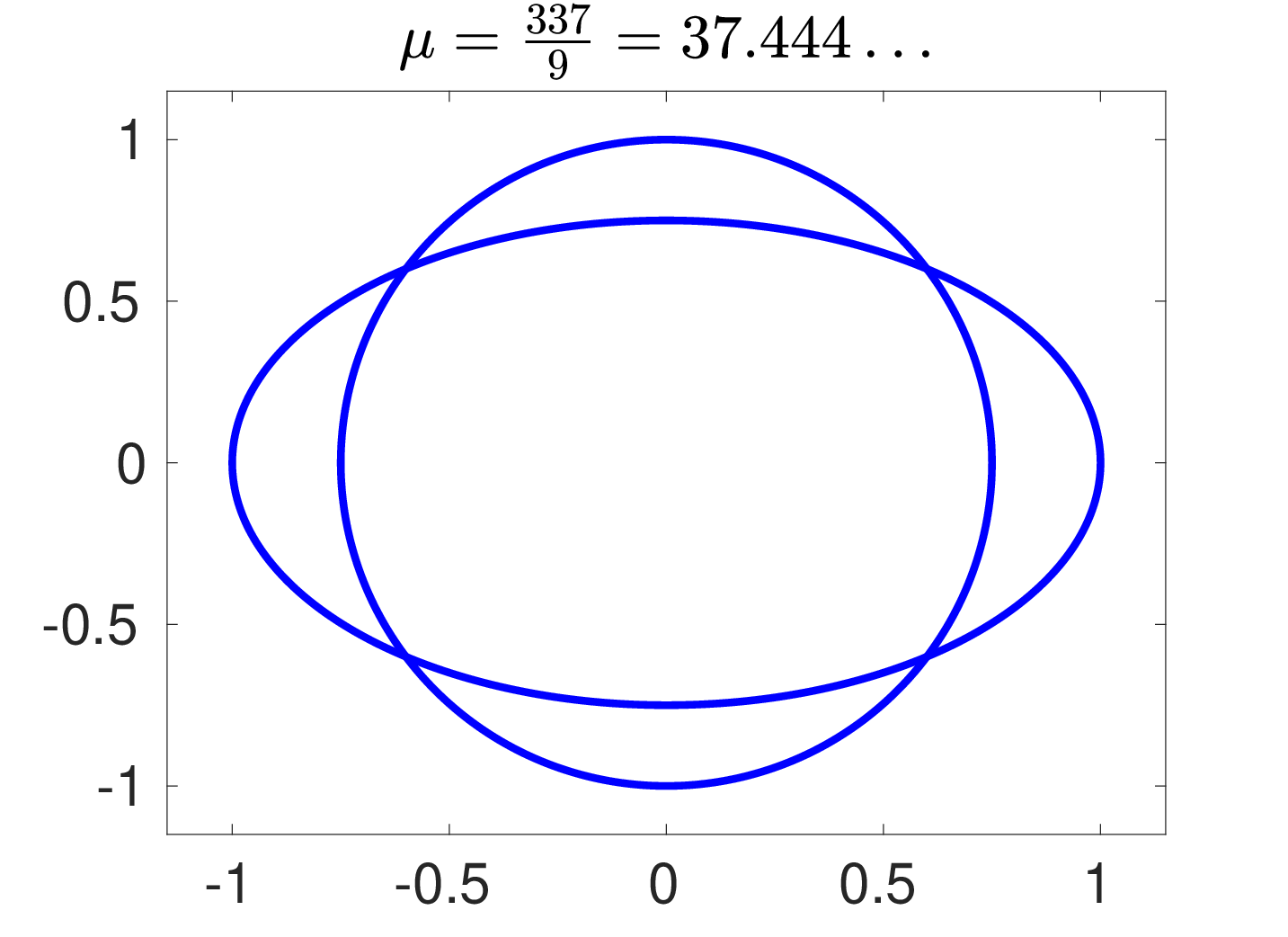}\hskip -1mm \includegraphics[height=4.1cm]{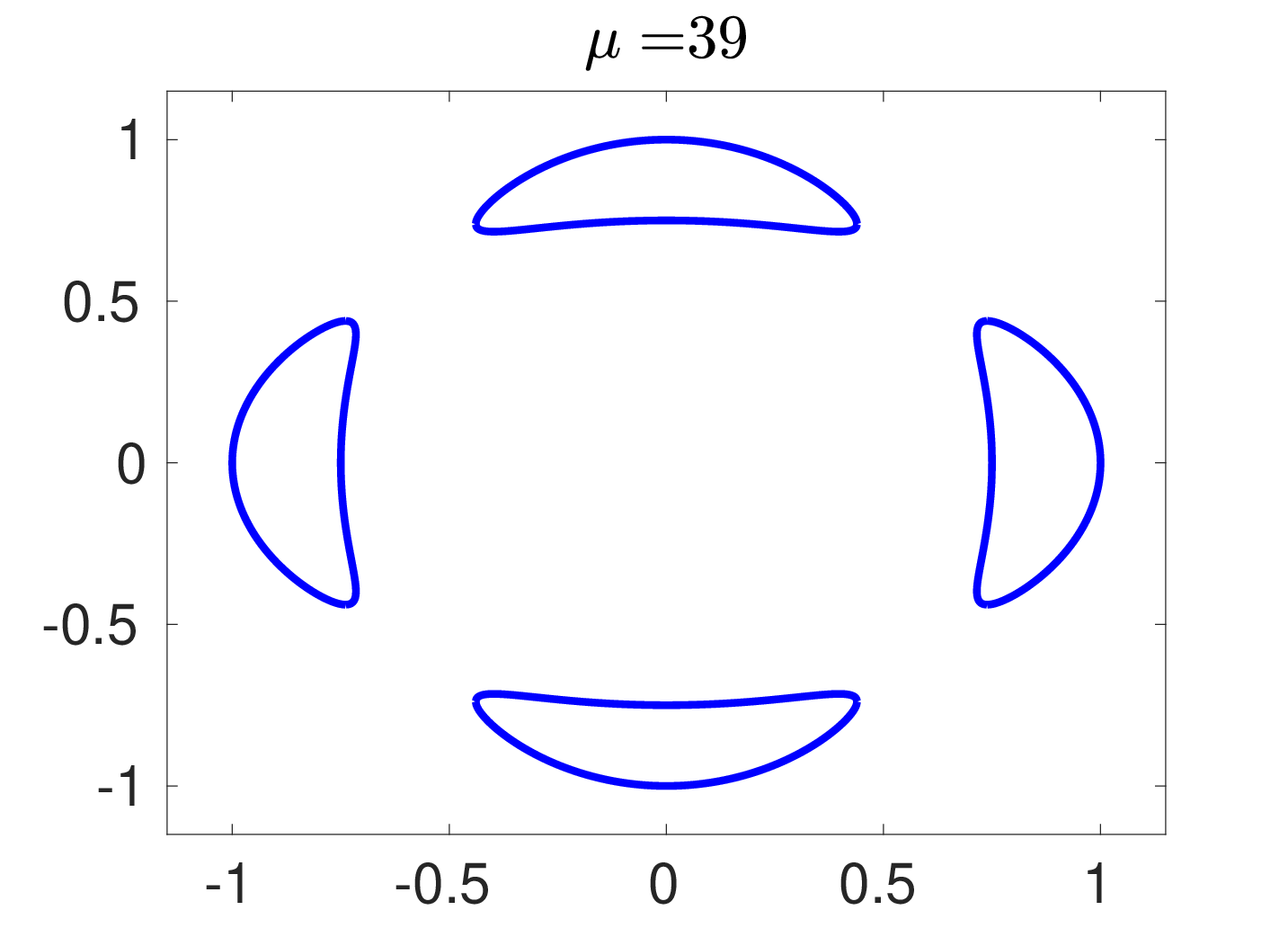} \hskip -1mm \includegraphics[height=4.1cm]{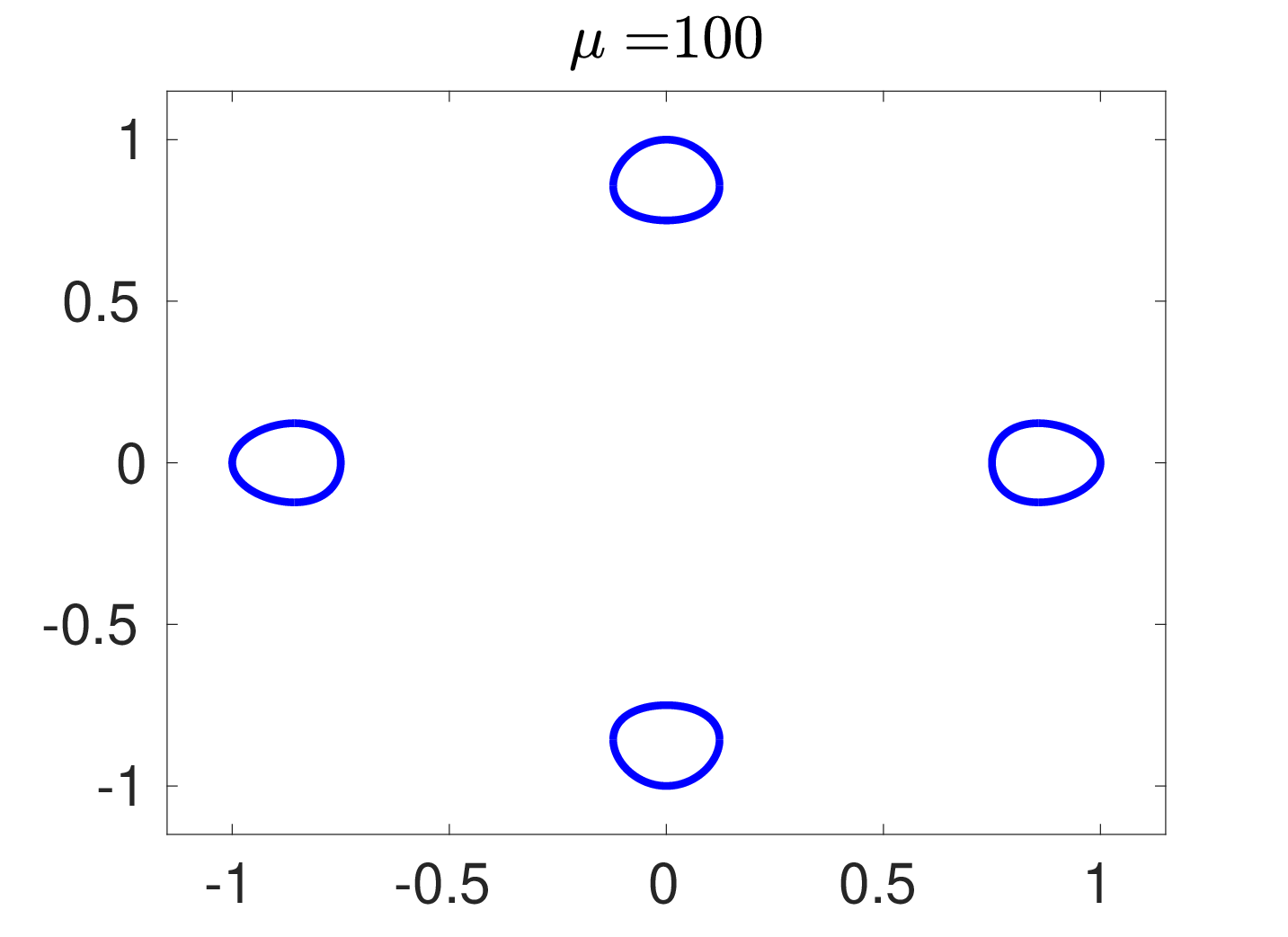}}
\caption{{\it The algebraic curve $q(x,y)=0$ given by~$(\ref{mucurve4})$ for} (a) $\mu=-100$; (b) $\mu=-32$; (c) $\mu=-28$; (d) $\mu=0$; (e) $\mu=20$; (f) $\mu=32$, (g) $\mu=337/9$; (h) $\mu=39$ and (i) $\mu=100.$}
\label{figure7}
\end{figure}

\noindent
The ovals of the algebraic curve $q(x,y)$ in Lemma~\ref{lemmaquartic} are outside of the positive quadrant. To apply Theorem~\ref{theorem1}, we first shift the curve $q(x,y)$ to get
\begin{equation}
h(x,y)
=
q(x-2,y-2),
\qquad
\mbox{where $q(x,y)$ is given by (\ref{mucurve4}).}
\label{hilexam}
\end{equation}
Then the ovals of $h(x,y)$ are in the positive quadrant for all nonnegative values of $\mu,$ see
Figure~\ref{figure7}. The phase plane of the ODE system~(\ref{odexchem})--(\ref{odeychem}) is plotted in Figure~\ref{figure8} for $\varepsilon=1$. We use two different values of $\mu$ corresponding to two ovals ($\mu=0$) and four ovals ($\mu=39$) of the algebraic curve $h(x,y)=0$ given by~(\ref{hilexam}). In both cases, we observe that all computed illustrative trajectories approach one of the ovals, confirming that Theorem~\ref{theorem1} leads to chemical systems with two (Figure~\ref{figure8}(a)) or four (Figure~\ref{figure8}(b)) stable limit cycles.

\begin{figure}
\leftline{(a) \hskip 7.6cm (b)}
\centerline{\hskip 3mm \includegraphics[height=6.7cm]{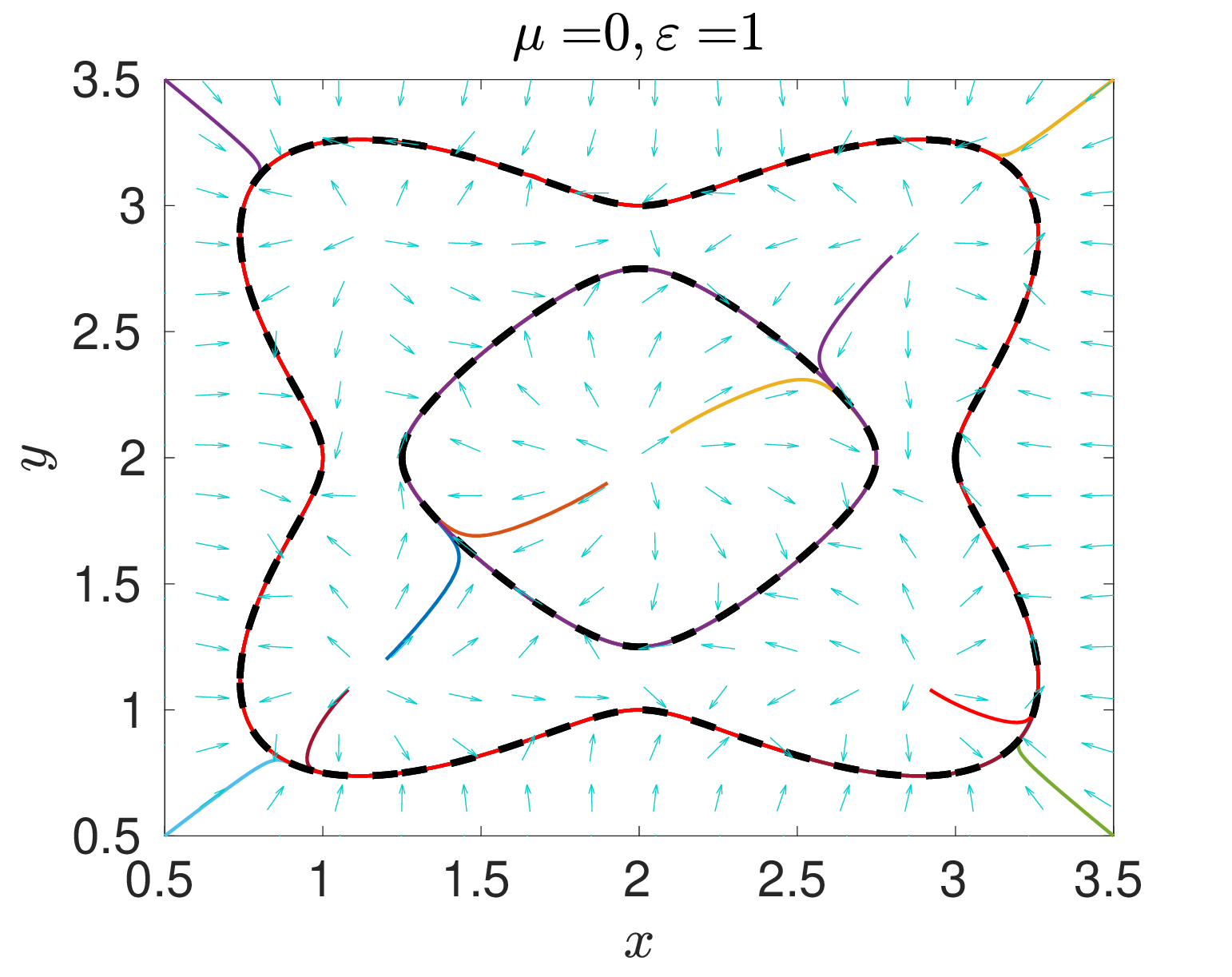}\hskip -3mm \includegraphics[height=6.7cm]{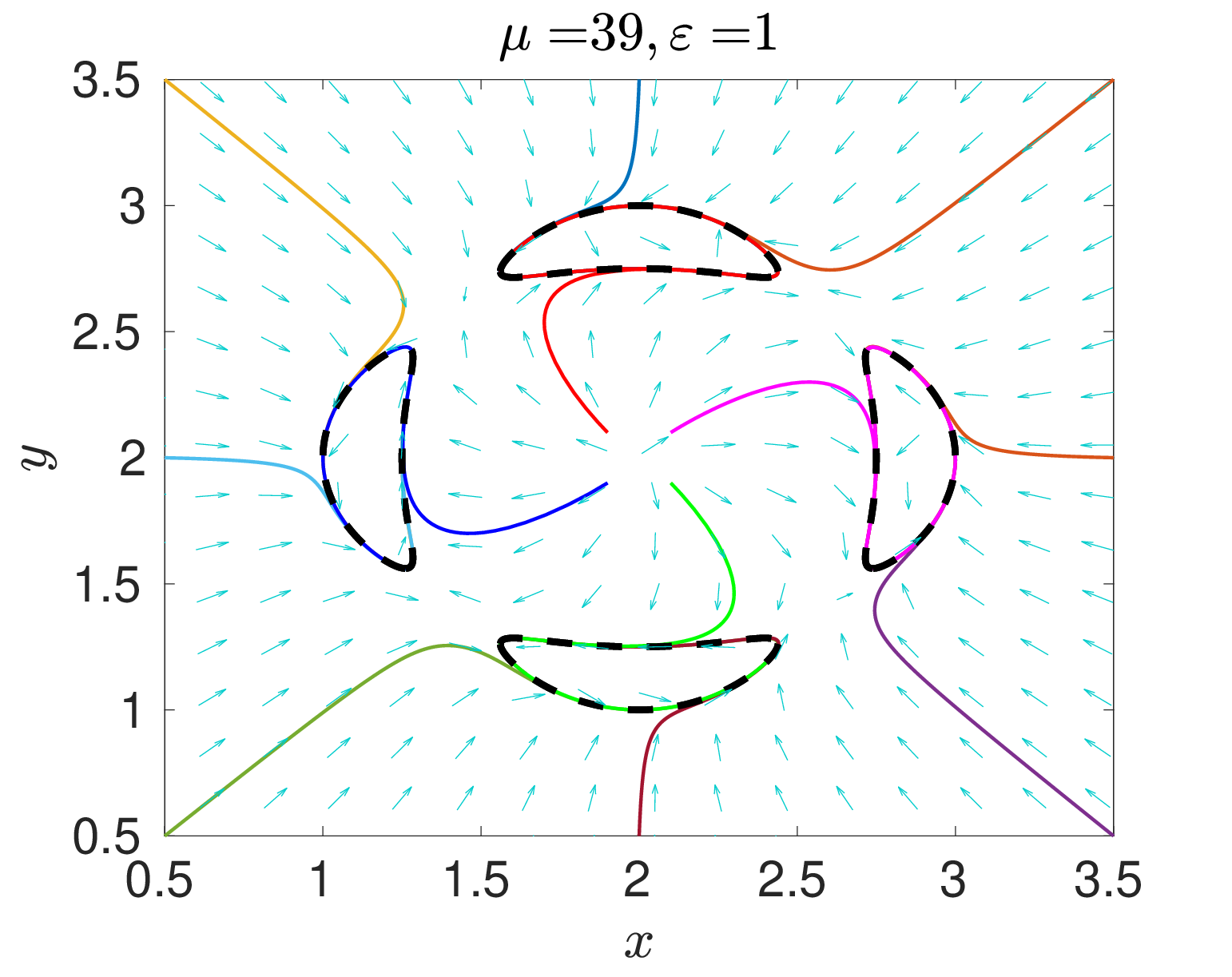}}
\caption{(a) {\it The phase plane of the {\rm ODE} system~$(\ref{odexchem})$--$(\ref{odeychem})$ with the polynomial $h(x,y)$ given by~$(\ref{hilexam})$ for parameters $\mu=0$ and $\varepsilon=1.$ We plot the two ovals of the algebraic curve $h(x,y)=0$ (black dashed line) together with ten illustrative trajectories showing convergence to one of the two ovals which are stable algebraic limit cycles of the {\rm ODE} system.} \hfill\break
(b) {\it The phase plane of the {\rm ODE} system~$(\ref{odexchem})$--$(\ref{odeychem})$ with the polynomial $h(x,y)$ given by~$(\ref{hilexam})$ 
for parameters $\mu=39$ and $\varepsilon=1,$ when the {\rm ODE} system has four stable algebraic limit cycles given as the four ovals of
the algebraic curve $h(x,y)=0$ $($visualized as the black dashed line together with illustrative trajectories converging to the limit cycles$)$.}}
\label{figure8}
\end{figure}

\section{Discussion}

\label{secdiscussion}

We have considered both algebraic and non-algebraic limit cycles in chemical reaction systems, with our results summarized together with the results in the literature in Tables~\ref{table1} and~\ref{table2}, respectively. To establish some lower bounds in Tables~\ref{table1} and~\ref{table2}, different techniques have to be applied. For example, small perturbations of an ODE system preserve the existence of a hyperbolic limit cycle, but an algebraic limit cycle can become non-algebraic after a perturbation. In particular, while the existence of a cubic weakly reversible chemical system with a limit cycle has been established in Table~\ref{table1}, it remains an open question whether a cubic weakly reversible system can have an algebraic limit cycle.

While a formulation of Hilbert's 16th problem restricted to algebraic limit cycles under generic conditions has been solved, see~\cite{Llibre:2010:OHP,gine2018cubic} for further discussion, these results are not considering the ODE systems which can be realized as chemical reaction networks. For example, a cubic system with two circular limit cycles is presented in~\cite{gine2018cubic}. Shifting the limit cycles to the positive quadrant, as we have done with our quartic example in equation~(\ref{hilexam}), and then multiplying the right-hand-side by $xy$ yields a fifth-order chemical system with two limit cycles. Other examples of cubic systems with 2 (non-generic) algebraic limit cycles appear in~\cite[Section 1]{Llibre:2010:OHP}, which could again be used to conclude that $S^a(5)\ge 2$. However, this does not improve the lower bound in Table~\ref{table2}, which implies $S^a(5) \ge S^a(4) \ge 3.$

Our investigation has focused on the ODE systems which can be realized as models of chemical reaction networks. However, such a realization is not unique: if an ODE system can be realized as the reaction rate equations of a chemical system, then there exists infinitely many chemical reaction networks corresponding to the same ODE system~\cite{Craciun_Pantea_2008,Plesa:2018:NCM,Craciun:2020:ECC}. For some studied ODE systems, we have been able to identify their realization as (weakly) reversible chemical reaction networks and this helped us to make conclusions on the values of numbers $W(n)$ and $W^a(n)$ (see, for example, our proof of Lemma~\ref{lemmaWa1}). In particular, chemical reaction networks (corresponding to the same ODE system) can be distinguished by having different structural properties. They can also be distinguished by considering their more detailed stochastic description~\cite{Enciso:2021:ISM}, written as the continuous-time discrete-space Markov chain and simulated by the Gillespie algorithm~\cite{Gillespie:1977:ESS,Erban:2020:SMR}. While the long-term dynamics of some chemical reaction networks can consist of a unique attractor of their ODE models, the long-term (stationary) probability distribution given by their stochastic model may display multiple maxima~\cite{Duncan:2015:NIM,Plesa:2019:NIM}. Considering chemical systems with limit cycles and oscillatory behaviour, stochastic models can bring additional possibilities for long-term dynamics including noise-induced oscillations~\cite{Muratov:2005:SSR,Erban:2009:ASC}. It may also happen that the ODE has a periodic solution and the long-term probability distribution is degenerate, converging to the state with zero molecules of all chemical species as time $t \to \infty$~\cite{Reddy:1975:ESS}.

The ODE system~(\ref{odexchem})--(\ref{odeychem}) or in its equivalent matrix form~(\ref{matrixform}) can be used to construct chemical reaction networks with stable algebraic limit cycles corresponding to the given algebraic curve $h(x,y)=0.$ In particular, if we want to construct a chemical system with more than one stable algebraic limit cycle, we can start with a quartic curve with more than one oval as shown in Section~\ref{sec61}. Another quartic curve with two ovals can be obtained as a product of two circles (quadratic curves). Such a product form construction has been used in our proof of Theorem~\ref{theoremWam}, see equation~(\ref{h_0}). Considering the product of two circles and using Theorem~\ref{theorem1}, we can obtain a chemical system which has the two circles as its two stable algebraic limit cycles. In Section~\ref{sec61}, we have considered quartic curve~(\ref{hilexam}) which had four closed connected components for $\mu>337/9$ and Theorem~\ref{theorem1} implied a chemical system with four stable algebraic limit cycles. To construct chemical systems with more stable limit cycles than four, we can apply Theorem~\ref{theorem1} to algebraic curves $h(x,y)=0$ of degree $n_h>4,$ which has the corresponding number of ovals. One possible way to find such algebraic curves is to construct them in the product form~(\ref{h_0}).

In this paper, we have considered chemical reaction systems with two chemical species $X$ and $Y$ which are described by planar ODE system~(\ref{odex_general})--(\ref{odey_general}). In particular, we could make connections to the results and open problems on limit cycles and periodic solutions in planar polynomial ODE systems, with attention to the results for systems with polynomials of low degree $n$ on the right hand side~\cite{Shi:1980:CEE,Li:2009:CST}. Our low degree $n$ investigation is also interesting from the applications point of view, because it decreases the order~(\ref{deforder}) of the chemical reactions when the ODE system~(\ref{odex_general})--(\ref{odey_general}) is realized as the chemical system. In particular, we have addressed some questions on `minimal' reaction systems with certain dynamics by minimizing the value of~$n$. The minimal reaction systems with oscillations can also be defined in terms of the minimal number $m$ of reactions in the chemical reaction network~(\ref{crn}), see~\cite{Banaji:2024:OTR} for some systems with two chemical species. In some applications, it is necessary to study chemical reaction systems with more than two chemical species, leading to three-dimensional or higher-dimensional ODE systems. For example, limit cycles in reaction networks with three or four chemical species are investigated under additional structural conditions on the reaction network in~\cite{Boros:2022:LCM,Boros:2023:SMB}. Multiple limit cycles for systems of two chemical species have also been reported in~\cite{Boros:2024:OPD} for the case when deficiency of the chemical reaction network is one, while it is well known that the deficiency-zero networks cannot have periodic solutions in the positive quadrant~\cite{Feinberg72}.

\vskip 15mm

\noindent
{\bf Funding.}
This work was supported by the Engineering and Physical Sciences Research Council, grant
number EP/V047469/1, awarded to Radek Erban, and by the National Science Foundation grant DMS-2051568, awarded to Gheorghe Craciun. Craciun also acknowledges support as a Visiting Scholar at Merton College Oxford. 

\newpage

\bibsep=0.1em

\begin{thebibliography}{50}
\ifx \bisbn   \undefined \def \bisbn  #1{ISBN #1}\fi
\ifx \binits  \undefined \def \binits#1{#1}\fi
\ifx \bauthor  \undefined \def \bauthor#1{#1}\fi
\ifx \batitle  \undefined \def \batitle#1{#1}\fi
\ifx \bjtitle  \undefined \def \bjtitle#1{#1}\fi
\ifx \bvolume  \undefined \def \bvolume#1{\textbf{#1}}\fi
\ifx \byear  \undefined \def \byear#1{#1}\fi
\ifx \bissue  \undefined \def \bissue#1{#1}\fi
\ifx \bfpage  \undefined \def \bfpage#1{#1}\fi
\ifx \blpage  \undefined \def \blpage #1{#1}\fi
\ifx \burl  \undefined \def \burl#1{\textsf{#1}}\fi
\ifx \doiurl  \undefined \def \doiurl#1{\url{https://doi.org/#1}}\fi
\ifx \betal  \undefined \def \betal{\textit{et al.}}\fi
\ifx \binstitute  \undefined \def \binstitute#1{#1}\fi
\ifx \binstitutionaled  \undefined \def \binstitutionaled#1{#1}\fi
\ifx \bctitle  \undefined \def \bctitle#1{#1}\fi
\ifx \beditor  \undefined \def \beditor#1{#1}\fi
\ifx \bpublisher  \undefined \def \bpublisher#1{#1}\fi
\ifx \bbtitle  \undefined \def \bbtitle#1{#1}\fi
\ifx \bedition  \undefined \def \bedition#1{#1}\fi
\ifx \bseriesno  \undefined \def \bseriesno#1{#1}\fi
\ifx \blocation  \undefined \def \blocation#1{#1}\fi
\ifx \bsertitle  \undefined \def \bsertitle#1{#1}\fi
\ifx \bsnm \undefined \def \bsnm#1{#1}\fi
\ifx \bsuffix \undefined \def \bsuffix#1{#1}\fi
\ifx \bparticle \undefined \def \bparticle#1{#1}\fi
\ifx \barticle \undefined \def \barticle#1{#1}\fi
\bibcommenthead
\ifx \bconfdate \undefined \def \bconfdate #1{#1}\fi
\ifx \botherref \undefined \def \botherref #1{#1}\fi
\ifx \url \undefined \def \url#1{\textsf{#1}}\fi
\ifx \bchapter \undefined \def \bchapter#1{#1}\fi
\ifx \bbook \undefined \def \bbook#1{#1}\fi
\ifx \bcomment \undefined \def \bcomment#1{#1}\fi
\ifx \oauthor \undefined \def \oauthor#1{#1}\fi
\ifx \citeauthoryear \undefined \def \citeauthoryear#1{#1}\fi
\ifx \endbibitem  \undefined \def \endbibitem {}\fi
\ifx \bconflocation  \undefined \def \bconflocation#1{#1}\fi
\ifx \arxivurl  \undefined \def \arxivurl#1{\textsf{#1}}\fi
\csname PreBibitemsHook\endcsname

\bibitem[\protect\citeauthoryear{Feinberg}{2019}]{Feinberg:2019:FCR}
\begin{bbook}
\bauthor{\bsnm{Feinberg}, \binits{M.}}:
\bbtitle{Foundations of Chemical Reaction Network Theory}.
\bpublisher{Springer},
\blocation{New York}
(\byear{2019})
\end{bbook}
\endbibitem

\bibitem[\protect\citeauthoryear{Angeli}{2009}]{Angeli:2009:TCR}
\begin{barticle}
\bauthor{\bsnm{Angeli}, \binits{D.}}:
\batitle{A tutorial on chemical reaction network dynamics}.
\bjtitle{European Journal of Control}
\bvolume{15},
\bfpage{398}--\blpage{406}
(\byear{2009})
\end{barticle}
\endbibitem

\bibitem[\protect\citeauthoryear{Bendixson}{1901}]{Bendixson:1901:SCD}
\begin{barticle}
\bauthor{\bsnm{Bendixson}, \binits{I.}}:
\batitle{{Sur les courbes d\'efinies par des \'equations diff\'erentielles}}.
\bjtitle{Acta Mathematica}
\bvolume{24},
\bfpage{1}--\blpage{88}
(\byear{1901})
\end{barticle}
\endbibitem

\bibitem[\protect\citeauthoryear{Christopher and
  Li}{2007}]{Christopher:2007:LCD}
\begin{bbook}
\bauthor{\bsnm{Christopher}, \binits{C.}},
\bauthor{\bsnm{Li}, \binits{C.}}:
\bbtitle{Limit Cycles of Differential Equations}.
\bpublisher{Birkh\"auser},
\blocation{Basel}
(\byear{2007})
\end{bbook}
\endbibitem

\bibitem[\protect\citeauthoryear{P\'ota}{1983}]{Pota:1983:TBS}
\begin{barticle}
\bauthor{\bsnm{P\'ota}, \binits{G.}}:
\batitle{Two-component bimolecular systems cannot have limit cycles: A complete
  proof}.
\bjtitle{Journal of Chemical Physics}
\bvolume{78},
\bfpage{1621}--\blpage{1622}
(\byear{1983})
\end{barticle}
\endbibitem

\bibitem[\protect\citeauthoryear{Schuman and T\'oth}{2003}]{Schuman:2003:LCT}
\begin{barticle}
\bauthor{\bsnm{Schuman}, \binits{B.}},
\bauthor{\bsnm{T\'oth}, \binits{J.}}:
\batitle{No limit cycle in two species second order kinetics}.
\bjtitle{Bulletin des Sciences Mathematiques}
\bvolume{127},
\bfpage{222}--\blpage{230}
(\byear{2003})
\end{barticle}
\endbibitem

\bibitem[\protect\citeauthoryear{Gasull and Giacomini}{2023}]{Gasull:2023:NLC}
\begin{barticle}
\bauthor{\bsnm{Gasull}, \binits{A.}},
\bauthor{\bsnm{Giacomini}, \binits{H.}}:
\batitle{Number of limit cycles for planar systems with invariant algebraic
  curves}.
\bjtitle{Qualitative Theory of Dynamical Systems}
\bvolume{22},
\bfpage{44}
(\byear{2023})
\end{barticle}
\endbibitem

\bibitem[\protect\citeauthoryear{Shi}{1980}]{Shi:1980:CEE}
\begin{barticle}
\bauthor{\bsnm{Shi}, \binits{S.}}:
\batitle{A concrete example of the existence of four limit cycles for plane
  quadratic systems}.
\bjtitle{Scientia Sinica}
\bvolume{23}(\bissue{2}),
\bfpage{153}--\blpage{158}
(\byear{1980})
\end{barticle}
\endbibitem

\bibitem[\protect\citeauthoryear{Li et~al.}{2009}]{Li:2009:CST}
\begin{barticle}
\bauthor{\bsnm{Li}, \binits{C.}},
\bauthor{\bsnm{Liu}, \binits{C.}},
\bauthor{\bsnm{Yang}, \binits{J.}}:
\batitle{A cubic system with thirteen limit cycles}.
\bjtitle{Journal of Differential Equations}
\bvolume{246},
\bfpage{3609}--\blpage{3619}
(\byear{2009})
\end{barticle}
\endbibitem

\bibitem[\protect\citeauthoryear{Plesa et~al.}{2016}]{Plesa:2016:CRS}
\begin{barticle}
\bauthor{\bsnm{Plesa}, \binits{T.}},
\bauthor{\bsnm{Vejchodsk\'y}, \binits{T.}},
\bauthor{\bsnm{Erban}, \binits{R.}}:
\batitle{{C}hemical reaction systems with a homoclinic bifurcation: an inverse
  problem}.
\bjtitle{Journal of Mathematical Chemistry}
\bvolume{54}(\bissue{10}),
\bfpage{1884}--\blpage{1915}
(\byear{2016})
\end{barticle}
\endbibitem

\bibitem[\protect\citeauthoryear{Erban et~al.}{2009}]{Erban:2009:ASC}
\begin{barticle}
\bauthor{\bsnm{Erban}, \binits{R.}},
\bauthor{\bsnm{Chapman}, \binits{S.J.}},
\bauthor{\bsnm{Kevrekidis}, \binits{I.}},
\bauthor{\bsnm{Vejchodsky}, \binits{T.}}:
\batitle{Analysis of a stochastic chemical system close to a {SNIPER}
  bifurcation of its mean-field model}.
\bjtitle{SIAM Journal on Applied Mathematics}
\bvolume{70}(\bissue{3}),
\bfpage{984}--\blpage{1016}
(\byear{2009})
\end{barticle}
\endbibitem

\bibitem[\protect\citeauthoryear{Craciun et~al.}{2020}]{Craciun:2020:ECC}
\begin{barticle}
\bauthor{\bsnm{Craciun}, \binits{G.}},
\bauthor{\bsnm{Jin}, \binits{J.}},
\bauthor{\bsnm{Yu}, \binits{P.}}:
\batitle{An efficient characterization of complex-balanced, detailed-balanced,
  and weakly reversible systems}.
\bjtitle{SIAM Journal on Applied Mathematics}
\bvolume{80}(\bissue{1}),
\bfpage{183}--\blpage{205}
(\byear{2020})
\end{barticle}
\endbibitem

\bibitem[\protect\citeauthoryear{Erban and Kang}{2023}]{Erban:2023:CSL}
\begin{barticle}
\bauthor{\bsnm{Erban}, \binits{R.}},
\bauthor{\bsnm{Kang}, \binits{H.}}:
\batitle{Chemical systems with limit cycles}.
\bjtitle{Bulletin of Mathematical Biology}
\bvolume{85},
\bfpage{76}
(\byear{2023})
\end{barticle}
\endbibitem

\bibitem[\protect\citeauthoryear{Chavarriga et~al.}{2004}]{Chavarriga:2004:ALC}
\begin{barticle}
\bauthor{\bsnm{Chavarriga}, \binits{J.}},
\bauthor{\bsnm{Llibre}, \binits{J.}},
\bauthor{\bsnm{Sorolla}, \binits{J.}}:
\batitle{Algebraic limit cycles of degree 4 for quadratic systems}.
\bjtitle{Journal of Differential Equations}
\bvolume{200},
\bfpage{206}--\blpage{244}
(\byear{2004})
\end{barticle}
\endbibitem

\bibitem[\protect\citeauthoryear{Escher}{1979}]{escher1979models}
\begin{barticle}
\bauthor{\bsnm{Escher}, \binits{C.}}:
\batitle{Models of chemical reaction systems with exactly evaluable limit cycle
  oscillations}.
\bjtitle{Zeitschrift f{\"u}r Physik B Condensed Matter}
\bvolume{35}(\bissue{4}),
\bfpage{351}--\blpage{361}
(\byear{1979})
\end{barticle}
\endbibitem

\bibitem[\protect\citeauthoryear{Escher}{1980a}]{Escher:1980:GSA}
\begin{barticle}
\bauthor{\bsnm{Escher}, \binits{C.}}:
\batitle{Global stability analysis of open two-variable quadratic mass-action
  systems with elliptical limit cycles}.
\bjtitle{Zeitschrift f{\"u}r Physik B Condensed Matter}
\bvolume{40}(\bissue{1}),
\bfpage{137}--\blpage{141}
(\byear{1980})
\end{barticle}
\endbibitem

\bibitem[\protect\citeauthoryear{Escher}{1980b}]{Escher:1980:MCR}
\begin{barticle}
\bauthor{\bsnm{Escher}, \binits{C.}}:
\batitle{Models of chemical reaction systems with exactly evaluable limit cycle
  oscillations and their bifurcation behaviour}.
\bjtitle{Berichte der Bunsengesellschaft f{\"u}r physikalische Chemie}
\bvolume{84}(\bissue{4}),
\bfpage{387}--\blpage{391}
(\byear{1980})
\end{barticle}
\endbibitem

\bibitem[\protect\citeauthoryear{Erban and Chapman}{2020}]{Erban:2020:SMR}
\begin{bbook}
\bauthor{\bsnm{Erban}, \binits{R.}},
\bauthor{\bsnm{Chapman}, \binits{S.J.}}:
\bbtitle{Stochastic Modelling of Reaction--diffusion Processes}
vol. \bseriesno{60}.
\bpublisher{Cambridge University Press},
\blocation{Cambridge}
(\byear{2020})
\end{bbook}
\endbibitem

\bibitem[\protect\citeauthoryear{H{\'a}rs and T{\'o}th}{1981}]{Hars:1981:IPR}
\begin{barticle}
\bauthor{\bsnm{H{\'a}rs}, \binits{V.}},
\bauthor{\bsnm{T{\'o}th}, \binits{J.}}:
\batitle{On the inverse problem of reaction kinetics}.
\bjtitle{Qualitative Theory of Differential Equations}
\bvolume{30},
\bfpage{363}--\blpage{379}
(\byear{1981})
\end{barticle}
\endbibitem

\bibitem[\protect\citeauthoryear{Boros et~al.}{2020}]{Boros:2020:WRM}
\begin{barticle}
\bauthor{\bsnm{Boros}, \binits{B.}},
\bauthor{\bsnm{Craciun}, \binits{G.}},
\bauthor{\bsnm{Yu}, \binits{P.}}:
\batitle{Weakly reversible mass-action systems with infinitely many positive
  steady states}.
\bjtitle{SIAM Journal on Applied Mathematics}
\bvolume{80}(\bissue{4}),
\bfpage{1936}--\blpage{1946}
(\byear{2020})
\end{barticle}
\endbibitem

\bibitem[\protect\citeauthoryear{Craciun}{2019}]{craciun2019polynomial}
\begin{barticle}
\bauthor{\bsnm{Craciun}, \binits{G.}}:
\batitle{Polynomial dynamical systems, reaction networks, and toric
  differential inclusions}.
\bjtitle{SIAM Journal on Applied Algebra and Geometry}
\bvolume{3}(\bissue{1}),
\bfpage{87}--\blpage{106}
(\byear{2019})
\end{barticle}
\endbibitem

\bibitem[\protect\citeauthoryear{Yu and Craciun}{2018}]{yu2018mathematical}
\begin{barticle}
\bauthor{\bsnm{Yu}, \binits{P.Y.}},
\bauthor{\bsnm{Craciun}, \binits{G.}}:
\batitle{Mathematical analysis of chemical reaction systems}.
\bjtitle{Israel Journal of Chemistry}
\bvolume{58}(\bissue{6-7}),
\bfpage{733}--\blpage{741}
(\byear{2018})
\end{barticle}
\endbibitem

\bibitem[\protect\citeauthoryear{Prohens and
  Torregrosa}{2018}]{Prohens:2018:NLB}
\begin{barticle}
\bauthor{\bsnm{Prohens}, \binits{R.}},
\bauthor{\bsnm{Torregrosa}, \binits{J.}}:
\batitle{New lower bounds for the {H}ilbert numbers using reversible centers}.
\bjtitle{Nonlinearity}
\bvolume{32}(\bissue{1}),
\bfpage{331}
(\byear{2018})
\end{barticle}
\endbibitem

\bibitem[\protect\citeauthoryear{Plesa}{2024}]{Plesa:2024:MDS}
\begin{botherref}
\oauthor{\bsnm{Plesa}, \binits{T.}}:
Mapping dynamical systems into chemical reactions
\bjtitle{https://arxiv.org/abs/2406.03473}
(\byear{2024})
\end{botherref}
\endbibitem

\bibitem[\protect\citeauthoryear{Gasull and Santana}{2024}]{Gasull:2024:NHP}
\begin{barticle}
\bauthor{\bsnm{Gasull}, \binits{A.}},
\bauthor{\bsnm{Santana}, \binits{P.}}:
\batitle{A note on Hilbert 16th problem}.
\bjtitle{Proceedings of the American Mathematical Society}
\bvolume{153},
\bfpage{669}--\blpage{677}
(\byear{2025})
\end{barticle}
\endbibitem

\bibitem[\protect\citeauthoryear{Tyson and Light}{1973}]{tyson1973properties}
\begin{barticle}
\bauthor{\bsnm{Tyson}, \binits{J.J.}},
\bauthor{\bsnm{Light}, \binits{J.C.}}:
\batitle{Properties of two-component bimolecular and trimolecular chemical
  reaction systems}.
\bjtitle{Journal of Chemical Physics}
\bvolume{59}(\bissue{8}),
\bfpage{4164}--\blpage{4173}
(\byear{1973})
\end{barticle}
\endbibitem

\bibitem[\protect\citeauthoryear{Frank-Kamenetsky and
  Salnikov}{1943}]{frank-kamenetsky:1943:possibility}
\begin{barticle}
\bauthor{\bsnm{Frank-Kamenetsky}, \binits{D.}},
\bauthor{\bsnm{Salnikov}, \binits{I.}}:
\batitle{On the possibility of auto-oscillation in homogeneous chemical systems
  with quadratic autocatalysis}.
\bjtitle{Zhurnal Fizicheskoi Khimii (Journal of Physical Chemistry)}
\bvolume{17}(\bissue{1}),
\bfpage{79}--\blpage{86}
(\byear{1943})
\end{barticle}
\endbibitem

\bibitem[\protect\citeauthoryear{Escher}{1981}]{escher1981bifurcation}
\begin{barticle}
\bauthor{\bsnm{Escher}, \binits{C.}}:
\batitle{Bifurcation and coexistence of several limit cycles in models of open
  two-variable quadratic mass-action systems}.
\bjtitle{Chemical Physics}
\bvolume{63}(\bissue{3}),
\bfpage{337}--\blpage{348}
(\byear{1981})
\end{barticle}
\endbibitem

\bibitem[\protect\citeauthoryear{Lloyd et~al.}{2002}]{Lloyd:2002:CKS}
\begin{barticle}
\bauthor{\bsnm{Lloyd}, \binits{N.}},
\bauthor{\bsnm{Pearson}, \binits{J.}},
\bauthor{\bsnm{S{\'a}ez}, \binits{E.}},
\bauthor{\bsnm{Sz{\'a}nt{\'o}}, \binits{I.}}:
\batitle{A cubic {K}olmogorov system with six limit cycles}.
\bjtitle{Computers \& Mathematics with Applications}
\bvolume{44}(\bissue{3-4}),
\bfpage{445}--\blpage{455}
(\byear{2002})
\end{barticle}
\endbibitem

\bibitem[\protect\citeauthoryear{Carvalho et~al.}{2023}]{Carvalho:2023:NLB}
\begin{barticle}
\bauthor{\bsnm{Carvalho}, \binits{Y.}},
\bauthor{\bsnm{Da~Cruz}, \binits{L.}},
\bauthor{\bsnm{Gouveia}, \binits{L.}}:
\batitle{New lower bound for the {H}ilbert number in low degree {K}olmogorov
  systems}.
\bjtitle{Chaos, Solitons \& Fractals}
\bvolume{175},
\bfpage{113937}
(\byear{2023})
\end{barticle}
\endbibitem

\bibitem[\protect\citeauthoryear{Smale and Hirsch}{1974}]{Smale:1974:DED}
\begin{bbook}
\bauthor{\bsnm{Smale}, \binits{S.}},
\bauthor{\bsnm{Hirsch}, \binits{M.}}:
\bbtitle{Differential Equations, Dynamical Systems, and Linear Algebra}
vol. \bseriesno{60}.
\bpublisher{Academic Press},
\blocation{New York}
(\byear{1974})
\end{bbook}
\endbibitem

\bibitem[\protect\citeauthoryear{Perko}{2013}]{Perko:2013:DED}
\begin{bbook}
\bauthor{\bsnm{Perko}, \binits{L.}}:
\bbtitle{Differential Equations and Dynamical Systems}
vol. \bseriesno{7}.
\bpublisher{Springer},
\blocation{New York}
(\byear{2013})
\end{bbook}
\endbibitem

\bibitem[\protect\citeauthoryear{Plesa et~al.}{2017}]{Plesa:2017:TMS}
\begin{bchapter}
\bauthor{\bsnm{Plesa}, \binits{T.}},
\bauthor{\bsnm{Vejchodsk{\'y}}, \binits{T.}},
\bauthor{\bsnm{Erban}, \binits{R.}}:
\bctitle{Test models for statistical inference: Two-dimensional reaction
  systems displaying limit cycle bifurcations and bistability}.
In: \bbtitle{Stochastic Processes, Multiscale Modeling, and Numerical Methods
  for Computational Cellular Biology},
pp. \bfpage{3}--\blpage{27}.
\bpublisher{Springer}, \blocation{???}
(\byear{2017})
\end{bchapter}
\endbibitem

\bibitem[\protect\citeauthoryear{Nagy et~al.}{2020}]{nagy2020two}
\begin{barticle}
\bauthor{\bsnm{Nagy}, \binits{I.}},
\bauthor{\bsnm{Romanovski}, \binits{V.G.}},
\bauthor{\bsnm{T{\'o}th}, \binits{J.}}:
\batitle{Two nested limit cycles in two-species reactions}.
\bjtitle{Mathematics}
\bvolume{8}(\bissue{10}),
\bfpage{1658}
(\byear{2020})
\end{barticle}
\endbibitem

\bibitem[\protect\citeauthoryear{Llibre et~al.}{2010}]{Llibre:2010:OHP}
\begin{barticle}
\bauthor{\bsnm{Llibre}, \binits{J.}},
\bauthor{\bsnm{Ram{\'\i}rez}, \binits{R.}},
\bauthor{\bsnm{Sadovskaia}, \binits{N.}}:
\batitle{On the 16th {H}ilbert problem for algebraic limit cycles}.
\bjtitle{Journal of Differential Equations}
\bvolume{248}(\bissue{6}),
\bfpage{1401}--\blpage{1409}
(\byear{2010})
\end{barticle}
\endbibitem

\bibitem[\protect\citeauthoryear{Christopher}{2001}]{christopher2001polynomial}
\begin{barticle}
\bauthor{\bsnm{Christopher}, \binits{C.}}:
\batitle{Polynomial vector fields with prescribed algebraic limit cycles}.
\bjtitle{Geometriae Dedicata}
\bvolume{88},
\bfpage{255}--\blpage{258}
(\byear{2001})
\end{barticle}
\endbibitem

\bibitem[\protect\citeauthoryear{Craciun and
  Pantea}{2008}]{Craciun_Pantea_2008}
\begin{barticle}
\bauthor{\bsnm{Craciun}, \binits{G.}},
\bauthor{\bsnm{Pantea}, \binits{C.}}:
\batitle{Identifiability of chemical reaction networks}.
\bjtitle{Journal of Mathematical Chemistry}
\bvolume{44}(\bissue{1}),
\bfpage{244}--\blpage{259}
(\byear{2008})
\end{barticle}
\endbibitem

\bibitem[\protect\citeauthoryear{Plesa et~al.}{2018}]{Plesa:2018:NCM}
\begin{barticle}
\bauthor{\bsnm{Plesa}, \binits{T.}},
\bauthor{\bsnm{Zygalakis}, \binits{K.}},
\bauthor{\bsnm{Anderson}, \binits{D.}},
\bauthor{\bsnm{Erban}, \binits{R.}}:
\batitle{Noise control for molecular computing}.
\bjtitle{Journal of the Royal Society Interface}
\bvolume{15}(\bissue{144}),
\bfpage{20180199}
(\byear{2018})
\end{barticle}
\endbibitem

\bibitem[\protect\citeauthoryear{Gin{\'e} et~al.}{2018}]{gine2018cubic}
\begin{barticle}
\bauthor{\bsnm{Gin{\'e}}, \binits{J.}},
\bauthor{\bsnm{Llibre}, \binits{J.}},
\bauthor{\bsnm{Valls}, \binits{C.}}:
\batitle{The cubic polynomial differential systems with two circles as
  algebraic limit cycles}.
\bjtitle{Advanced Nonlinear Studies}
\bvolume{18}(\bissue{1}),
\bfpage{183}--\blpage{193}
(\byear{2018})
\end{barticle}
\endbibitem

\bibitem[\protect\citeauthoryear{Enciso et~al.}{2021}]{Enciso:2021:ISM}
\begin{barticle}
\bauthor{\bsnm{Enciso}, \binits{G.}},
\bauthor{\bsnm{Erban}, \binits{R.}},
\bauthor{\bsnm{Kim}, \binits{J.}}:
\batitle{Identifiability of stochastically modelled reaction networks}.
\bjtitle{European Journal of Applied Mathematics}
\bvolume{32}(\bissue{5}),
\bfpage{865}--\blpage{887}
(\byear{2021})
\end{barticle}
\endbibitem

\bibitem[\protect\citeauthoryear{Gillespie}{1977}]{Gillespie:1977:ESS}
\begin{barticle}
\bauthor{\bsnm{Gillespie}, \binits{D.}}:
\batitle{Exact stochastic simulation of coupled chemical reactions}.
\bjtitle{Journal of Physical Chemistry}
\bvolume{81}(\bissue{25}),
\bfpage{2340}--\blpage{2361}
(\byear{1977})
\end{barticle}
\endbibitem

\bibitem[\protect\citeauthoryear{Duncan et~al.}{2015}]{Duncan:2015:NIM}
\begin{barticle}
\bauthor{\bsnm{Duncan}, \binits{A.}},
\bauthor{\bsnm{Liao}, \binits{S.}},
\bauthor{\bsnm{Vejchodsk{\`y}}, \binits{T.}},
\bauthor{\bsnm{Erban}, \binits{R.}},
\bauthor{\bsnm{Grima}, \binits{R.}}:
\batitle{Noise-induced multistability in chemical systems: Discrete versus
  continuum modeling}.
\bjtitle{Physical Review E}
\bvolume{91}(\bissue{4}),
\bfpage{042111}
(\byear{2015})
\end{barticle}
\endbibitem

\bibitem[\protect\citeauthoryear{Plesa et~al.}{2019}]{Plesa:2019:NIM}
\begin{barticle}
\bauthor{\bsnm{Plesa}, \binits{T.}},
\bauthor{\bsnm{Erban}, \binits{R.}},
\bauthor{\bsnm{Othmer}, \binits{H.}}:
\batitle{Noise-induced mixing and multimodality in reaction networks}.
\bjtitle{European Journal of Applied Mathematics}
\bvolume{30}(\bissue{5}),
\bfpage{887}--\blpage{911}
(\byear{2019})
\end{barticle}
\endbibitem

\bibitem[\protect\citeauthoryear{Muratov et~al.}{2005}]{Muratov:2005:SSR}
\begin{barticle}
\bauthor{\bsnm{Muratov}, \binits{C.}},
\bauthor{\bsnm{Vanden-Eijnden}, \binits{E.}},
\bauthor{\bsnm{E}, \binits{W.}}:
\batitle{Self-induced stochastic resonance in excitable systems}.
\bjtitle{Physica D}
\bvolume{210},
\bfpage{227}--\blpage{240}
(\byear{2005})
\end{barticle}
\endbibitem

\bibitem[\protect\citeauthoryear{Reddy}{1975}]{Reddy:1975:ESS}
\begin{barticle}
\bauthor{\bsnm{Reddy}, \binits{V.}}:
\batitle{On the existence of the steady state in the stochastic volterra-lotka
  model}.
\bjtitle{Journal of Statistical Physics}
\bvolume{13},
\bfpage{61}--\blpage{64}
(\byear{1975})
\end{barticle}
\endbibitem

\bibitem[\protect\citeauthoryear{Banaji et~al.}{2024}]{Banaji:2024:OTR}
\begin{barticle}
\bauthor{\bsnm{Banaji}, \binits{M.}},
\bauthor{\bsnm{Boros}, \binits{B.}},
\bauthor{\bsnm{Hofbauer}, \binits{J.}}:
\batitle{Oscillations in three-reaction quadratic mass-action systems}.
\bjtitle{Studies in Applied Mathematics}
\bvolume{152}(\bissue{1}),
\bfpage{249}--\blpage{278}
(\byear{2024})
\end{barticle}
\endbibitem

\bibitem[\protect\citeauthoryear{Boros and Hofbauer}{2022}]{Boros:2022:LCM}
\begin{barticle}
\bauthor{\bsnm{Boros}, \binits{B.}},
\bauthor{\bsnm{Hofbauer}, \binits{J.}}:
\batitle{Limit cycles in mass-conserving deficiency-one mass-action systems}.
\bjtitle{Electronic Journal of Qualitative Theory of Differential Equations}
\bvolume{2022}(\bissue{42}),
\bfpage{1}--\blpage{18}
(\byear{2022})
\end{barticle}
\endbibitem

\bibitem[\protect\citeauthoryear{Boros and Hofbauer}{2023}]{Boros:2023:SMB}
\begin{barticle}
\bauthor{\bsnm{Boros}, \binits{B.}},
\bauthor{\bsnm{Hofbauer}, \binits{J.}}:
\batitle{Some minimal bimolecular mass-action systems with limit cycles}.
\bjtitle{Nonlinear Analysis: Real World Applications}
\bvolume{72},
\bfpage{103839}
(\byear{2023})
\end{barticle}
\endbibitem

\bibitem[\protect\citeauthoryear{Boros and Hofbauer}{2024}]{Boros:2024:OPD}
\begin{barticle}
\bauthor{\bsnm{Boros}, \binits{B.}},
\bauthor{\bsnm{Hofbauer}, \binits{J.}}:
\batitle{Oscillations in planar deficiency-one mass-action systems}.
\bjtitle{Journal of Dynamics and Differential Equations}
\bvolume{36}(\bissue{Suppl 1}),
\bfpage{175}--\blpage{197}
(\byear{2024})
\end{barticle}
\endbibitem

\bibitem[\protect\citeauthoryear{Feinberg}{1972}]{Feinberg72}
\begin{barticle}
\bauthor{\bsnm{Feinberg}, \binits{M.}}:
\batitle{{Complex balancing in general kinetic systems}}.
\bjtitle{Archive for Rational Mechanics and Analysis}
\bvolume{49},
\bfpage{187}--\blpage{194}
(\byear{1972})
\end{barticle}
\endbibitem

\end{thebibliography}

\end{document}